\documentclass[a4paper, reqno]{amsart}
\usepackage{amsmath, amsthm, amssymb, tikz, tikz-cd, enumerate, stmaryrd, multirow, array}
\usepackage[hidelinks]{hyperref}
\usepackage[top=30mm, bottom=30mm, left=25mm, right=25mm]{geometry}
\usepackage[shortalphabetic]{amsrefs}

\tikzset{dnode/.style={circle,fill=white,draw=black,inner sep=0ex,minimum size=1ex}}
\tikzset{dfnode/.style={circle,fill=black,draw=black,inner sep=0ex,minimum size=1ex}}
\tikzset{dlabel/.style={%
text height=1ex,%
execute at begin node=$\scriptstyle,%
execute at end node=$}}

\makeatletter
\let\orgdescriptionlabel\descriptionlabel
\renewcommand*{\descriptionlabel}[1]{%
\let\orglabel\label
\let\label\@gobble
\phantomsection
\edef\@currentlabel{#1}
\let\label\orglabel
\orgdescriptionlabel{#1}}
\makeatother

\renewenvironment{description}{\list{}{\labelwidth=0pt \leftmargin=0pt }}{\endlist}

\tikzcdset{row sep/normal=2.7em}
\tikzcdset{column sep/normal=3.0em}

\let\hat\widehat

\newcommand{\mb}[1]{\mathbb{#1}}
\newcommand{\mc}[1]{\mathcal{#1}}
\newcommand{\mf}[1]{\mathfrak{#1}}
\newcommand{\mrm}[1]{\mathrm{#1}}
\newcommand{\tn}[1]{\textnormal{#1}}
\newcommand{\bun}{\mrm{Bun}}
\newcommand{\tbun}{\widetilde{\bun}}
\renewcommand{\sslash}{{/\mkern-6mu/}}
\newcommand{\B}{\mb{B}}
\newcommand{\Deg}{{\mf{D}\mrm{eg}}}

\newcommand{\mmid}{{\,|\,}}

\newcommand{\longhookrightarrow}{\ensuremath{\lhook\joinrel\relbar\joinrel\rightarrow}}
\newcommand{\longhookleftarrow}{\ensuremath{\leftarrow\joinrel\relbar\joinrel\rhook}}

\renewcommand{\hom}{\mrm{Hom}}

\DeclareMathOperator{\spec}{\mrm{Spec}}
\DeclareMathOperator{\spf}{\mrm{Spf}}

\DeclareMathOperator{\coker}{\mrm{coker}}

\DeclareMathOperator{\sym}{\mrm{Sym}}

\theoremstyle{plain}
\newtheorem{thm}{Theorem}[subsection]
\newtheorem{prop}[thm]{Proposition}
\newtheorem{lem}[thm]{Lemma}
\newtheorem{cor}[thm]{Corollary}
\theoremstyle{definition}
\newtheorem{defn}[thm]{Definition}
\newtheorem{rmk}[thm]{Remark}

\newtheorem{notation}[thm]{Notation}
\newtheorem{construction}[thm]{Construction}

\numberwithin{equation}{subsection}

\title{On subregular slices of the elliptic Grothendieck-Springer resolution}
\author{Dougal Davis}
\address{School of Mathematics, University of Edinburgh, James Clerk Maxwell Building, The King's Buildings, Peter Guthrie Tait Road, Edinburgh, EH9 3FD, UK}
\email{dougal.davis@ed.ac.uk}
\date{\today}

\begin{document}

\begin{abstract}
In \cite{davis19}, the author constructed an elliptic version of the Grothendieck-Springer resolution for the stack $\mathrm{Bun}_G$ of principal bundles under a simply connected simple group $G$ on an elliptic curve $E$. This is a simultaneous log resolution of a map from $\mathrm{Bun}_G$ to the union of the coarse moduli space of semistable $G$-bundles and a single stacky point. In this paper, we study singularities, resolutions and deformations coming from subregular slices of this elliptic Grothendieck-Springer resolution. More precisely, we construct explicit slices of $\mathrm{Bun}_G$ through all subregular unstable bundles, for every $G$. For $G \neq SL_2$, we describe the pullbacks of the elliptic Grothendieck-Springer resolution to these slices as concrete varieties, extending and refining earlier work of I. Grojnowski and N. Shepherd-Barron, who related these varieties to del Pezzo surfaces in type $E$. We use the resolutions to identify the singularities of the unstable locus of the subregular slices, and prove that that the extended coarse moduli space map gives deformations that are miniversal among torus-equivariant deformations with appropriate weights.
\end{abstract}

\maketitle

\tableofcontents

\section{Introduction}
Since the work of E. Brieskorn \cite{brieskorn70} and P. Slodowy \cite{slodowy80a}, it has been well known that du Val (aka simple, Kleinian, $ADE$, etc.) singularities of algebraic surfaces arise naturally in the geometry of simple algebraic groups and their Lie algebras. If $G$ is a simply connected simple algebraic group, say over an algebraically closed field $k$ of characteristic $0$, then the cone $\mc{N}$ of nilpotent elements inside the Lie algebra $\mf{g}$ of $G$ is a singular variety canonically associated to $\mf{g}$. The cone $\mc{N}$ has dimension $\dim G - l$, where $l$ is the rank of $G$; to obtain a surface singularity, one first chooses a subregular nilpotent element $x \in \mf{g}$ (i.e., one satisfying $\dim \mrm{Stab}_G(x) = l + 2$) and a transversal slice $Z$ (a locally closed subvariety $Z \subseteq \mf{g}$ transverse to all $G$-orbits for the adjoint representation) such that $x$ is the unique subregular nilpotent in $Z$. Then $Z \cap \mc{N}$ is a surface, with a unique du Val singularity at $x$ whose Dynkin diagram is the same as that of $G$ when $G$ is of type $A$, $D$ or $E$.

The singular surfaces constructed in this way are also furnished with natural Lie-theoretic deformations and resolutions. The deformations arise from the \emph{(additive) adjoint quotient map}
\begin{equation} \label{eq:introadditiveadjointquotient}
\chi^{add} \colon \mf{g} \longrightarrow \mf{g}\sslash G = \spec k[\mf{g}]^G,
\end{equation}
where $G$ acts on $\mf{g}$ via the adjoint representation. The morphism $\chi^{add}$ is a flat family of affine varieties with central fibre $\mc{N} = (\chi^{add})^{-1}(0)$; the restriction $\chi_Z = \chi|_Z \colon Z \to \mf{g}\sslash G$ gives a flat deformation of the singular surface $Z \cap \mc{N} = \chi_Z^{-1}(0)$. In types $ADE$, it was proved by Brieskorn that this recovers the miniversal deformation of the singular surface $Z \cap \mc{N}$, while in types $BCFG$ it was shown by Slodowy that the deformation is miniversal among those preserving a ``folding symmetry'' of the $ADE$ du Val singularity.

The resolutions arise from a commutative diagram
\begin{equation} \label{eq:introadditivegs}
\begin{tikzcd}
\tilde{\mf{g}} \ar[r, "\psi^{add}"] \ar[d, "\tilde{\chi}^{add}"'] & \mf{g} \ar[d, "\chi^{add}"] \\
\mf{t} \ar[r, "q"] & \mf{t}\sslash W \cong \mf{g}\sslash G,
\end{tikzcd}
\end{equation}
where $\tilde{\mf{g}} = G \times^B \mf{b}$ for $B \subseteq G$ a Borel subgroup with Lie algebra $\mf{b}$, $\mf{t}$ is the Lie algebra of a maximal torus $T \subseteq G$ and $W = N_G(T)/T$ is the Weyl group. The diagram \eqref{eq:introadditivegs} is called the \emph{(additive) Grothendieck-Springer resolution}; it is a simultaneous resolution of singularities in the sense that $\tilde{\chi}^{add}$ is smooth, $\psi^{add}$ is proper, and $(\tilde{\chi}^{add})^{-1}(t) \to (\chi^{add})^{-1}(q(t))$ is a resolution of singularities for all $t \in \mf{t}$. Setting $\tilde{Z} = \tilde{\mf{g}} \times_\mf{g} Z$, the induced diagram
\[
\begin{tikzcd}
\tilde{Z} \ar[r, "\psi_Z"] \ar[d, "\tilde{\chi}_Z"'] & Z \ar[d, "\chi_Z"] \\
\mf{t} \ar[r] & \mf{t} \sslash W
\end{tikzcd}
\]
is a simultaneous resolution for $\chi_Z$. One pleasing way to identify the du Val singularity of $\chi_Z^{-1}(0)$ is to compute the fibre $\psi_Z^{-1}(x)$, and to show that this gives the correct Dynkin configuration of $(-2)$-curves on the resolution $\tilde{\chi}_Z^{-1}(0)$.

This paper is concerned with an elliptic version of the above additive story. (From now on, $k$ can be any algebraically closed field.) Building on earlier work \cite{helmke-slodowy04}\cite{ben-zvi-nadler15}\cite{grojnowski-shep19}, the author constructed in \cite{davis19} a commutative diagram
\begin{equation} \label{eq:introellipticgs}
\begin{tikzcd}
\tbun_G \ar[r, "\psi"] \ar[d, "\tilde{\chi}"'] & \bun_G \ar[d, "\chi"] \\
\Theta_Y^{-1}/\mb{G}_m \ar[r, "q"] & (\hat{Y}\sslash W)/\mb{G}_m,
\end{tikzcd}
\end{equation}
where $\bun_G$ is the stack of principal $G$-bundles on an elliptic curve $E$, $\tbun_G$ is the Kontsevich-Mori compactification of the stack $\bun_B^0$ of degree $0$ $B$-bundles on $E$, $\Theta_Y^{-1}$ is an anti-ample $W$-linearised line bundle on the coarse moduli space $Y = \hom(\mb{X}^*(T), \mrm{Pic}^0(E))$ of degree $0$ $T$-bundles on $E$, $\hat{Y}$ is the affine cone over $Y$ obtained by contracting the zero section of $\Theta_Y^{-1}$ to a point, and $/\mb{G}_m$ denotes the stack quotient by $\mb{G}_m$. Away from the image of the cone point of $\hat{Y}$, $\chi$ agrees with the semistable coarse moduli space map $\chi^{ss} \colon \bun_G^{ss} \to Y \sslash W$ of R. Friedman and J. Morgan \cite{friedman-morgan98}, and the preimage of the (stacky) cone point is precisely the locus of unstable bundles.

The diagram \eqref{eq:introellipticgs} is called the \emph{elliptic Grothendieck-Springer resolution}, and is closely analogous to a stacky version of \eqref{eq:introadditivegs} where the varieties $\mf{g}$ and $\tilde{\mf{g}}$ are replaced by the stack quotients $\mf{g}/G$ and $\tilde{\mf{g}}/G$. It was shown in \cite{davis19}*{Corollary 4.4.7} that it is a simultaneous \emph{log} resolution with respect to the zero section of $\Theta_Y^{-1}$ \cite{davis19}*{Definition 1.0.3}; this means that the total space $\tbun_G$ is smooth, $\tilde{\chi}$ is smooth away from the zero section, the preimage of the zero section is a divisor with normal crossings, the map $\psi$ is proper (with finite relative stabilisers) and for all $y \in \Theta_Y^{-1}/\mb{G}_m$, the map $\tilde{\chi}^{-1}(y) \to \chi^{-1}(q(y))$ is an isomorphism over a dense open subset of the target. In particular, the restriction to semistable bundles is a genuine simultaneous resolution, and for $y$ in the zero section of $\Theta_Y^{-1}$, each irreducible component of the locus $\chi^{-1}(q(y)) = \chi^{-1}(0)$ of unstable bundles is resolved by some component of $\tilde{\chi}^{-1}(y)$.

Subregular slices in the elliptic setting have been studied by S. Helmke and P. Slodowy \cite{helmke-slodowy01} \cite{helmke-slodowy04} and I. Grojnowski and N. Shepherd-Barron \cite{grojnowski-shep19}. In \cite{helmke-slodowy01}, Helmke and Slodowy classified the subregular unstable bundles (Definition \ref{defn:subregularunstable}) and gave simple descriptions of their coarse moduli spaces for all simply connected groups $G$; these bundles play the role of subregular nilpotent elements in elliptic Springer theory. In \cite{helmke-slodowy04}, they constructed a version of the coarse quotient map $\chi$ using loop groups, and briefly sketched the associated surface singularities arising from slices through subregular unstable bundles in types $A$, $D$ and $E$. In \cite{grojnowski-shep19}, Grojnowski and Shepherd-Barron considered certain subregular slices $Z \to \bun_G$ for $G$ of types $D_5 = E_5$, $E_6$, $E_7$ and $E_8$ only, and studied simultaneous log resolutions
\begin{equation} \label{eq:introslicedellipticgs}
\begin{tikzcd}
\tilde{Z} \ar[r, "\psi_Z"] \ar[d, "\tilde{\chi}_Z"'] & Z \ar[d, "\chi_Z"] \\
\Theta_Y^{-1} \ar[r] & \hat{Y} \sslash W
\end{tikzcd}
\end{equation}
deduced from \eqref{eq:introellipticgs}, where $\tilde{Z} = \tbun_G \times_{\bun_G} Z$. They showed that, in their examples, the preimage $\tilde{\chi}_Z^{-1}(0_{\Theta_Y^{-1}})$ of the zero section decomposes as a simple normal crossings divisor
\[ \tilde{\chi}_Z^{-1}(0_{\Theta_Y^{-1}}) = D_0 + D_1 + Q,\]
where $D_0 \to Y$ is a family of resolutions of the singular surface $\chi_Z^{-1}(0)$, $D_1 \to Y$ is some other family of projective surfaces, and $Q \to Y$ is a $\mb{P}^1 \times \mb{P}^1$-bundle. Moreover, they showed that contracting $Q$ along a ruling and flopping an unknown number of curves from $D_0$ to $D_1$ produces a birational modification $\tilde{Z} \dashrightarrow \tilde{Z}^-$ such that the preimage of $0_{\Theta_Y^{-1}}$ decomposes as $D_0^- + D_1^-$, where $D_0^-$ is a line bundle over $Y \times E$ and $D_1^- \to Y$ is a family of del Pezzo surfaces of degree $9 - l$, from which they deduced that $\chi_Z^{-1}(0)$ has a simply elliptic singularity of the same degree. Their results show that the elliptic Grothendieck-Springer resolution in some sense ``contains'' the well-known combinatorial correspondence between exceptional groups, del Pezzo surfaces, and simply elliptic singularities.

\begin{rmk}
One of the nice features of Grojnowski and Shepherd-Barron's construction is that the stack quotients by $\mb{G}_m$ in the bottom row of \eqref{eq:introellipticgs} are exchanged for a global action of $\mb{G}_m$ on the sliced diagram \eqref{eq:introslicedellipticgs}. This desirable behaviour is axiomatised by the notion of \emph{equivariant slices} in \cite{davis19}*{Definition 4.1.9}; these are stacks $Z$ equipped with an action of a torus $H$ (the \emph{equivariance group}), a morphism $Z/H \to \bun_G$ (or to the rigidification $\bun_{G, rig}$ \cite{davis19}*{\S 2.2}), and a lift of $Z \to (\hat{Y}\sslash W)/\mb{G}_m$ to an $H$-equivariant morphism $Z \to \hat{Y}\sslash W$, where $H$ acts on $\hat{Y}\sslash W$ through some fixed weight $H \to \mb{G}_m$, such that the morphism $Z \to \bun_G$ (or $\bun_{G, rig}$) is smooth modulo translations.
\end{rmk}

The goal of the present work is to describe all the singularities and log resolutions obtained from the elliptic Grothendieck-Springer resolution by taking equivariant slices through subregular unstable bundles, for all simply connected groups $G$.

Our first main result gives the existence of an equivariant slice with particularly nice properties through any subregular bundle (when $G \neq SL_2$). In order to to ensure the existence of slices $Z$ with generically trivial inertia, we have chosen to work with the rigidified stack $\bun_{G, rig}$ (cf.\ \cite{davis19}*{\S 2.2}) obtained by taking the quotient of all automorphism groups in $\bun_G$ by the centre $Z(G)$ of $G$.

\begin{thm} \label{thm:introsubregularsliceexistence}
Let $\xi_G \to E$ be a subregular unstable $G$-bundle, and assume that $G \neq SL_2$. Then there exists an equivariant slice $Z \to \bun_{G, rig}$ with equivariance group
\[ H = \begin{cases} \mb{G}_m \times \mb{G}_m, & \text{in type}\; A,\\ \mb{G}_m,& \text{otherwise,} \end{cases} \]
with the following properties.
\begin{enumerate}[(1)]
\item \label{itm:introsubregularsliceexistence1} The $H$-fixed locus $Z_0 = Z^H$ is a proper Artin stack with finite and generically trivial inertia.
\item \label{itm:introsubregularsliceexistence2} The set of points $z \in Z$ such that the associated $G$-bundle is subregular unstable is equal to $Z_0$, and the given family identifies the coarse moduli space of $Z_0$ with the connected component of the coarse moduli space of subregular unstable $G$-bundles up to translation containing $\xi_G$.
\item \label{itm:introsubregularsliceexistence3} All nonempty geometric fibres of the morphism $Z_0 \to \bun_{G, rig}/E$ are connected.
\end{enumerate}
\end{thm}

There are no essentially new ideas in the proof of Theorem \ref{thm:introsubregularsliceexistence}: following a suggestion of Helmke and Slodowy \cite{helmke-slodowy01}*{Remark 5.14}, the slices $Z \to \bun_{G, rig}$ are constructed by parabolic induction from regular slices $Z_0 \to \bun_{L, rig}$, for $L$ the Harder-Narasimhan Levi of $\xi_G$, which are either obvious or in turn constructed by parabolic induction from a single unstable $L$-bundle according to the recipe of Friedman and Morgan \cite{friedman-morgan00}. The only new thing we do in Theorem \ref{thm:introsubregularsliceexistence} is to check by hand that the morphisms $Z \to \bun_{G, rig}$ constructed in this way are actually equivariant slices with the desired properties.

For each of the equivariant slices $Z$ constructed in Theorem \ref{thm:introsubregularsliceexistence}, we get a morphism
\[ \tilde{\chi}_Z \colon \tilde{Z} = \tbun_{G, rig} \times_{\bun_{G, rig}} Z \longrightarrow \Theta_Y^{-1}.\]
Our second main result gives very explicit descriptions of the unstable fibres $\tilde{\chi}_Z^{-1}(y)$ for $y \in 0_{\Theta_Y^{-1}}$. This result is really the core content of the paper.

\begin{thm} \label{thm:introsubregularresolutions}
Assume that $G \neq SL_2$, let $\xi_G \to E$ be a subregular unstable $G$-bundle, and let $Z \to \bun_{G, rig}$ be the equivariant slice of Theorem \ref{thm:introsubregularsliceexistence}. Then we have the following.
\begin{enumerate}[(1)]
\item \label{itm:introsubregularresolutions1} The preimage of the zero section of $\Theta_Y^{-1}$ decomposes as a divisor with normal crossings
\[ \tilde{\chi}_Z^{-1}(0_{\Theta_Y^{-1}}) = dD_{\alpha_i^\vee}(Z) + D_{\alpha_j^\vee}(Z) + D_{\alpha_i^\vee + \alpha_j^\vee}(Z),\]
where each component $D_\lambda(Z)$ is smooth over $Y$, and
\[ d = \begin{cases} 1, & \text{if}\;\; \xi_G \;\text{is of type}\; A, B, D\;\text{or}\; E, \\ 2, & \text{if}\;\; \xi_G \;\text{is of type}\; C\; \text{or}\; F, \\ 3, & \text{if}\;\; \xi_G \; \text{is of type}\; G. \end{cases}\]
\item \label{itm:introsubregularresolutions2} The divisor $D_{\alpha_j^\vee}(Z)$ is isomorphic to the iterated blowup of a line bundle $D_1$ on $Y \times E$ at a specified sequence of sections (given in Proposition \ref{prop:subregularresolutions2}) over $Y$.
\item \label{itm:introsubregularresolutions3} Each fibre of the morphism $D_{\alpha_i^\vee + \alpha_j^\vee}(Z) \to Y$ is isomorphic to the Hirzebruch surface $\mb{F}_{d - 1}$.
\item \label{itm:introsubregularresolutions4} The divisor $D_{\alpha_i^\vee}(Z)$ is the iterated blowup of a smooth family of surfaces $D_1' \to Y$ at a specified sequence of sections (given in Proposition \ref{prop:subregularresolutions4}) over $Y$, where each fibre of $D_1' \to Y$ is isomorphic to
\begin{itemize}
\item a line bundle on $E$, if $\xi_G$ is of type $A$,
\item one of the Hirzebruch surfaces $\mb{F}_0$ or $\mb{F}_2$, if $\xi_G$ is of type $C$, $D$ or $F$,
\item one of the stacky Hirzebruch surfaces $\mb{P}_{\mb{P}(1, 2)}(\mc{O} \oplus \mc{O}(1))$ or $\mb{P}_{\mb{P}(1,2)}(\mc{O} \oplus \mc{O}(3))$, if $\xi_G$ is of type $B$, or
\item the projective plane $\mb{P}^2$, if $\xi_G$ is of type $E$ or $G$.
\end{itemize}
\end{enumerate}
\end{thm}

\begin{rmk}
In Theorem \ref{thm:introsubregularresolutions}, we have referred to the type of the subregular unstable $G$-bundle $\xi_G$, rather than to the type of the group $G$. This follows the terminology introduced in \S\ref{subsection:subregularclassification}. The idea is that a given algebraic group $G$ may belong to multiple series in the classification (the relevant examples here being $D_5 = E_5$ and $B_3 = F_3$); in these cases, there are connected components of the locus of subregular unstable bundles corresponding to each of the different series.
\end{rmk}

\begin{rmk}
In type $E_l$, Theorem \ref{thm:introsubregularresolutions} recovers Grojnowski and Shepherd-Barron's result discussed above, with $D_0 = D_{\alpha_j^\vee}(Z)$, $D_1 = D_{\alpha_i^\vee}(Z)$ and $Q = D_{\alpha_i^\vee + \alpha_j^\vee}(Z)$. Moreover, the slightly mysterious flopping curves are made manifest in our description as the exceptional fibres of the blowups of the line bundle $D_1$ (excluding the last one, which undoes the contraction of $Q$). In particular, the detailed statement Proposition \ref{prop:subregularresolutions4} specifies the exact number ($n_0 = l - 4$) and configuration of these curves, which was not accessible using the Grojnowski and Shepherd-Barron's proof. The del Pezzo surfaces also appear very concretely as blowups of $D_1' = \mb{P}^2$ at $l$ points; the first $4$ are the blowups in \eqref{itm:introsubregularresolutions4} giving $D_{\alpha_i^\vee}(Z)$, and the remaining $l - 4$ are the result of the flops.
\end{rmk}

As an application of Theorem \ref{thm:introsubregularresolutions}, we deduce the following descriptions of the singular surfaces $\chi_Z^{-1}(0)$ and their deformations. For completeness, we have also included the case $G = SL_2$ with the subregular slice $Z = \mrm{Ind}_{T}^G(Z_0)$ with equivariance group $\mb{G}_m$ of Remark \ref{rmk:sl2slice}, although this slice does \emph{not} satisfy the hypotheses of Theorem \ref{thm:introsubregularsliceexistence}.

\begin{thm}[Theorems \ref{thm:subregularsingularities} and \ref{thm:subregulardeformations}] \label{thm:introsubregularsingularities}
If the characteristic of $k$ is not $2$ or $3$, then the surface $\chi_Z^{-1}(0)$ can be constructed explicitly as follows.
\begin{enumerate}[(1)]
\item In type $A_l$, $l > 1$, $\chi_Z^{-1}(0)$ is obtained by gluing together two line bundles on $E$ along their zero sections.
\item In types $C$ and $D$ (resp., $B$), $\chi_Z^{-1}(0)$ is obtained by identifying points in the fibres of a degree $2$ map $E \to \mb{P}^1$ (resp., $E \to \mb{P}(1, 2)$) inside the zero section of a line bundle on $E$.
\item In types $A_1$, $E$, $F$ and $G$, $\chi_Z^{-1}(0)$ is a cone obtained by contracting the zero section of a line bundle on $E$ to a point.
\end{enumerate}
In each case, the deformation $\chi_Z \colon Z \to \hat{Y}\sslash W$ is miniversal among $H$-equivariant deformations with weights in $\mb{Z}_{>0}\lambda$, where $\lambda \in \mb{X}^*(H)$ is the weight of the equivariant slice $Z \to \bun_{G, rig}$.
\end{thm}

\begin{rmk}
The description of the singularities in types $A$, $D$ and $E$ was given without proof in \cite{helmke-slodowy04}. As far as we know, the description for types $B$, $C$, $F$ and $G$ is new.
\end{rmk}

\begin{rmk}
It follows from the explicit degrees and weights given in Theorems \ref{thm:subregularsingularities} and \ref{thm:subregulardeformations} and in Table \ref{tab:weights} that the deformations of types $A_1$, $C$, $F$ and $G$ are related to those of types $D$ and $E$ by a curious twist on the usual folding story for du Val singularities. For each pair $(A_1, E_5)$, $(C_l, D_{l + 4})$, $(F_l, E_{l + 3})$ and $(G_2, E_8)$, the surfaces $\chi_Z^{-1}(0)$ are isomorphic in both cases, and the deformation for the first case is naturally identified with the subspace preserving the action of $\mu_d \subseteq \mb{G}_m$ inside the deformation for the second, where $d = 2$ or $3$. Note that this links different pairs of groups to the usual folding, i.e., the du Val singularities are \emph{not} the same in these cases.
\end{rmk}

\subsection{Plan of the paper}

The paper consists of 4 sections, including this introduction.

The main purpose of \S\ref{section:slices} is to prove Theorem \ref{thm:introsubregularsliceexistence}. We lay the groundwork in \S\ref{subsection:subregularclassification} by reviewing Helmke and Slodowy's classification of subregular unstable bundles (Theorem \ref{thm:subregularclassification}). In \S\ref{subsection:induction}, we review the theory of parabolic induction for equivariant slices, and use it to reduce Theorem \ref{thm:introsubregularsliceexistence} to a statement about existence of slices for Levi subgroups of $G$ (Theorem \ref{thm:subregularsliceexistence}). We prove this theorem in \S\ref{subsection:sliceexistence} using a detailed study of the structure of the relevant Levis in \S\ref{subsection:levistructure}.

In \S\ref{section:resolutions}, we prove Theorem \ref{thm:introsubregularresolutions}. The theorem is broken into four parts, Proposition \ref{prop:subregularresolutions1}, \ref{prop:subregularresolutions2}, \ref{prop:subregularresolutions3} and \ref{prop:subregularresolutions4}, concerning the decomposition of $\tilde{\chi}_Z^{-1}(0_{\Theta_Y^{-1}})$ into irreducible components and the detailed structure of each of the three components respectively, which are proved in subsections \ref{subsection:decomposition}, \ref{subsection:dalphaj}, \ref{subsection:dalphaij} and \ref{subsection:dalphai}. This section also features a brief review of the construction of ``Bruhat cells'' for principal bundles in \S\ref{subsection:bruhat} and an important auxiliary calculation of certain Bruhat cells $G = GL_n$ in \S\ref{subsection:glnbruhat}.

In \S\ref{section:singularities}, we give the application to the identification of the singular surfaces $\chi_Z^{-1}(0)$ and their deformations. We give the identification of the surfaces (Theorem \ref{thm:subregularsingularities}) in \S\ref{subsection:singularities}. In \S\ref{subsection:deformations}, we briefly discuss deformation theory with weights, and prove (Theorem \ref{thm:subregulardeformations}) that the deformations $\chi_Z \colon Z \to \hat{Y}\sslash W$ have the miniversality properties asserted in Theorem \ref{thm:introsubregularsingularities}.

\subsection{Notation and conventions} \label{subsection:notation}

Our notations and conventions are all consistent with \cite{davis19}.

Unless otherwise specified, by a \emph{reductive group} we will mean a split connected reductive group scheme over $\spec \mb{Z}$.

Throughout the paper, we will fix a connected regular stack $S$, a smooth elliptic curve $E \to S$ with origin $O_E \colon S \to E$, and a simply connected simple reductive group $G$ (over $\spec \mb{Z}$) with maximal torus and Borel subgroup $T \subseteq B \subseteq G$. 

We will write $(\mb{X}^*(T), \Phi, \mb{X}_*(T), \Phi^\vee)$ for the root datum of $G$, where
\[ \mb{X}^*(T) = \hom(T, \mb{G}_m) \quad \text{and} \quad \mb{X}_*(T) = \hom(\mb{G}_m, T)\]
are the groups of characters and cocharacters of the split torus $T$. The set of roots $\Phi$ is by definition the set of weights of $T$ acting on the Lie algebra $\mf{g} = \mrm{Lie}(G)$; we will adopt the convention that the set $\Phi_- \subseteq \Phi$ of negative roots is the set of nonzero weights of $T$ acting on $\mrm{Lie}(B)$, and let $\Phi_+ = -\Phi_-$ be the corresponding set of positive roots. Note that this convention means that for $\lambda \in \mb{X}^*(T)$, the line bundle $\mc{L}_\lambda = G \times^B \mb{Z}_\lambda$ on the flag variety $G/B$ is nef if and only if $\lambda$ is dominant (i.e., $\langle \lambda, \alpha^\vee \rangle \geq 0$ for all $\alpha^\vee \in \Phi_+^\vee$). We will write $\Delta = \{\alpha_1, \ldots, \alpha_l\} \subseteq \Phi_+$ and $\Delta^\vee = \{\alpha_1^\vee, \ldots, \alpha_l^\vee\} \subseteq \Phi_+^\vee$ for the sets of positive simple roots and coroots respectively, and $\{\varpi_1, \ldots, \varpi_l\}$ and $\{\varpi_1^\vee, \ldots, \varpi_l^\vee\}$ for the bases of $(\mb{Z}\Phi^\vee)^\vee$ and $(\mb{Z}\Phi)^\vee$ dual to $\Delta$ and $\Delta^\vee$ respectively. Note that $\mb{Z}\Phi^\vee = \mb{X}_*(T)$ since $G$ is simply connected, so $\{\alpha_1^\vee, \ldots, \alpha_l^\vee\}$ is a basis for $\mb{X}_*(T)$ and $\{\varpi_1, \ldots, \varpi_l\}$ is a basis for $\mb{X}^*(T)$. We will also write $W \cong N_G(T)/T$ for the Weyl group of $G$ generated by the reflections $s_i \in W$ in the simple roots $\alpha_i \in \Delta$.

We will also use the notation
\[ \mb{X}_*(T)_+ = \{\lambda \in \mb{X}_*(T) \mid \langle \varpi_i, \lambda \rangle \geq 0 \; \text{for all}\; \alpha_i \in \Delta\} = \mb{Z}_{\geq 0}\Phi_+^\vee\]
and set $\mb{X}_*(T)_{-} = - \mb{X}_*(T)_+$. We have a related partial ordering on $\mb{X}_*(T)$ defined by $\lambda \leq \mu$ if $\mu - \lambda \in \mb{X}_*(T)+$. Similarly, for any reductive group and coweights $\lambda$ and $\mu$, we defined $\lambda \leq \mu$ if $\mu - \lambda$ is an integer linear combination of positive coroots with nonnegative coefficients.

If $P \subseteq G$ is a parabolic subgroup, we will say that $P$ is \emph{standard} if $B \subseteq P$, and that a Levi factor $L \subseteq P$ is \emph{standard} if $T \subseteq L$. Every parabolic subgroup is conjugate to a unique standard one, and every standard parabolic has a unique standard Levi. If $P$ is standard, the \emph{type of $P$} is the set
\[ t(P) = \{\alpha_i \in \Delta \mid \alpha_i \;\text{is not a root of}\; P\} \subseteq \Delta.\]
More generally, one defines the type of a parabolic subgroup for any reductive group with a choice of Borel as a subset of the positive simple roots in the same way. The construction $P \mapsto t(P)$ defines a bijection between (proper) parabolic subgroups of $G$ and (nonempty) subsets of $\Delta$. For any parabolic subgroup $P$, we will often write $T_P = P/[P, P]$ for the torus with character group $\mb{X}^*(T_P) = \mb{X}^*(P)$.

We also fix the following notation for the root datum and parabolic subgroups of $GL_n$. Define parabolic subgroups
\[ Q^n_k = \{(a_{p, q})_{1 \leq p, q \leq n} \in GL_n \mid a_{p, q} = 0 \;\text{for}\; p < \mrm{min}(q, k)\}\]
for $1 \leq k \leq n$. Note that $Q^n_n \subseteq GL_n$ is the Borel subgroup of lower triangular matrices, so $T_{Q^n_n} := Q^n_n/[Q^n_n, Q^n_n]$ is naturally identified with the maximal torus of diagonal matrices in $GL_n$. We will write $e_1, \ldots, e_n \in \mb{X}^*(T_{Q^n_n})$ for the basis given by $e_i(a_{j, k}) = a_{i, i}$, and $e_1^*, \ldots, e_n^* \in \mb{X}_*(T_{Q^n_n})$ for the dual basis. We label the simple roots of $GL_n$ as $\beta_i = e_i - e_{i + 1}$ for $1 \leq i < n$, so $Q^n_k \subseteq GL_n$ is the standard parabolic subgroup of type $\{\beta_1, \ldots, \beta_{k - 1}\}$.

For any group scheme $H$ over $\spec \mb{Z}$, we will write $\bun_H$ for the relative stack of $H$-bundles on $E$ over $S$. If the quotient $H/R_u(H)$ of $H$ by its unipotent radical $R_u(H)$ is split reductive and $\xi_H \to X$ is an $H$-bundle on a curve $X$, then we write $\deg \xi_H \in \mb{X}_*(H/R_u(H)[H, H])$ and $\mu(\xi_H) \in \mb{X}_*(Z(H/R_u(H))^\circ)_\mb{Q}$ for degree and slope of $\xi_H$ in the sense of \cite[\S 1.2]{davis19}. Note that these are related by
\[ \langle \lambda, \deg(\xi_H) \rangle = \langle \lambda, \mu(\xi_H) \rangle \]
for all $\lambda \in \mb{X}^*(H)$ under the obvious pairings, so in fact there is a canonical bijection between degrees and slopes. We write $\bun_H^d \subseteq \bun_H$ and $\bun_H^\mu$ for the open and closed substacks of $H$-bundles of degree $d$ and slope $\mu$ respectively. A superscript $(-)^{ss}$ denotes the open substack of semistable bundles.

For any split torus $T'$ and $\lambda \in \mb{X}^*(T')$, we write $Y_{T'}^\lambda$ for the coarse moduli space of $\bun_{T'}^\lambda$ over $S$. This can also be described as the quotient by the natural $\B T'$-action, and the fine moduli space of $T'$-bundles of degree $\lambda$ on $E$ trivialised at $O_E$. For the sake of brevity, we will drop the subscript $(-)_{T'}$ when $T' = T$ is the maximal torus of $G$, and drop the superscript $(-)^\lambda$ when $\lambda = 0$. So, in particular, $Y$ denotes the coarse moduli space of $T$-bundles on $E$ of degree $0$. We will also write $Y_P = Y_{T_P}$ when $T_P = P/[P, P]$ for some parabolic subgroup $P$ of a reductive group.

For any reductive group $H$ and parabolic subgroup $P \subseteq H$ with Levi subgroup $L \cong P/R_u(P)$, and a degree $d \in \mb{X}_*(L/[L, L])$ (resp., slope $d \in \mb{X}_*(Z(L)^\circ)_\mb{Q}$) we will write $\mrm{KM}_{P, H}^d$ for the Kontsevich-Mori compactification of $\bun_P^d$ over $\bun_G$. This is a smooth stack, proper over $\bun_H$, containing $\bun_P^d$ as a dense open substack whose complement is a divisor with normal crossings, such that the natural map $\bun_P^d \to Y_P^d$ extends to $\mrm{KM}_{P, H}^d \to Y_P^d$. It parametrises tuples $(\xi_H, C, \sigma)$ where $\xi_H \to E$ is an $H$-bundle and $\sigma \colon C \to \xi_H/P$ is a stable map from a nodal curve of genus $1$ such that $C \to E$ is degree $1$ and $\deg \sigma^*(\xi_H \times^P \mb{Z}_\lambda) = \langle \lambda, d \rangle$ for all $\lambda \in \mb{X}^*(P)$. As in the introduction, we will write $\tbun_G = \mrm{KM}_{B, G}^0$. For a detailed discussion of these compactifications, see \cite[Chapter 3]{davis19a} or \cite{campbell16}.

If $X$ is any stack equipped with an injective action of the classifying stack $\B Z(G)$ of the centre of $G$, then we write $X_{rig}$ for the rigidification of $X$ with respect to $Z(G)$ obtained by taking the quotient of all automorphism groups in $X$ by $Z(G)$ \cite[Definition 2.2.2]{davis19}. For example, if $H$ is any group scheme with $Z(G) \subseteq Z(H)$, then $\B Z(G)$ acts injectively on $\bun_H$, so we have a rigidification $\bun_{H, rig}$.

If $X \to S$ is a morphism of Artin stacks, we will write $\mb{L}_{X/S}$ for the relative cotangent complex \cite{olsson07}*{\S 8} and $\mb{T}_{X/S} = (\mb{L}_{X/S})^\vee$ for the relative tangent complex.

If $V$ is a vector space or a vector bundle on a scheme, we adopt the convention that the projectivisation $\mb{P}(V)$ parametrises $1$-dimensional subspaces or subbundles (rather than quotients).

\subsection{Acknowledgements}

The author would like to thank Ian Grojnowski, Travis Schedler, Nicholas Shepherd-Barron and Richard Thomas for many helpful conversations.

The majority of this research was done while the author was a PhD student at King's College London, and completed at the University of Edinburgh. This work was supported by the Engineering and Physical Sciences Research Council grants [EP/L015234/1] (The EPSRC Centre for Doctoral Training in Geometry and Number Theory (The London School of Geometry and Number Theory), University College London), and [EP/R034826/1].

\section{Subregular slices} \label{section:slices}

The purpose of this section is to prove Theorem \ref{thm:introsubregularsliceexistence}. We prepare the ground in \S\ref{subsection:subregularclassification}, where we review the classification of subregular unstable bundles, and \S\ref{subsection:induction}, where we review the parabolic induction construction for slices and use it to reduce Theorem \ref{thm:introsubregularsliceexistence} to a statement for Levi subgroups (Theorem \ref{thm:subregularsliceexistence}). We give very explicit descriptions of all the relevant Levi subgroups in \S\ref{subsection:levistructure}, and use these descriptions to give a case-by-case proof of Theorem \ref{thm:subregularsliceexistence} in \S\ref{subsection:sliceexistence}.

\subsection{Subregular unstable bundles} \label{subsection:subregularclassification}

In this subsection, we review Helmke and Slodowy's classification of subregular unstable bundles \cite{helmke-slodowy01}.

\begin{defn} \label{defn:subregularunstable}
Let $s \colon \spec k \to S$ be a geometric point and let $\xi_G \to E_s$ be an unstable $G$-bundle. We say that $\xi_G$ is \emph{regular} (resp., \emph{subregular}) if $\dim \mrm{Aut}(\xi_G) = l + 2$ (resp.\ $l + 4$).
\end{defn}

In the following theorem, if $s \colon \spec k \to S$ is a geometric point, $L \subseteq G$ is a Levi subgroup, and $\xi_L$ is a semistable $L$-bundle on $E_s$ of slope $\mu \in \mb{X}_*(Z(L)^\circ)_\mb{Q}$, then we say that $\xi_L$ is \emph{regular} if its automorphism group has minimal dimension among all automorphism groups of semistable $L$-bundles on $E_s$ of slope $\mu$.

\begin{thm} \label{thm:subregularclassification}
Let $s \colon \spec k \to S$ be a geometric point and let $\xi_G \to E_s$ be an unstable $G$-bundle. Then either $\xi_G$ is regular and $\dim \mrm{Aut}(\xi_G) = l + 2$, or $\dim \mrm{Aut}(\xi_G)\geq  l + 4$. If $\xi_G$ has Harder-Narasimhan reduction $\xi_P$ to a standard parabolic $P$ with Levi factor $L$, and associated $L$-bundle $\xi_L$ of slope $\mu$, then $\xi_G$ is subregular if and only if $\xi_L$ is regular semistable and $(G, P, \mu)$ satisfies one of the following conditions.
\begin{description}
\item[(Type $A_1$)\label{itm:subregularclass1}] $G$ is of type $A_1$, $t(P) = \{\alpha_1\}$ and $\langle \varpi_1, \mu \rangle = -2$.
\item[(Type $A_l$)\label{itm:subregularclass2}] $G$ is of type $A_l$ for $l > 1$, $t(P) = \{\alpha_i, \alpha_{i + 1}\}$ for some $i$ with $1 \leq i < l$, and $\langle \varpi_i, \mu \rangle = \langle \varpi_{i + 1}, \mu \rangle = -1$.
\item[(Type $B_l$)\label{itm:subregularclass3}] $G$ is of type $B_l$ for $l \geq 3$, $t(P) = \{\alpha_{l - 2}\}$ and $\langle \varpi_{l - 2}, \mu \rangle = -1$.
\item[(Type $C_l$)\label{itm:subregularclass4}] $G$ is of type $C_l$ for $l \geq 2$, $t(P) = \{\alpha_{l - 1}\}$ and $\langle \varpi_{l - 1}, \mu \rangle = -1$.
\item[(Type $D_l$)\label{itm:subregularclass5}] $G$ is of type $D_l$ for $l \geq 4$, $t(P) = \{\alpha_i\}$ and $\langle \varpi_i, \mu \rangle = -1$, where $i \in \{1, 3, 4\}$ if $l = 4$ and $i = l - 3$ otherwise.
\item[(Type $E_l$)\label{itm:subregularclass6}] $G$ is of type $D_5$, $E_6$, $E_7$ or $E_8$, $t(P) = \{\alpha_i\}$ and $\langle \varpi_i, \mu \rangle = - 1$, where $i \in \{4, 5\}$ if $G$ is of type $D_5$, $i \in \{2, 5\}$ if $G$ is of type $E_6$, and $i = 5$ if $G$ is of type $E_7$ or $E_8$.
\item[(Type $F_l$)\label{itm:subregularclass7}] $G$ is of type $B_3$ or $F_4$, $t(P) = \{\alpha_3\}$ and $\langle \varpi_3, \mu \rangle = -1$.
\item[(Type $G_l$)\label{itm:subregularclass8}] $G$ is of type $G_2$, $t(P) = \{\alpha_1\}$ and $\langle \varpi_1, \mu \rangle = -1$.
\end{description}
Here the labelling of the Dynkin diagrams is as in Table \ref{tab:dynkin}. 
\begin{table}[h]
\centering
\begin{tabular}{>{\hfill}m{3em}<{\vspace{2em}} c >{\hfill}m{3em}<{\vspace{2em}} c}
$A_l:$ & \begin{tikzpicture}[every label/.style={dlabel}]
\draw (0, 0) node[dnode, label=below:{1}] {} -- (1, 0) node[dnode, label=below:{2}] {} -- (2, 0) node[dnode, label=below:{3}] (aa) {} (3.5, 0) node[dnode, label=below:{l - 1}] (ab) {} -- (4.5, 0) node[dnode, label=below:{l}] {};
\draw[dashed] (aa) -- (ab); \end{tikzpicture} & $E_l:$ & \begin{tikzpicture}[every label/.style={dlabel}]
\draw (0, 0) node[dnode, label=below:{1}] {} -- (1, 0) node[dnode, label=below:{2}] {} -- (2, 0) node[dnode, label=below:{3}] (c) {} -- (2, 0.8) node[dnode, label=above:{4}] {} (3, 0) node[dnode, label=below:{5}] (a) {} (4.5, 0) node[dnode, label=below:{l}] (b) {};
\draw (c) -- (a);
\draw[dashed] (a) -- (b);
\end{tikzpicture} \tabularnewline
$B_l:$ & \begin{tikzpicture}[every label/.style={dlabel}]
\draw (0, -1) node[dnode, label=below:{1}] {} -- (1, -1) node[dnode, label=below:{2}] (ba) {} (2.5, -1) node[dnode, label=below:{l - 2}] (bb) {} -- (3.5, -1) node[dnode, label=below:{l - 1}] (bc) {} (4.5, -1) node[dnode, label=below:{l}] (bd) {};
\draw[dashed] (ba) -- (bb);
\draw (bc.30) -- (bd.150);
\draw (bc.330) -- (bd.210);
\draw (4, -1) node {$>$};
\end{tikzpicture}  & $F_4:$ & \begin{tikzpicture}[every label/.style={dlabel}]
\draw (0, 0) node[dnode, label=below:{1}] {} -- (1, 0) node[dnode, label=below:{2}] (a) {} (2, 0) node[dnode, label=below:{3}] (b) {} -- (3, 0) node[dnode, label=below:{4}] {};
\draw (a.30) -- (b.150);
\draw (a.330) -- (b.210);
\draw (1.5, 0) node {$>$};
\end{tikzpicture}
\tabularnewline
$C_l:$ & \begin{tikzpicture}[every label/.style={dlabel}]
\draw (0, 0) node[dnode, label=below:{1}] {} -- (1, 0) node[dnode, label=below:{2}] (a) {} (2.5, 0) node[dnode, label=below:{l - 2}] (b) {} -- (3.5, 0) node[dnode, label=below:{l - 1}] (c) {} (4.5, 0) node[dnode, label=below:{l}] (d) {};
\draw[dashed] (a) -- (b);
\draw (c.30) -- (d.150);
\draw (c.330) -- (d.210);
\draw (4, 0) node {$<$};
\end{tikzpicture}  & $G_2:$ & \begin{tikzpicture}[every label/.style={dlabel}]
\draw (0, 0) node[dnode, label=below:{1}] (a) {} -- (1, 0) node[dnode, label=below:{2}] (b) {} (a.45) -- (b.135) (a.315) -- (b.225) (0.5, 0) node {$<$};
\end{tikzpicture}
\tabularnewline
$D_l:$ & \begin{tikzpicture}[every label/.style={dlabel}]
\draw (0, 0) node[dnode, label=below:{1}] {} -- (1, 0) node[dnode, label=below:{2}] (c) {} (2.5, 0) node[dnode, label=below:{l - 3}] (d) {} -- (3.5, 0) node[dnode, label=below:{l - 2}] (a) {} -- (4.5, 0) node[dnode, label=below:{l}] {} (3.5, 0.8) node[dnode, label=above:{l - 1}] (b) {};
\draw (a) -- (b);
\draw[dashed] (c) -- (d);
\end{tikzpicture}
\end{tabular}
\caption{Labelling of the Dynkin diagrams}
\label{tab:dynkin}
\end{table}
\end{thm}
\begin{proof}
The theorem is a selection of statements from \cite[Theorems 5.1 and 5.12]{helmke-slodowy01}, which are proved there when $S = \spec \mb{C}$. To deduce the theorem in general, note that by specialisation (and \cite[Proposition 2.4, (c)]{helmke-slodowy01}, whose proof works over any field) we have
\[ \dim \mrm{Aut}(\xi_G) = -\langle 2\rho, \mu \rangle + \dim \mrm{Aut}(\xi_L) \geq - \langle 2 \rho, \mu \rangle + d(L, \mu),\]
where $d(L, \mu)$ is the dimension of the automorphism group of a regular semistable $L$-bundle with slope $\mu$ on an elliptic curve over $\mb{C}$. So \cite[Proposition 4.2.3]{davis19} and the statement of the theorem over $\mb{C}$ imply that there are no unstable bundles with $\dim \mrm{Aut}(\xi_G) = l + 3$ and that the Harder-Narasimhan reduction of any subregular unstable bundle must appear on the list above. A priori, there may be an elliptic curve $E_s$ over a field of positive characteristic such that regular semistable $L$-bundles $\xi_L$ on $E_s$ of slope $\mu$ have $\dim \mrm{Aut}(\xi_L) > d(L, \mu)$, and hence $G$-bundles with Harder-Narasimhan reductions on the list above that are not subregular. However, in case \ref{itm:subregularclass1} this cannot happen since $L = T$, and the proof of Theorem \ref{thm:subregularsliceexistence} shows that this does not happen for the other Levis and slopes on the list (see Remark \ref{rmk:automorphismbound}). So the theorem holds in all characteristics.
\end{proof}

\begin{defn}
We will say that a tuple $(G, P, \mu)$ consisting of a simply connected simple group $G$, a standard parabolic $P$ with Levi factor $L$, and a Harder-Narasimhan vector $\mu$ for $P$ is a \emph{subregular Harder-Narasimhan class} if $\xi_L \times^L G$ is subregular unstable for $\xi_L$ a regular semistable $L$-bundle of slope $\mu$. (Recall from \cite[Definition 2.3.3]{davis19} that $\mu$ is a \emph{Harder-Narasimhan vector} if, for every root $\alpha \in \Phi$ of $G$, $\alpha$ is a root of $P$ if and only if $\langle \alpha, \mu \rangle \geq 0$. By definition, the slope of a Harder-Narasimhan reduction is always a Harder-Narasimhan vector.) We will say that $(G, P, \mu)$ is \emph{of type $A_1$} (resp., \emph{type $A_l$}, \emph{type $B_l$}, etc.\@) if it satisfies \ref{itm:subregularclass1} (resp., \ref{itm:subregularclass2}, \ref{itm:subregularclass3}, etc.\@) of Theorem \ref{thm:subregularclassification}.
\end{defn}

\begin{rmk}
We stress that the type of a subregular Harder-Narasimhan class $(G, P, \mu)$ is often, but not always, the type of the group $G$. For example, for $G$ of type $B_3$, there are subregular Harder-Narasimhan classes of types $B_3$ and $F_3$, and for $G$ of type $D_5$, there are subregular Harder-Narasimhan classes of types $D_5$ and $E_5$.
\end{rmk}

\subsection{Slicing by parabolic induction} \label{subsection:induction}

In this subsection, we explain how the proof of Theorem \ref{thm:introsubregularsliceexistence} can be reduced to the construction of well-behaved slices of $\bun_{L, rig}^{ss, \mu}$ for each subregular Harder-Narasimhan class. We first recall the definitions.

\begin{defn}
Let $L \subseteq G$ be a Levi subgroup. A \emph{slice of $\bun_{L, rig}$} is stack $Z$ equipped with a map $Z \to \bun_{L, rig}$ such that the map $Z \to \bun_{L, rig}/E$ is smooth, where the quotient is taken with respect to the natural action of $E$ on $\bun_{L, rig}$ by translations. If $H$ is a torus and $\lambda \in \mb{X}^*(H)$ is an \emph{equivariant slice (of $\bun_{G, rig}$) with equivariance group $H$ and weight $\lambda$} is a stack $Z$ equipped with an action of $H$, a slice $Z/H \to \bun_{G, rig}$, and an $H$-equivariant lift $Z \to \hat{Y}\sslash W$ of the coarse quotient map $Z \to \bun_{G, rig} \to (\hat{Y}\sslash W)/\mb{G}_m$, where $H$ acts on $\hat{Y}\sslash W$ through the $\mb{G}_m$-action and the homomorphism $\lambda \colon H \to \mb{G}_m$.
\end{defn}

We also recall a few elements of the theory of parabolic induction for slices, the idea of which goes back to R. Friedman and J. Morgan \cite{friedman-morgan00}. A more detailed exposition can be found in \cite[\S 4.1]{davis19} or \cite[\S 5.1-5.2]{davis19a}.

\begin{defn}[{\cite[Definition 4.1.1]{davis19} or \cite[Definition 5.2.1]{davis19a}}] \label{defn:parabolicinduction}
Let $L \subseteq L' \subseteq G$ be Levi subgroups, let $\mu \in \mb{X}_*(Z(L)^\circ)_\mb{Q}$, and let $P^+ \subseteq L'$ be the unique parabolic subgroup with Levi factor $L$ for which $-\mu$ is a Harder-Narasimhan vector for $P^+$ \cite[Definition 2.3.3]{davis19}. If $Z_0 \to \bun_{L, rig}^{ss, \mu}$ is a slice, then the \emph{parabolic induction of $Z_0$ to $L'$} is the slice
\[ \mrm{Ind}_L^{L'}(Z_0) = \bun_{P^+, rig} \times_{\bun_{L, rig}} Z_0 \longrightarrow \bun_{P^+, rig}^{\mu'} \longrightarrow \bun_{L', rig}^{\mu'},\]
where $\mu' \in \mb{X}_*(Z(L')^\circ)_\mb{Q}$ is the image of $\mu$.
\end{defn}

In the following proposition, we write
\[ \Phi_\mu = \{\alpha \in \Phi_{L'} \mid \langle \alpha, \mu \rangle < 0\} \]
for the set of roots in the unipotent radical $R_u(P^+) \subseteq L'$, and
\[ 2\rho_{P^+} = - \sum_{\alpha \in \Phi_\mu} \alpha.\]

\begin{prop} \label{prop:inductionaffinebundle}
In the situation of Definition \ref{defn:parabolicinduction}, the natural morphism $\mrm{Ind}_L^{L'}(Z_0) \to Z_0$ is an affine space bundle, with fibres of dimension $\langle 2\rho_{P^+}, \mu \rangle$. Moreover, the torus $Z(L)_{rig} := Z(L)/Z(G)$ naturally acts on the fibres of this bundle with weights in $-\Phi_\mu$. The morphism $\mrm{Ind}_L^{L'}(Z_0) \to \bun_{L', rig}$ is equivariant with respect to this action.
\end{prop}
\begin{proof}
This is an immediate consequence of \cite[Propositions 4.1.6 and 4.1.8]{davis19}.
\end{proof}

From now on, we will take $L' = G$.

Recall from \cite[\S 3.2]{davis19} that there is a unique positive generator $\Theta_{\bun_{G, rig}} \in \mrm{Pic}(\bun_{G, rig})$; this theta bundle is nothing but the inverse of the pullback of the universal line bundle under
\[ \bun_{G, rig} \longrightarrow (\hat{Y}\sslash W)/\mb{G}_m \longrightarrow \B \mb{G}_m. \]

\begin{defn}[{\cite[Definition 4.1.15]{davis19}}]
Let $L$ and $\mu$ be as in Definition \ref{defn:parabolicinduction}. A \emph{$\Theta$-trivial slice} of $\bun_{L, rig}^{ss, \mu}$ is a slice $Z_0 \to \bun_{L, rig}^{ss, \mu}$ equipped with a trivialisation of the pullback of $\Theta_{\bun_{G, rig}}$ along
\[ Z_0 \longrightarrow \bun_{L, rig}^{ss, \mu} \longrightarrow \bun_{G, rig}.\]
\end{defn}

The point of $\Theta$-trivial slices is that they naturally give equivariant slices after parabolic induction. In the following proposition, we write $(\,\mmid\,) \colon \mb{X}_*(T) \otimes \mb{X}_*(T) \to \mb{Z}$ for the Killing form normalised so that $(\alpha^\vee \mmid \alpha^\vee) = 2$ for $\alpha^\vee \in \Phi^\vee$ a short coroot.

\begin{prop}[{\cite[Proposition 4.1.12]{davis19}}] \label{prop:inductionequivariantslice}
Let $Z_0 \to \bun_{L, rig}^{ss, \mu}$ be a $\Theta$-trivial slice. Then $\mrm{Ind}_L^G(Z_0) \to \bun_{G, rig}$ is naturally endowed with the structure of an equivariant slice with equivariance group $Z(L)_{rig}$ and weight $(\mu \mmid -)$.
\end{prop}

The parabolic induction construction allows us to deduce Theorem \ref{thm:introsubregularsliceexistence} from the following statement.

\begin{thm} \label{thm:subregularsliceexistence}
Let $(G, P, \mu)$ be a subregular Harder-Narasimhan class not of type $A_1$, and let $d \in \{1, 2, 3\}$ be as in Theorem \ref{thm:introsubregularresolutions}. Then there is a $\mu_d$-gerbe $\mf{G}^{uni}$ on the stack $M_{1, 1}$ of elliptic curves such that if the pullback $\mf{G}$ of $\mf{G}^{uni}$ to $S$ is trivial then there exists a $\Theta$-trivial slice $Z_0 \to \bun_{L, rig}^{ss, \mu}$ with the following properties.
\begin{enumerate}[(1)]
\item \label{itm:subregularsliceexistence1} The morphism $Z_0 \to S$ is smooth and proper with finite and generically trivial relative stabilisers.
\item \label{itm:subregularsliceexistence2} The morphism $Z_0 \to \bun_{L, rig}^{ss, \mu}/E$ is smooth with connected fibres.
\item \label{itm:subregularsliceexistence3} The image of $Z_0 \to \bun_{L, rig}^{ss, \mu}/E$ is equal to the locus of regular semistable bundles.
\item \label{itm:subregularsliceexistence4} The induced equivariant slice $Z = \mrm{Ind}_L^G(Z_0) \to \bun_{G, rig}$ has relative dimension $l + 3$ over $S$.
\end{enumerate}
\end{thm}

We will prove Theorem \ref{thm:subregularsliceexistence} in \S\ref{subsection:sliceexistence} by writing down explicit slices in each case of Theorem \ref{thm:subregularclassification}. Although a classification-free proof is probably possible, the explicit slices are also useful in the proof of Theorem \ref{thm:introsubregularresolutions}.

\begin{rmk} \label{rmk:automorphismbound}
The proof will show that Theorem \ref{thm:subregularsliceexistence} holds for every tuple $(G, P, \mu)$ on the list of Theorem \ref{thm:subregularclassification}, excluding \ref{itm:subregularclass1}. In the notation of the proof of Theorem \ref{thm:subregularclassification}, this shows that in each case we have a slice $Z_0 \to \bun_{L, rig}^{ss, \mu}$ with relative dimension $l + 3 + \langle 2 \rho, \mu \rangle = d(L, \mu) - 1$ over $S$, and hence relative dimension $d(L, \mu)$ over $\bun_{L, rig}^{ss, \mu}/E$. Since $Z_0 \to S$ has finite relative stabilisers, this shows that $\dim \mrm{Aut}(\xi_L) \leq d(L, \mu)$ for a regular semistable $L$-bundle in all characteristics.
\end{rmk}

\begin{cor}
Theorem \ref{thm:introsubregularsliceexistence} is true (with $S = \spec k$ for $k$ an algebraically closed field as in the introduction).
\end{cor}
\begin{proof}
Let $(G, P, \mu)$ be the subregular Harder-Narasimhan class of $\xi_G$. Since $S = \spec k$, the $\mu_d$-gerbe $\mf{G}$ on $S$ of Theorem \ref{thm:subregularsliceexistence} is necessarily trivial, so there exists a $\Theta$-trivial slice $Z_0 \to \bun_{L, rig}^{ss, \mu}$ satisfying conditions \eqref{itm:subregularsliceexistence1}--\eqref{itm:subregularsliceexistence4}. We let $Z = \mrm{Ind}_L^G(Z_0) \to \bun_{G, rig}$ be the parabolic induction of $Z_0$ to $G$, endowed with the equivariant slice structure of Proposition \ref{prop:inductionequivariantslice}. Note that the equivariance group $H = Z(L)_{rig}$ is isomorphic to $\mb{G}_m \times \mb{G}_m$ for $(G, P, \mu)$ of type $A$ and $\mb{G}_m$ otherwise, as required for the statement of Theorem \ref{thm:introsubregularsliceexistence}.

Condition \eqref{itm:introsubregularsliceexistence1} of Theorem \ref{thm:introsubregularsliceexistence} follows immediately from Proposition \ref{prop:inductionaffinebundle} and \eqref{itm:subregularsliceexistence1} of Theorem \ref{thm:subregularsliceexistence} (note that $Z_0 = \mrm{Ind}_L^G(Z_0)^{Z(L)_{rig}}$). Condition \eqref{itm:introsubregularsliceexistence3} of Theorem \ref{thm:introsubregularsliceexistence} follows from \eqref{itm:subregularsliceexistence2} of Theorem \ref{thm:subregularsliceexistence}.

To prove that Theorem \ref{thm:introsubregularsliceexistence} \eqref{itm:introsubregularsliceexistence2} is satisfied, first note that for any $z \in Z \setminus Z_0$, comparing the codimensions of $Z(L)_{rig} \cdot z$ in $Z$ and the corresponding $G$-bundle $\xi_{G, z}$ in $\bun_{G, rig}/E$ shows that $\dim \mrm{Aut}(\xi_{G, z}) \leq l + 3$, so $\xi_{G, z}$ is not subregular. Moreover, we claim that the smooth morphism $Z_0 \to \bun_{L, rig}^{ss, reg, \mu}/E$ is a bijection on $K$-points for any algebraically closed field $K$, from which it follows that it is an isomorphism on coarse moduli spaces. This proves \eqref{itm:introsubregularsliceexistence2}, modulo the claim.


To prove the claim, note that Proposition \ref{prop:inductionaffinebundle} shows that the dimension of the fibre $(Z_0)_x$ of $Z_0 \to \bun_{L, rig}^{ss, \mu}/E$ over the image $x$ of $\xi_L = \xi_P \times^P L$ is given by
\[ l + 4 - \langle 2 \rho_{P^+}, \mu \rangle = l + 4 + \langle 2 \rho, \mu \rangle.\]
But from Remark \ref{rmk:automorphismbound} and the proof of Theorem \ref{thm:subregularclassification}, this is equal to the dimension $d(L, \mu)$ of $\mrm{Aut}(\xi_L)$. So $(Z_0)_x/\mrm{Aut}(x) \subseteq (Z_0)_s \subseteq Z_s$ is a closed connected substack of dimension $0$, where $s$ is the image of $x$ in $S$, and hence has a single point over any algebraically closed field since $(Z_0)_s$ has finite stabilisers.
\end{proof}

\begin{rmk}
From the proof, we can also read off the weights of the equivariant slices in Theorem \ref{thm:introsubregularsliceexistence}: abstractly, they are the characters $(\mu \mmid -) \in \mb{X}^*(Z(L)_{rig})$ by Proposition \ref{prop:inductionequivariantslice}. More explicitly, if we identify $Z(L)_{rig}$ with $\mb{G}_m$ via the cocharacter $-\varpi_i^\vee \in \mb{X}_*(Z(L)_{rig})$ where $t(P) = \{\alpha_i\}$ (resp., with $\mb{G}_m \times \mb{G}_m$ via $(-\varpi_i^\vee, -\varpi_{i + 1}^\vee)$ in type $A$), then the weight is identified with $(1, 1)$ in type $A$ and with $d \in \{1, 2, 3\}$ in the other types.
\end{rmk}

\begin{rmk} \label{rmk:sl2slice}
We have deliberately excluded the subregular Harder-Narasimhan class of type $A_1$ from Theorem \ref{thm:subregularsliceexistence}. In this case, we have $L = T \cong \mb{G}_m$ and $\bun_{L, rig}^{ss, \mu} = \bun_{\mb{G}_m, rig}^{-2}$, and one can try to construct the desired slice $Z_0 = S \to \bun_L^{ss, \mu}$ by lifting the natural section $\mc{O}(-2O_E) \colon Z_0 = S \to \mrm{Pic}^{-2}_S(E)$. It follows from \cite[Proposition 4.1.15]{davis19} that the fibre of the map
\[ \bun_{\mb{G}_m, rig}^{-2} = \mrm{Pic}_S^{-2}(E) \times \B \mb{G}_m \longrightarrow \B\mb{G}_m \]
classifying the pullback of the theta bundle is a $\mu_2$-gerbe on $\mrm{Pic}_S^{-2}(E)$, which is trivial if and only if $Z_0 \to \mrm{Pic}^{-2}_S(E)$ lifts to a $\Theta$-trivial map $Z_0 \to \bun_{L, rig}^{ss, \mu}$. This map will be a slice as long as $2$ is invertible in $\mc{O}_S$ (so that the stabiliser $E[2]$ of a point in $\mrm{Pic}^{-2}_S(E)$ is smooth). This slice satisfies \eqref{itm:subregularsliceexistence1}, \eqref{itm:subregularsliceexistence3} and \eqref{itm:subregularsliceexistence4}, but the map $Z_0 \to \bun_{L, rig}^{ss, -2}/E$ is a torsor under an extension of $E[2]$ by $\mb{G}_m$ and hence has disconnected fibres. See also, however, Proposition \ref{prop:sl2sing}.
\end{rmk}

\subsection{The structure of Levi subgroups} \label{subsection:levistructure}

In this subsection, we explicitly describe the Levi subgroups $L \subseteq P$ for each subregular Harder-Narasimhan class, i.e.\ for each $(G, P, \mu)$ on the list of Theorem \ref{thm:subregularclassification}.

We begin with a general description of Levi subgroups $L \subseteq G$ whose Dynkin diagrams are of type $A$. Suppose that $L$ is the Levi subgroup of a standard parabolic of type $t \subseteq \Delta$. Then the Dynkin diagram of $L$ is obtained from the Dynkin diagram of $G$ by deleting the nodes labelled by elements of $t$. We will assume that the Dynkin diagram of $L$ is a union of connected components of type $A$.

The reductive group $L$ can be described directly in terms of the following data. First, write $\pi_0 = \pi_0(\Delta \setminus t)$ for the set of connected components of the Dynkin diagram of $L$. For each component $c \in \pi_0$, write $n_c$ for the number of nodes in $c$, and choose a labelling $\alpha_{c, 1}, \ldots, \alpha_{c, n_c}$ of the nodes of $c$ so that $\alpha_{c, i}$ is adjacent to $\alpha_{c, i + 1}$ for $1 \leq i \leq n_c - 1$. For each $\alpha_k \in t$ adjacent to a node of $c$, let $\alpha_{c, i_{c, k}}$ be the unique node adjacent to $\alpha_k$, and for each $\alpha_k \in t$ not adjacent to any node of $c$, set $i_{c, k} = n_c + 1$. Finally, write
\[ m_{c, k} = - \sum_{i = 1}^{n_c} \langle \alpha_{c, i}, \alpha_k^\vee \rangle = \begin{cases} -\langle \alpha_{c, i_{c, k}}, \alpha_k^\vee \rangle, & \text{if}\;\; i_{c, k} \leq n_c, \\ 0, & \text{if}\;\; i_{c, k} = n_c + 1, \end{cases} \]
for $c \in \pi_0$ and $\alpha_k \in t$.

\begin{prop} \label{prop:typealevi}
Assume we are in the setup above. Then there is an isomorphism
\begin{equation} \label{eq:typealevi}
 L \overset{\sim}\longrightarrow \left\{((A_c)_{c \in \pi_0}, (\lambda_k)_{\alpha_k \in t}) \in \prod_{c \in \pi_0}GL_{n_c + 1} \times \prod_{\alpha_k \in t} \mb{G}_m \,\left|\, \det A_c = \prod_{\alpha_k \in t} \lambda_k^{m_{c, k}(n_c + 1 - i_{c, k})}\right.\right\}
\end{equation}
with the property that for each $\alpha_k \in t$, the character $\varpi_k$ of $L$ is given by \eqref{eq:typealevi} composed with the projection $((A_c)_{c \in \pi_0}, (\lambda_j)_{\alpha_j \in t}) \mapsto \lambda_k$.
\end{prop}
\begin{proof}
Since both sides of \eqref{eq:typealevi} are split reductive groups over $\spec \mb{Z}$, it is enough to specify an isomorphism between their root data.

The root datum $(M_0, \Psi_0, M_0^\vee, \Psi_0^\vee)$ of $\prod_{c \in \pi_0} GL_{n_c + 1} \times \prod_{\alpha_k \in t} \mb{G}_m$ is specified as follows. The weight lattice is
\[ M_0 = \bigoplus_{c \in \pi_0} \mb{Z}^{n_c + 1} \oplus \bigoplus_{\alpha_k \in t} \mb{Z}\omega_k.\]
The roots and coroots $\Psi_0$ and $\Psi_0^\vee$ are determined by requiring that
\[ \{\beta_{c, j} = e_{c, j} - e_{c, j + 1} \mid c \in \pi_0 \; \text{and}\; 1 \leq j \leq n_c\} \subseteq M_0 \]
be a set of positive simple roots for $\Psi_0$, and that
\[ \beta_{c, j}^\vee = e_{c, j}^* - e_{c, j + 1}^*\]
where $\{e_{c, j} \mid 1 \leq j \leq n_c + 1\}$ is the standard basis for $\mb{Z}^{n_c + 1}$, and $e_{c, j}^* \in M_0^\vee$ satisfies
\[ \langle e_{c', j'}, e_{c, j}^* \rangle = \begin{cases} 1, & \text{if}\;\; (c', j') = (c, j), \\ 0, & \text{otherwise},\end{cases} \quad \text{and} \quad \langle \omega_k, e_{c, j}^* \rangle = 0.\]
The root datum $(M, \Psi, M^\vee, \Psi^\vee)$ is given by setting
\[ M = \frac{M_0}{\mb{Z}\tn{-span}\left\{\left.\sum_{j = 1}^{n_c + 1} e_{c, j} - \sum_{\alpha_k \in t} m_{c, k}(n_c + 1 - i_{c, k})\omega_k \, \right|\, c \in \pi_0 \right\}},\]
setting $\Psi$ to be the image of $\Psi_0$ in $M$, and setting $\Psi^\vee \subseteq M^\vee$ to be the preimage of $\Psi_0^\vee$ under the injection $M^\vee \hookrightarrow M_0^\vee$. Note that $M$ is indeed a lattice, so this is the root datum of a connected reductive group.

We define an isomorphism of $(M, \Psi, M^\vee, \Psi^\vee)$ with the root datum $(\mb{X}^*(T), \Phi_t, \mb{X}_*(T), \Phi_t^\vee)$ of $L$ via the isomorphism
\begin{align*}
\phi \colon \mb{X}_*(T) &\overset{\sim}\longrightarrow M^\vee \\
\alpha_{c, j}^\vee &\longmapsto e_{c, j}^* - e_{c, j + 1}^* \\
\alpha_k^\vee &\longmapsto \omega_k^* + \sum_{c \in \pi_0}\sum_{j = i_{c, k} + 1}^{n_c + 1} m_{c, k} e_{c, j}^*,
\end{align*}
for $c \in \pi_0$, $1 \leq j \leq n_c$ and $\alpha_k \in t$, where $\omega_k^* \in M_0^\vee$ satisfies $\langle e_{c, j}, \omega_k^* \rangle = 0$ and $\langle \omega_{k'}, \omega_k^*\rangle = \delta_{k, k'}$. It is clear by inspection that $\phi$ is a well-defined homomorphism of free abelian groups such that the dual is surjective. Since $M^\vee$ and $\mb{X}_*(T)$ have the same rank, $\phi$ is therefore an isomorphism. To prove that $\phi$ defines an isomorphism of root data, it is enough to show that $\phi \colon \mb{X}_*(T) \to M^\vee$ sends $\alpha_{c, j}^\vee$ to $\beta_{c, j}^\vee$ and that $\phi^* \colon M \to \mb{X}^*(T)$ sends $\beta_{c, j}$ to $\alpha_{c, j}$ for all $c \in \pi_0$ and $1 \leq j \leq n_c$. This is easily checked by direct calculation, so we are done.
\end{proof}

Now fix a subregular Harder-Narasimhan class $(G, P, \mu)$ not of type $A_1$. It will be convenient to decompose the Dynkin diagram of $G$ as follows.

\begin{notation} \label{notation:dynkindecomposition}
If $(G, P, \mu)$ is of type $A$, then let $\{\alpha_i, \alpha_j\} = \{\alpha_i, \alpha_{i + 1}\} = t(P)$. Otherwise, we let $\{\alpha_i\} = t(P)$ and let $\alpha_j \in \Delta$ be the unique special root. (Recall \cite[Definition 3.1.1]{friedman-morgan00} \cite[Definition 4.2.1]{davis19} that a simple root $\alpha \in \Delta$ is \emph{special} if it is a long root such that the Dynkin diagram $\Delta \setminus \{\alpha\}$ is a union of components of type $A$ each meeting $\alpha$ at a single end.) Theorem \ref{thm:subregularclassification} shows that in each case, $\alpha_i$ is adjacent to $\alpha_j$. Deleting the edge joining $\alpha_i$ and $\alpha_j$ breaks the Dynkin diagram of $G$ into two connected components; we write $c_0$ (resp., $c_1$) for the component containing $\alpha_i$ (resp., $\alpha_j$) and $n_0$ (resp., $n_1$) for the number of vertices in $c_0$ (resp., $c_1$). Since $\alpha_j$ is special, the Dynkin diagram of $c_0$ is of type $A_{n_0}$. We write $\{\alpha_{c_0, 1}, \ldots, \alpha_{c_0, n_0}\} \subseteq \Delta$ for the vertices of $c_0$, labelled so that $\alpha_{c_0, k}$ is adjacent to $\alpha_{c_0, k + 1}$ for all $k < n_0$ and $\alpha_{c_0, n_0} = \alpha_i$. For $c \in \{c_0, c_1\}$ and $\alpha_{c, k}$ a root of $c$, we also write $\varpi_{c, k} \in \mb{X}^*(T)$ for the corresponding fundamental dominant weight.
\end{notation}

Our descriptions of the Levi subgroup $L \subseteq P$ fall into four distinct cases.

\vspace{1ex}
\noindent {\it Case 1: $(G, P, \mu)$ is of type $A$.} In this case, we have the following elementary description of the Levi $L$.

\begin{lem} \label{lem:typealeviiso}
In the setup above, there is an isomorphism
\[ L \cong GL_i \times GL_{l - i} = GL_{n_0} \times GL_{n_1}\]
so that the characters $\varpi_i$ and $\varpi_{i + 1}$ are identified with the determinants of the first and second factors respectively.
\end{lem}
\begin{proof}
The desired isomorphism is given by applying Proposition \ref{prop:typealevi} with an appropriate labelling. Explicitly, it is given by
\begin{align*}
GL_i \times GL_{l - i} &\overset{\sim}\longrightarrow L \subseteq{SL_{l + 1}} \\
(A, B) &\longmapsto \left(\begin{matrix} A & 0 & 0 \\ 0 & (\det A)^{-1} \det B & 0 \\ 0 & 0 & v(B^t)^{-1}v^{-1} \end{matrix}\right),
\end{align*}
where $v \in S_{l - i}$ is the matrix of the permutation of $\{1, \ldots, l - i\}$ sending $j$ to $l - i - j + 1$.
\end{proof}

\noindent {\it Case 2: $(G, P, \mu)$ is of type $C$, $D$, $E$ or $F$.} In this case, the connected component $c_1$ containing $\alpha_j$ of the Dynkin diagram with the edge joining $\alpha_i$ and $\alpha_j$ deleted is of type $A_{n_1}$, and we can choose a labelling $\alpha_{c_1, 1}, \ldots, \alpha_{c_1, n_1}$ such that $\alpha_{c_1, p}$ is adjacent to $\alpha_{c_1, p + 1}$ for each $p$ and $\alpha_j$ is either $\alpha_{c_1, n_1}$ (in types $C$ and $F$) or $\alpha_{c_1, n_1 - 1}$ (in types $D$ and $E$). We have the following description of $L$.

\begin{lem} \label{lem:typecdefleviiso}
In the setup above, there is an isomorphism
\[ L \cong \{(A, B) \in GL_{n_0} \times GL_{n_1 + 1} \mid \det B = (\det A)^2\},\]
such that $\varpi_i$ is identified with the character $(A, B) \mapsto \det A$ and $L \cap B$ is the preimage of the Borel subgroup $Q_{n_0}^{n_0} \times Q_{n_1 + 1}^{n_1 + 1} \subseteq GL_{n_0} \times GL_{n_1 + 1}$. Moreover, the induced map
\[ \mb{X}^*(Q_{n_1 + 1}^{n_1 + 1}) \longrightarrow \mb{X}^*(L \cap B) = \mb{X}^*(T) \]
is given in types $D$ and $E$ by
\[ e_k \longmapsto \begin{cases} \varpi_{c_1, 1}, & \text{if}\;\; k = 1, \\ \varpi_{c_1, k} - \varpi_{c_1, k - 1}, & \text{if}\;\; 1 < k < n_1, \\
\varpi_{c_1, n_1} - \varpi_{c_1, n_1 - 1} + \varpi_i, & \text{if}\;\; k = n_1, \\
-\varpi_{c_1, n_1} + \varpi_i, & \text{if}\;\; k = n_1 + 1, \end{cases} \]
and in types $C$ and $F$ by
\[ e_k \longmapsto \begin{cases} \varpi_{c_1, 1}, & \text{if}\;\; k = 1, \\ \varpi_{c_1, k} - \varpi_{c_1, k - 1}, & \text{if}\;\; 1 < k \leq n_1, \\
-\varpi_{c_1, n_1} + 2\varpi_i, & \text{if}\;\; k = n_1 + 1. \end{cases} \]
\end{lem}
\begin{proof}
Apply Proposition \ref{prop:typealevi}; the expressions for $\mb{X}^*(Q_{n_1 + 1}^{n_1 + 1}) \to \mb{X}^*(T)$ follow by examining the specific isomorphism given in the proof.
\end{proof}

\noindent {\it Case 3: $(G, P, \mu)$ is of type $G$.} In type $G$, the Levi $L$ has a similarly simple form.

\begin{lem} \label{lem:typegleviiso}
Assume that $(G, P, \mu)$ is of type $G$. Then there is an isomorphism
\begin{equation} \label{eq:typegleviiso1}
 L \overset{\sim}\longrightarrow \{(\lambda, A) \in \mb{G}_m \times GL_2 \mid \det A = \lambda^3\}
\end{equation}
such that $\varpi_1$ is identified with the character $(\lambda, A) \mapsto \lambda$ and $L \cap B$ is the preimage of the Borel subgroup $\mb{G}_m \times Q^2_2 \subseteq \mb{G}_m \times GL_2$. Moreover, the induced morphism
\[ \mb{X}^*(Q^2_2) \longrightarrow \mb{X}^*(L \cap B) = \mb{X}^*(T) \]
sends the characters $e_1$ and $e_2$ to $\varpi_2$ and $3\varpi_1 -\varpi_2$ respectively.
\end{lem}
\begin{proof}
Apply Proposition \ref{prop:typealevi} again and inspect the explicit isomorphism given in the proof to compute $\mb{X}^*(Q^2_2) \to \mb{X}^*(T)$.
\end{proof}

\noindent {\it Case 4: $(G, P, \mu)$ is of type $B$.} This case is somewhat more complicated, as the Levi subgroup $L$ is not of type $A$. In what follows, we let
\[ GSp_4 = \{B \in GL_4 \mid B^tJB = \chi(B)J\; \text{for some}\; \chi(B) \in \mb{G}_m\},\]
where
\[ J = \left(\begin{matrix} 0 & 0 & 0 & 1 \\ 0 & 0 & 1 & 0 \\ 0 & -1 & 0 & 0 \\ -1 & 0 & 0 & 0 \end{matrix}\right).\]
Note that $GSp_4$ is a reductive group with weight lattice $\mb{X}^*(GSp_4 \cap Q^4_4) = \bigoplus_{k = 1}^4 \mb{Z}f_k/\mb{Z}(f_1 - f_2 - f_3 + f_4)$, where $f_k$ is the character sending a matrix to its $k$th diagonal entry, simple roots $\beta_1 = f_2 - f_3$ and $\beta_2 = f_1 - f_2$, and simple coroots $\beta_1^\vee = f_2^* - f_3^*$ and $\beta_2^\vee = f_1^* - f_2^* + f_3^* - f_4^*$. In this description, $\chi$ is the character $\chi = f_1 + f_4 = f_2 + f_3$.

\begin{lem} \label{lem:typebleviiso}
If $(G, P, \mu)$ is of type $B$, then there is an isomorphism
\[ L \overset{\sim}\longrightarrow \{(A, B) \in GL_{l - 2} \times GSp_4 \mid \det(A) = \chi(B)\},\]
such that $\varpi_i = \varpi_{l - 2}$ is identified with the character $(A, B) \mapsto \det(A) = \chi(B)$ and $L \cap B$ is the preimage of the Borel subgroup $Q^{l - 2}_{l - 2} \times (GSp_4 \cap Q^4_4) \subseteq GL_{l - 2} \times GSp_4$. Moreover, the induced morphism
\[ \mb{X}^*(GSp_4 \cap Q^4_4) = \bigoplus_{k = 1}^4 \mb{Z}f_k \longrightarrow \mb{X}^*(L \cap B) = \mb{X}^*(T) \]
sends $f_1$, $f_2$, $f_3$ and $f_4$ to $\varpi_l$, $\varpi_{l - 1} - \varpi_l$, $\varpi_{l - 2} - \varpi_{l - 1} + \varpi_l$ and $\varpi_{l - 2} - \varpi_l$ respectively.
\end{lem}
\begin{proof}
We describe the isomorphism at the level of root data. 

Write
\[ L_0 = \{(A, B) \in GL_{l - 2} \times GSp_4 \mid \det(A) = \chi(B)\} \subseteq GL_{l - 2} \times GSp_4.\]
The root datum $(M, \Psi, M^\vee, \Psi^\vee)$ of $L_0$ is specified as follows. The weight lattice is
\[ M = \frac{\bigoplus_{i = 1}^{l - 2} \mb{Z}e_i \oplus \bigoplus_{j = 1}^4 \mb{Z}f_j}{\langle f_1 - f_2 - f_3 + f_4, f_1 + f_4 - \sum_{i = 1}^{l - 2} e_i \rangle},\]
and the coweight lattice is therefore
\[ M^\vee = \left\{\left. \lambda \in \bigoplus_{i = 1}^{l - 2} \mb{Z}e_i^* \oplus \bigoplus_{j = 1}^4 \mb{Z}f_j^* \, \right| \, \langle f_1 + f_4, \lambda \rangle = \langle f_2 + f_3, \lambda \rangle = \sum_{i = 1}^{l - 2} \langle e_i, \lambda \rangle\right\}.\]
The roots $\Psi$ and coroots $\Psi^\vee$ and the bijection $\Psi \to \Psi^\vee$ are determined by requiring that
\[ \{\gamma_i \mid 1 \leq i \leq l,\, i \neq l - 2\}\]
be a set of simple roots, where
\[ \gamma_i = \begin{cases} e_i - e_{i + 1}, & \text{if}\;i < l - 2,\\ f_2 - f_3, & \text{if}\;i = l - 1, \\ f_1 - f_2, & \text{if}\;i = l, \end{cases} \qquad \text{and} \qquad \gamma_i^\vee = \begin{cases} e_i^* - e_{i + 1}^*, & \text{if}\;i < l - 2, \\ f_2^* - f_3^*,& \text{if}\;i = l - 1, \\ f_1^* - f_2^* + f_3^* - f_4^*, & \text{if}\; i = l.\end{cases}\]

There is an isomorphism
\[ \phi \colon \mb{X}_*(T) = \bigoplus_{i = 1}^l \mb{Z}\alpha_i^\vee \overset{\sim}\longrightarrow M^\vee \]
sending $\alpha_i^\vee$ to $\gamma_i^\vee$ for $i \neq l - 2$, and $\alpha_{l - 2}^\vee$ to $e_{l - 2}^* + f_3^* + f_4^*$, such that the dual $\phi^* \colon M \to \mb{X}^*(T)$ sends $\beta_i$ to $\alpha_i$ for $i \neq l - 2$. So $\phi$ defines an isomorphism of root data, which has the desired properties by inspection.
\end{proof}

\subsection{Existence of slices} \label{subsection:sliceexistence}

In this section, we give the proof of Theorem \ref{thm:subregularsliceexistence}. The proof we give here is somewhat ad hoc, and relies on the explicit descriptions of the Levi subgroups given in \S\ref{subsection:levistructure}.

We first give the construction in type $A$.

\begin{proof}[Proof of Theorem \ref{thm:subregularsliceexistence} in type $A$]
First note that since $d = 1$ in this case, the $\mu_d = \mu_1$-gerbe $\mf{G}^{uni}$ must be the trivial one.

Using the identification $L \cong GL_i \times GL_{l - i}$, Atiyah's classification of stable vector bundles (in the form \cite[Theorem 4.2.6]{davis19}) implies that the morphism
\[ (\varpi_i, \varpi_{i + 1}) \colon \bun_{L, rig}^{ss, \mu} \longrightarrow \mrm{Pic}^{-1}_S(E) \times_S \mrm{Pic}^{-1}_S(E)\]
is a trivial $Z(L)_{rig}$-gerbe. Note that in particular, all semistable $L$-bundles of slope $\mu$ are regular in this case.

By \cite[Proposition 4.1.15]{davis19}, the pullback of $\Theta$ to $\bun_{L, rig}^{ss, \mu}$ has $Z(L)_{rig}$-weight $(-\mu \mmid -) \in \mb{X}^*(Z(L)_{rig})$. (This means that tensoring a map $X \to \bun_{L, rig}$ with a $Z(L)_{rig}$-torsor $\eta$ on $X$ tensors the pullback of $\Theta_{\bun_{G, rig}}$ to $X$ with the line bundle $(-\mu \mmid \eta)$.) Since the corresponding homomorphism $\mb{X}_*(Z(L)_{rig}) \to \mb{Z}$ is surjective, it follows that there exists a section
\[ \mrm{Pic}^{-1}_S(E) \times_S \mrm{Pic}^{-1}_S(E) \longrightarrow \bun_{L, rig}^{ss, \mu} \]
such that the pullback of $\Theta_{\bun_{G, rig}}$ is trivial. Since such a section is necessarily smooth, composing it with any choice of section of
\[ \mrm{Pic}^{-1}_S(E) \times_S \mrm{Pic}^{-1}_S(E) \longrightarrow \mrm{Pic}^{-1}_S(E) \times_S \mrm{Pic}^{-1}_S(E)/E \cong E\]
gives a $\Theta$-trivial slice $Z_0 \to \bun_{L, rig}^{ss, \mu}$ with $Z_0 = E$, such that $Z_0 \to \bun_{L, rig}^{ss, \mu}/E$ is surjective with fibres isomorphic to $Z(L)_{rig}$, hence connected. So \eqref{itm:subregularsliceexistence1}, \eqref{itm:subregularsliceexistence2} and \eqref{itm:subregularsliceexistence3} are satisfied. A simple root-theoretic calculation shows that $-\langle 2\rho, \mu \rangle = l + 2$, so \eqref{itm:subregularsliceexistence4} follows from Proposition \ref{prop:inductionaffinebundle}. So this proves the theorem in this case.
\end{proof}

The construction in the exceptional types $E$, $F$ and $G$ is also fairly straightforward.

\begin{proof}[Proof of Theorem \ref{thm:subregularsliceexistence} in types $E$, $F$ and $G$]
In these cases, Proposition \ref{prop:typealevi} and Atiyah's theorem show that the morphism
\begin{equation} \label{eq:typeefgsubregularresolutions1}
\varpi_i \colon \bun_{L, rig}^{ss, \mu} \longrightarrow \mrm{Pic}^{-1}_S(E)
\end{equation}
is a $\mb{G}_m = Z(L)_{rig}$-gerbe. Let $Z_0 = S$, and let $\mf{G}'$ be the $Z(L)_{rig}$-gerbe given by the pullback along $\mc{O}(-O_E) \colon Z_0 \to \mrm{Pic}^{-1}_S(E)$. By \cite[Proposition 4.1.15]{davis19}, the pullback of the theta bundle defines a $\B Z(L)_{rig}$-equivariant morphism $\mf{G}' \to \B \mb{G}_m$, where $\B Z(L)_{rig}$ acts on $\B \mb{G}_m$ through the homomorphism
\[ -(\mu \mmid - ) \colon Z(L)_{rig} \longrightarrow \mb{G}_m.\]
So a section of $\mf{G}'$ such that the pullback of $\Theta_{\bun_{G, rig}}$ is trivial is the same thing as a section of the $\mu_d = \ker(\mu \mmid - )$-gerbe $\mf{G} = \mf{G}' \times_{\B \mb{G}_m} \spec \mb{Z}$. The $\mu_d$-gerbe is by construction pulled back from one $\mf{G}^{uni}$ on $M_{1, 1}$, defined in the same way, and if it is trivial then there is a morphism $Z_0 \to \bun_{L, rig}^{ss, \mu}$ lifting the section $\mc{O}(-O_E) \colon Z_0 \to \mrm{Pic}^{-1}_S(E)$ such that the pullback of $\Theta_{\bun_{G, rig}}$ is trivial.

It is immediately clear that \eqref{itm:subregularsliceexistence1} is satisfied. Letting $(\bun_{L, rig}^{ss, \mu})_0$ be the fibre of \eqref{eq:typeefgsubregularresolutions1} over $\mc{O}(-O_E) \colon S \to \mrm{Pic}^{-1}_S(E)$, we have that $(\bun_{L, rig}^{ss, \mu})_0 \cong \bun_{L, rig}^{ss, \mu}/E$ is a $Z(L)_{rig}$-gerbe over $S = Z_0$ and the map $Z_0 \to \bun_{L, rig}^{ss, \mu}/E$ is a section. In particular, it is smooth with connected fibres (isomorphic to $Z(L)_{rig}$), so \eqref{itm:subregularsliceexistence2} is satisfied, and surjective, so \eqref{itm:subregularsliceexistence3} is also satisfied. Finally, to prove \eqref{itm:subregularsliceexistence4}, simply note that Proposition \ref{prop:inductionaffinebundle} and a simple root-theoretic calculation show that $Z = \mrm{Ind}_L^G(Z_0) \to Z_0 = S$ is an affine space bundle of relative dimension $l + 3$.
\end{proof}

The proof of Theorem \ref{thm:subregularsliceexistence} in types $B$, $C$ and $D$ will require a few more preliminaries. First, we remark on the following realisation of $GSp_4$-bundles in terms of vector bundles.

\begin{defn}
A \emph{conformally symplectic vector bundle} is a tuple $(W, M, \omega)$, where $W$ is a vector bundle, $M$ is a line bundle, and $\omega \colon \wedge^2 W \to M$ is a morphism such that the induced morphism $W \to W^\vee \otimes M$ is an isomorphism.
\end{defn}

\begin{lem} \label{lem:gsp4bundles}
There is an isomorphism of $\bun_{GSp_4}$ with the relative stack of conformally symplectic vector bundles $(W, M, \omega)$ on $E$ over $S$, which identifies $\chi \colon \bun_{GSp_4} \to \bun_{\mb{G}_m}$ with the map $(W, M, \omega) \mapsto M$.
\end{lem}
\begin{proof}
Let $V$ be the standard representation of $GSp_4$ coming from the inclusion $GSp_4 \subseteq GL_4$. Then $J$ defines a homomorphism of $GSp_4$-representations $J \colon \wedge^2 V \to \mb{Z}_\chi$. The isomorphism from $\bun_{GSp_4}$ to the stack of conformally symplectic vector bundles sends a $GSp_4$-bundle $\xi$ to $(\xi \times^{GSp_4} V, \xi \times^{GSp_4} \mb{Z}_\chi, \xi \times^{GSp_4} J)$.
\end{proof}

Now assume that $(G, P, \mu)$ is a subregular Harder-Narasimhan class of type $B$, $C$ or $D$ with corresponding Levi $L \subseteq P$. Let $P' \subseteq L$ denote the standard parabolic of type $t(P') = \{\alpha_l\}$, and $L' \subseteq P'$ its standard Levi subgroup. In types $C$ and $D$, let $\rho_L \colon L \to GL_{n_1 + 1}$ be the composition of the isomorphism of Lemma \ref{lem:typecdefleviiso} with the projection to the second factor (where for concreteness we choose the labelling so that $\alpha_{c_1, n_1} = \alpha_l$), and in type $B$ let $\rho_L \colon L \to GL_4$ be the composition of the isomorphism of Lemma \ref{lem:typebleviiso} with the projection to the second factor and the inclusion $GSp_4 \subseteq GL_4$.

\begin{lem} \label{lem:typebcdparabolicbundles}
Assume we are in types $B$, $C$ or $D$. Then there is an isomorphism of $\bun_{P'}$ with the stack of pairs $(\xi_L, M \subseteq W)$, where $\xi_L \in \bun_L$ and $M \subseteq W$ is a line subbundle of the vector bundle $W$ associated to $\xi_L$ under the representation $\rho_L$, such that the morphism
\[ \varpi_l \colon \bun_{P'} \longrightarrow \bun_{\mb{G}_m} \]
is identified with the morphism
\[ (\xi_L, M \subseteq W) \longmapsto \begin{cases} \varpi_i(\xi_L) \otimes M^{-1}, & \text{in types }B\text{ and }D, \\ \varpi_i(\xi_L)^{\otimes 2} \otimes M^{-1}, & \text{in type }C. \end{cases}\]
In types $C$ and $D$ (resp., type $B$), if $\xi_{P'}$ corresponds to $(\xi_L, M \subseteq W)$ and $V$ is the vector bundle induced by $\xi_L$ under the projection $L \to GL_{n_0}$ coming from Lemma \ref{lem:typecdefleviiso} (resp., \ref{lem:typebleviiso}), then the $L'$-bundle $\xi_{P'} \times^{P'} L'$ is semistable if and only if the vector bundles $V$ and $W/M$ (resp., $\ker(\omega\colon W/M \to \det V \otimes M^\vee)$) are semistable.
\end{lem}
\begin{proof}
In types $C$ and $D$, the isomorphism of Lemma \ref{lem:typecdefleviiso} identifies $P'$ with the parabolic $GL_{n_0} \times R_{n_1 + 1}$, where $R_{n_1 + 1} \subseteq GL_{n_1 + 1}$ is the maximal parabolic of type $\{\beta_{n_1}\}$, and the result follows routinely. In type $B$, using Lemma \ref{lem:typebleviiso} we have an $L$-equivariant identification $L/P' \cong GSp_4/(GSp_4 \cap R_4) \cong GL_4/R_4 \cong \mb{P}^4$ with the space of lines in the representation $\rho_L$, where
\[ R_4 = \left\{\left(\begin{matrix} * & * & * & 0 \\ * & * & * & 0 \\ * & * & * & 0 \\ * & * & * & * \end{matrix}\right)\right\} \subseteq GL_4.\]
The claimed isomorphism in this case now follows. To get the desired identification of the semistable bundles, notice that the Levi factor of $GSp_4 \cap R_4$ is
\[ \left\{\left.\left(\begin{array}{c|cc|c} \lambda^{-1}\det A & 0 & 0 & 0 \\ \hline 0 & & & 0 \\ 0 & \multicolumn{2}{c|}{\smash{\raisebox{.5\normalbaselineskip}{$A$}}} & 0 \\ \hline \\[-\normalbaselineskip] 0 & 0 & 0 & \lambda \end{array}\right)\right| A \in GL_2, \lambda \in \mb{G}_m \right\} \cong GL_2 \times \mb{G}_m,\]
so we have an isomorphism
\[ \bun_{L'} \cong \bun_{GL_{n_0}} \times_{\bun_{\mb{G}_m}} \bun_{GL_2} \times_S \bun_{\mb{G}_m},\]
such that the map $\bun_{P'} \to \bun_{L'}$ is identified with
\[ (\xi_L, M \subseteq W) \longmapsto (V, \ker(W/M \to \det V \otimes M^\vee), M). \]
This now implies the claim.
\end{proof}

\begin{lem} \label{lem:typebcdleviautos}
Let $(G, P, \mu)$ be of type $B$, $C$ or $D$, and assume that $\xi_L \to E_s$ is a semistable $L$-bundle of slope $\mu$ over a geometric fibre of $E \to S$. Then $\dim \mrm{Aut}(\xi_L) \geq 2$.
\end{lem}
\begin{proof}
By Lemmas \ref{lem:typecdefleviiso}, \ref{lem:typebleviiso} and \ref{lem:gsp4bundles} and \cite[Theorem 4.2.6]{davis19}, it suffices to show that if $W$ is a semistable vector bundle of degree $-2$ and rank $2r$ (resp., $(W, M, \omega)$ is a conformally symplectic vector bundle with $W$ semistable and $\deg M = -1$), then $\dim \mrm{Aut}(W) \geq 2$ (resp., $\dim \mrm{Aut}(W, M, \omega) \geq 2$).

In the first case, observe that if $U$ and $U'$ are nonisomorphic semistable vector bundles of degree $-1$ and rank $r$, then $U \otimes (U')^\vee$ is a vector bundle of degree $0$ with $H^0(E, U \otimes (U')^\vee) = 0$, and hence $H^1(E, U \otimes (U')^\vee) = 0$ also. It follows that the morphism
\begin{align*}
\bun_{GL_r}^{ss, -1} \times \bun_{GL_r}^{ss, -1} &\longrightarrow \bun_{GL_{2r}}^{ss, -2} \\
(U, U') &\longmapsto U \oplus U'
\end{align*}
is \'etale at $(U, U')$ if $U \not\cong U'$. Since the locus of vector bundles $W$ in $\bun_{GL_{2d}}^{ss, -2}$ with $\dim \mrm{Aut}(W) < 2$ is open, it is either empty or dense. So by openness of \'etale morphisms, if it is nonempty, then there exists such a bundle $W = U \oplus U'$ with $U \not\cong U'$. But $\mrm{Aut}(W) = \mrm{Aut}(U) \times \mrm{Aut}(U') = \mb{G}_m \times \mb{G}_m$ for such bundles, so this is a contradiction and we are done in this case.

The proof for conformally symplectic bundles is similar. Consider the Levi subgroup
\[ GL_2 \times \mb{G}_m \cong L'' = \left\{\left.\left(\begin{array}{c|c} \lambda J_0 (A^t)^{-1} J_0 & 0 \\ \hline 0 & \lambda A \end{array}\right)\right| A \in GL_2, \lambda \in \mb{G}_m \right\},\]
where
\[ J_0 = \left(\begin{matrix} 0 & 1 \\ 1 & 0 \end{matrix}\right).\]
Given $(U, M) \in \bun_{GL_2}^{ss, -1} \times_S \bun_{\mb{G}_m}^{-1}$ corresponding to an $L''$-bundle $\xi_{L''}$, with $U \not\cong U^\vee \otimes M$, we have that
\[ \xi_{L''} \times^{L''} \mf{gsp}_4/\mf{l}'' \subseteq U^\vee \otimes (U^\vee \otimes M) \oplus U \otimes (U^\vee \otimes M)^\vee \]
is a degree $0$ vector bundle on $E_s$ with $H^0(E_s, \xi_{L''} \times^{L''} \mf{gsp}_4/\mf{l}'') = 0$ and hence $H^1(E_s, \xi_{L''} \times^{L''} \mf{gsp}_4/\mf{l}'') = 0$ also, where $\mf{gsp}_4 = \mrm{Lie}(GSp_4)$ and $\mf{l}'' = \mrm{Lie}(L'')$. So we conclude that the morphism
\[ \bun_{L''} \longrightarrow \bun_{GSp_4}^{-1} \]
is \'etale at $(U, M)$.

Since the locus of conformally symplectic vector bundles $(W, M, \omega)$ in $\bun_{GSp_4}^{ss, -1}$ with automorphism group of dimension $< 2$ is open, it is either empty or dense. If it is nonempty, then by openness of \'etale morphisms we can find such a bundle of the form $W = U \oplus U^\vee \otimes M$ as above. But $\dim \mrm{Aut}_{GSp_4}(W) = \dim \mrm{Aut}(U) \times \dim \mrm{Aut}(M) = 2$, so this is a contradiction and the lemma is proved.
\end{proof}

\begin{proof}[Proof of Theorem \ref{thm:subregularsliceexistence} in types $B$, $C$ and $D$]
Let $\mu' \in \mb{X}_*(Z(L'))_\mb{Q}$ be the unique vector with $\langle \varpi_i, \mu' \rangle = -1$ and $\langle \varpi_l, \mu' \rangle = 0$. Then Proposition \ref{prop:typealevi} and Atiyah's classification show that the morphism
\[ (\varpi_i, \varpi_l) \colon \bun_{L', rig}^{ss, \mu'} \longrightarrow \mrm{Pic}^{-1}_S(E) \times_S \mrm{Pic}^0_S(E) \]
is a $Z(L')_{rig}$-gerbe. Let $\mf{G}''$ be the $Z(L')_{rig}$-gerbe on $S$ given by pulling back along the section
\[ (\mc{O}(-O_E), \mc{O}) \colon S \to \mrm{Pic}^{-1}_S(E) \times_S \mrm{Pic}^0_S(E).\]
The pullback of the theta bundle gives a $\B Z(L')_{rig}$-equivariant morphism $\mf{G}'' \to \B \mb{G}_m$ where $\B Z(L')_{rig}$ acts through the homomorphism $(-\mu'\mmid -)$, so we get a $\ker(\mu' \mmid -)$-gerbe $\mf{G}' = \mf{G}'' \times_{\B \mb{G}_m} \spec \mb{Z}$. Let $\mf{G}$ be the rigidification of $\mf{G}'$ with respect to $\varpi_l^\vee(\mb{G}_m)$. Then $\mf{G}$ is a $\ker(\mu' \mmid -)/\varpi_l^\vee(\mb{G}_m) \cong \mu_d$-gerbe, pulled back from a gerbe $\mf{G}^{uni}$ on $M_{1, 1}$, and if it is trivial then we have a $\B \mb{G}_m$-equivariant morphism $\B_S\mb{G}_m \to \bun_{L', rig}^{ss, \mu'}$ (with $\B \mb{G}_m$ acting through $\varpi_l^\vee$) lifting the section $(\mc{O}(-O_E), \mc{O})$ such that the pullback of the theta bundle is trivial. Define
\[ Z_0 = \mrm{Ind}_{L'}^L(\B_S \mb{G}_m) \setminus \B_S \mb{G}_m \longrightarrow \bun_{L, rig}^\mu,\]
and observe that the pullback of $\Theta_{\bun_{G, rig}}$ to $Z_0$ is also trivial since $Z_0 \to \B_S \mb{G}_m$ is an affine space bundle.

\begin{table}[h]
\centering
\begin{tabular}{c|c|>{\rule{0pt}{2.6ex}}c|c}
Type & $\alpha \in \Phi_L$ with $\langle \alpha, \mu' \rangle < 0$ & $\langle \alpha, \mu' \rangle$ & $\langle \alpha, \varpi_l^\vee \rangle$ \\ \hline
\multirow{3}*{$B$} & $-\alpha_l$ & $-\frac{1}{2}$ & $-1$ \\
& $-\alpha_{l - 1} - \alpha_l$ & $-\frac{1}{2}$ & $-1$ \\
& $-\alpha_{l - 1} - 2\alpha_l$ & $-1$ & $-2$ \\ \hline
$C$ & $-\alpha_l$ & $-2$ & $-1$ \\ \hline
\multirow{3}*{$D$} & $-\alpha_l$ & $-\frac{2}{3}$ & $-1$ \\
& $-\alpha_{l - 2} - \alpha_l$ & $-\frac{2}{3}$ & $-1$ \\
& $-\alpha_{l - 2} - \alpha_{l - 1} - \alpha_l$ & $-\frac{2}{3}$ & $-1$
\end{tabular}
\vspace{1ex}
\caption{Roots of $L$ with $\langle \alpha, \mu' \rangle < 0$}
\label{tab:typebcdroots}
\end{table}

We now check that $Z_0$ satisfies the conditions of Theorem \ref{thm:subregularsliceexistence}. Since the claims are local on $S$, we can assume for convenience that the section $\B_S\mb{G}_m \to \bun_{L', rig}^{ss, \mu'}$ lifts to a morphism $S \to \bun_{L'}^{ss, \mu'}$ and that the line bundle on $E$ associated to this section via the character $\varpi_l$ is trivial. Note that in this case, we have a natural identification
\[ Z_0 \cong (\mrm{Ind}_{L'}^L(S) \setminus S)/\mb{G}_m.\]

First, the roots $\alpha \in \Phi_L$ with $\langle \alpha, \mu' \rangle < 0$ are given in Table \ref{tab:typebcdroots}, along with the values of $\langle \alpha, \mu' \rangle$ and $\langle \alpha, \varpi_l^\vee \rangle$. Using Proposition \ref{prop:inductionaffinebundle} and \cite[Proposition 5.2.7]{davis19a}, it follows that $\mrm{Ind}_{L'}^L(S) \to S$ is an $\mb{A}^2$-bundle on which $\mb{G}_m$ acts with weight $1$ in types $C$ and $D$, and weights $1$ and $2$ in type $B$. So $Z_0 \to S$ is a $\mb{P}(1, 2)$-bundle in type $B$ and a $\mb{P}^1$-bundle in types $C$ and $D$. In particular, \eqref{itm:subregularsliceexistence1} is satisfied.

We next show that $Z_0 \to \bun_{L, rig}^\mu$ factors through $\bun_{L, rig}^{ss, \mu}$. Note that Table \ref{tab:typebcdroots} shows that $-\mu'$ is a Harder-Narasimhan vector for $P' \subseteq L$, so $\mrm{Ind}_{L'}^L(S) = \bun_{P', rig}^{ss, \mu'} \times_{\bun_{L', rig}^{ss, \mu'}} S$. So Lemma \ref{lem:typebcdparabolicbundles} shows that $\xi_L$ is in the image of $\mrm{Ind}_{L'}^L(S)$ if and only if $V$ is semistable of determinant $\mc{O}(-O_E)$ and there exists a nonvanishing section of $W \otimes \mc{O}(d O_E)$ such that the vector bundle
\[ U = \begin{cases} W/\mc{O}(-d O_E), & \text{in types }C\text{ and }D, \\ \ker(W/\mc{O}(-(d O_E) \to \mc{O}), & \text{in type }B, \end{cases} \]
is semistable. Here $V$ and $W$ are as in the statement of Lemma \ref{lem:typebcdparabolicbundles}, and
\[ d = \begin{cases} 1, & \text{in types}\; B \;\text{and}\; D, \\ 2, & \text{in type}\;C\end{cases} \]
is as in the statement of the theorem. The bundle $\xi_L$ is in the image of $\mrm{Ind}_{L'}^L(S) \setminus S$ if and only if $\mc{O}(-d O_E) \to W$ can be chosen not to admit a retraction. Suppose that $\xi_L$ is such a bundle and that $\xi_L$ is unstable; we deduce a contradiction in each type.

In type $B$, $W$ is an unstable conformally symplectic vector bundle of rank $4$ and degree $-2$, so there exists a quotient $W \to N$ where $N$ has slope $< -1/2$. Replacing $N$ with $\coker(N^\vee \otimes \mc{O}(-O_E) \to W)$ if necessary, we can assume that $N$ has rank $\leq 2$. Since any vector bundle of rank $2$ and slope $<-1/2$ surjects onto some line bundle of negative degree, we can therefore assume without loss of generality that $N$ is a line bundle. Examining slopes, we see from semistability of $U$ that $W \to N$ does not factor through $W/\mc{O}(-O_E)$, and hence that $\mc{O}(-O_E) \to N$ is nonzero. So $\mc{O}(-O_E) \to N$ must be an isomorphism since $\deg N \leq \deg \mc{O}(-O_E)$, and we therefore have a retraction $W \to \mc{O}(-O_E) = N$. Since this is a contradiction, we are done in this case.

In type $C$, $W$ is an unstable vector bundle of rank $2$ and degree $-2$, so there exists a quotient $W \to N$ where $N$ is a a line bundle of degree $< -1$. Examining slopes, we see that $W \to N$ does not factor through $W/\mc{O}(-2O_E)$ and hence that $\mc{O}(-2O_E) \to N$ is nonzero. So $\mc{O}(-2O_E) \to N$ must be an isomorphism since $\deg N \leq \deg \mc{O}(-2O_E)$, and we therefore have a retraction $W \to \mc{O}(-2O_E) = N$. Since this is a contradiction, we are done in this case as well.

Finally, in type $D$, $W$ is an unstable vector bundle of rank $4$ and degree $-2$, so there exists a quotient $W \to N$ where $N$ is a semistable vector bundle of slope $< -1/2$. Examining slopes and using semistability of $W/\mc{O}(-O_E)$ and of $N$, we see that $W \to N$ does not factor through $W/\mc{O}(-O_E)$ and we again get a retraction $W \to N \cong \mc{O}(-O_E)$. So we have shown that $\xi_L$ must be semistable in all cases.

We next show that the morphism $Z_0 \to \bun_{L, rig}^{ss, \mu}/E$ is smooth with connected fibres, which proves \eqref{itm:subregularsliceexistence2} and that $Z_0 \to \bun_{L, rig}^{ss, \mu}$ is a $\Theta$-trivial slice. Write $(\bun_L^{ss, \mu})_0$ for the fibre of $\varpi_i \colon \bun_L^{ss, \mu} \to \mrm{Pic}^{-1}_S(E)$ over $\mc{O}(-O_E)$ and $(\bun_{P'}^{\mu'})^{ss}_0 = \bun_{P'}^{\mu'} \times_{\bun_{L}^{\mu}} (\bun_{L}^{ss, \mu})_0$. Then Lemma \ref{lem:typebcdparabolicbundles} gives an open immersion
\[ (\bun_{P'}^{\mu'})_0^{ss} \subseteq \mb{P}_{(\bun_L^{ss, \mu})_0} \pi_*(W^{uni} \otimes \mc{O}(d O_E)),\]
where we write $W^{uni}$ for the universal bundle on $(\bun_L^{ss, \mu})_0 \times_S E$ induced by the representation $\rho_L$ of $L$ and $\pi \colon (\bun_L^{ss, \mu})_0 \times_S E \to (\bun_L^{ss, \mu})_0$ for the natural projection. Moreover,
\[ Z_0 \times_{(\bun_{L, rig}^{ss, \mu})_0} (\bun_L^{ss, \mu})_0 \longrightarrow (\bun_{P'}^{\mu'})^{ss}_0 \]
is a $\mb{G}_m = Z(L)_{rig}/\varpi_l^\vee(\mb{G}_m)$-torsor over the open substack where the associated $L'$-bundle is semistable. This shows in particular that $Z_0 \times_{(\bun_{L, rig}^{ss, \mu})_0} (\bun_L^{ss, \mu})_0 \to (\bun_L^{ss, \mu})_0$ is smooth with connected fibres of dimension $2$, and hence that the same is true for $Z_0 \to (\bun_{L, rig}^{ss, \mu})_0 \cong \bun_{L, rig}^{ss, \mu}/E$ as claimed.

To prove \eqref{itm:subregularsliceexistence3}, first observe that since $Z_0 \to S$ has finite relative stabilisers, any $L$-bundle in the image of $Z_0 \to (\bun_{L, rig}^{ss, \mu})_0 \subseteq \bun_{L, rig}^{ss, \mu}$ can have automorphism group of dimension at most $2$, and is hence regular by Lemma \ref{lem:typebcdleviautos}. For the converse, note that since every regular semistable $L$-bundle is a translate of one in $(\bun_L^{ss, \mu})_0$, it suffices to show that any regular semistable bundle in $(\bun_L^{ss, \mu})_0$ is in the image of $(\bun_{P'}^{ss, \mu})_0 \to \bun_L^\mu$, and hence in the image of $Z_0 \to \bun_{L, rig}^{ss, \mu}$.

Suppose then that $\xi_L \to E_s$ is a semistable $L$-bundle in $(\bun_L^{ss, \mu})_0$ over $s \colon \spec k \to S$ that is not in the image of $(\bun_{P'}^{ss, \mu'})_0$. We show in each type that $\dim \mrm{Aut}(\xi_L) > 2$ so $\xi_L$ is not regular.

In type $B$, in the notation of Lemma \ref{lem:typebcdparabolicbundles}, we have that for every nonzero morphism $\gamma \colon \mc{O}(-O_E) \to W$, the vector bundle $U_\gamma = \ker(W/\mc{O}(-O_E) \to \mc{O})$ is unstable. (Note that $W$ is semistable of rank $4$ and degree $-2$, so any such morphism is a subbundle.) Using semistability of $W$, the Harder-Narasimhan decomposition of $U_\gamma$ must be of the form $U_\gamma = N_\gamma \oplus N_\gamma^\vee \otimes \mc{O}(-O_E)$, where $N_\gamma$ is a line bundle of degree $0$ on $E_s$ and the preimage of $N_\gamma$ in $W$ is the unique non-split extension $N_\gamma'$ of $N_\gamma$ by $\mc{O}(-O_E)$. By uniqueness of Harder-Narasimhan filtrations, it follows that we have a morphism $\mb{P}^1_k = \mb{P}H^0(E_s, W \otimes \mc{O}(O_E)) \to \mrm{Pic}^0(E_s)$ sending $\gamma$ to the isomorphism class of $N_\gamma$. Since there are no non-constant morphisms from $\mb{P}^1_k$ to any elliptic curve over $k$, we deduce that $N_\gamma = N$ and $N_\gamma' = N'$ are independent of $\gamma$. So every nonzero morphism $\mc{O}(-O_E) \to W$ factors through some Lagrangian inclusion $N' \hookrightarrow W$. Choosing any such morphism gives an exact sequence
\[ 0 \longrightarrow N' \longrightarrow W \longrightarrow (N')^\vee \otimes \mc{O}(-O_E) \longrightarrow 0.\]
Since $\dim \hom(\mc{O}(-O_E), N') = 1$ and $\dim \hom(\mc{O}(-O_E), W) = 2$, we can choose another homomorphism $\mc{O}(-O_E)$ not factoring through the given copy of $N'$, and hence get another Lagrangian inclusion $N' \hookrightarrow W$, which must map $N'$ isomorphically onto $(N')^\vee \otimes \mc{O}(-O_E)$. So the above exact sequence splits, and we have
\[ W \cong N' \oplus N',\]
where both summands are Lagrangian. In particular, $W$ and hence $\xi_L$ carries a faithful action of $Sp_2$, so $\dim \mrm{Aut}(\xi_L) > 2$ as claimed.

In type $C$, we have that every nonzero morphism $\gamma \colon \mc{O}(-2O_E) \to W$ must vanish at some unique point $x_\gamma \in E_s$. So again we have a morphism $\mb{P}^1_k = \mb{P}H^0(E_s, W \otimes \mc{O}(-2O_E)) \to E_s$ sending $\gamma$ to $x_\gamma$, which must be constant. So $x_\gamma = x$ is independent of $\gamma$, and every morphism $\mc{O}(-2O_E) \to W$ therefore factors through a subbundle $\mc{O}(x - 2O_E) \subseteq W$. Since $W$ is semistable of trivial determinant, choosing any two linearly independent morphisms gives an isomorphism $W \cong \mc{O}(x - 2O_E) \oplus \mc{O}(x - 2O_E)$. So $SL_2$ acts faithfully on $W$ and hence on $\xi_L$ and $\dim \mrm{Aut}(\xi_L) > 2$ as claimed.

In type $D$, we have that $U_\gamma = W/\mc{O}(-O_E)$ is unstable for every nonzero morphism $\gamma \colon \mc{O}(-O_E) \to W$. (Note that again any such $\gamma$ must be a subbundle since $W$ is semistable of slope $-1/2$.) Since $W$ is semistable, one sees that the Harder-Narasimhan decomposition of $U_\gamma$ must be of the form $U_\gamma = N_\gamma \oplus \det(N_\gamma)^\vee \otimes \mc{O}(-OE)$, where $N_\gamma$ is a rank $2$ semistable vector bundle of degree $-1$. Again we get a morphism $\mb{P}^1_k = \mb{P}H^0(E_s, W \otimes \mc{O}(O_E)) \to \mrm{Pic}^{-1}(E_s)$ sending $\gamma$ to the isomorphism class of $\det(N_\gamma)$, which again must be constant. So $\det(N_\gamma)$, and hence $N_\gamma = N$ are independent of $\gamma$, and every nonzero morphism $\mc{O}(-O_E) \to W$ factors through the kernel of some surjection $W \to N$. Choosing two linearly independent morphisms $\mc{O}(-O_E) \to W$ therefore gives a map $W \to N \oplus N$, which one easily sees must be an isomorphism. So again $SL_2$ acts faithfully on $W$ fixing the determinant, and hence on $\xi_L$, which proves that $\dim \mrm{Aut}(\xi_L) > 2$ in this case as well.

Finally, to prove \eqref{itm:subregularsliceexistence4}, simply note that Proposition \ref{prop:inductionaffinebundle} implies that $Z \to Z_0$ is an affine space bundle of relative dimension $- \langle 2 \rho, \mu \rangle = l + 2$, so $Z \to S$ has relative dimension $l + 3$ as required.
\end{proof}

\section{Computing resolutions} \label{section:resolutions}

The purpose of this section is to give the proof of Theorem \ref{thm:introsubregularresolutions}. We prove \eqref{itm:introsubregularresolutions1}, \eqref{itm:introsubregularresolutions2}, \eqref{itm:introsubregularresolutions3} and \eqref{itm:introsubregularresolutions4} separately (as Propositions \ref{prop:subregularresolutions1}, \ref{prop:subregularresolutions2}, \ref{prop:subregularresolutions3} and \ref{prop:subregularresolutions4}) in \S\ref{subsection:decomposition}, \S\ref{subsection:dalphaj}, \S\ref{subsection:dalphaij} and \S\ref{subsection:dalphai} respectively.

The proofs of Propositions \ref{prop:subregularresolutions2} and \ref{prop:subregularresolutions4} make use of the idea that sections of a flag variety bundle decompose naturally according to which Bruhat cells they meet. This is used to give decompositions of the divisor $D_{\alpha_i^\vee}(Z)$ and $D_{\alpha_j^\vee}(Z)$ into locally closed subsets, each of which can be identified in terms of an analogous set of sections of a flag variety for some copy of $GL_n$ inside $G$. We manage to show that these ``Bruhat cells'' fit together into the blowups in Theorem \ref{thm:introsubregularresolutions} by explicitly constructing the blow downs as spaces of (stable) sections of \emph{partial} flag variety bundles. The Bruhat cells are discussed in general in \S\ref{subsection:bruhat}, and the specific cells of interest for $GL_n$ are studied in \S\ref{subsection:glnbruhat}.

\subsection{Decomposition of $\tilde{\chi}_Z^{-1}(0_{\Theta_Y^{-1}})$} \label{subsection:decomposition}

In this subsection, we prove the following proposition, which is a slightly more general version of part \eqref{itm:introsubregularresolutions1} of Theorem \ref{thm:introsubregularresolutions}.

\begin{prop} \label{prop:subregularresolutions1}
Let $(G, P, \mu)$ be a subregular Harder-Narasimhan class not of type $A_1$. Assume that the $\mu_d$-gerbe $\mf{G}$ of Theorem \ref{thm:subregularsliceexistence} is trivial, let $Z_0 \to \bun_{L, rig}^{ss, \mu}$ be the corresponding $\Theta$-trivial slice, and let $Z = \mrm{Ind}_L^G(Z_0) \to \bun_{G, rig}$ be the induced equivariant slice. Then the preimage of the zero section of $\Theta_Y^{-1}$ in $\tilde{Z} = \tbun_{G, rig} \times_{\bun_G} Z$ decomposes as a divisor with normal crossings
\begin{equation} \label{eq:subregularresolutions1:1}
 \tilde{\chi}_Z^{-1}(0_{\Theta_Y^{-1}}) = dD_{\alpha_i^\vee}(Z) + D_{\alpha_j^\vee}(Z) + D_{\alpha_i^\vee + \alpha_j^\vee}(Z)
\end{equation}
such that each summand is smooth over $Y$, where $D_\lambda(Z)$ denotes the closure of the locus of stable maps with a single rational component of degree $\lambda \in \mb{X}_*(T)$.
\end{prop}
\begin{proof}
Since $Z \to \bun_{G, rig}$ is a slice, \cite[Proposition 2.1.10 and Corollary 3.3.8]{davis19} imply that the preimage of the zero section decomposes as a divisor with normal crossings
\[ \tilde{\chi}_Z^{-1}(0_{\Theta_Y^{-1}}) = \sum_{\lambda \in \mb{X}_*(T)_+} \frac{1}{2}(\lambda \mmid \lambda) D_\lambda(Z),\]
where $(\,\mmid\,)$ is the normalised Killing form on $\mb{X}_*(T)$. By Lemma \ref{lem:subregularemptydivisors} below, $D_\lambda(Z) = \emptyset$ unless $\lambda \in \{\alpha_i^\vee, \alpha_j^\vee, \alpha_i^\vee + \alpha_j^\vee\}$, so this simplifies to \eqref{eq:subregularresolutions1:1} as required, since $\alpha_j^\vee$ and $\alpha_i^\vee + \alpha_j^\vee$ are short coroots and $(\alpha_i^\vee \mmid \alpha_i^\vee) = 2d$ in each case.

It remains to show that each $D_\lambda(Z)$ is smooth over $Y$. Note that $\tilde{\chi}_Z^{-1}(0_{\Theta_Y^{-1}})$ is in fact a divisor with normal crossings relative to $Y$: this follows from \cite[Proposition 3.5.3]{davis19a}, the definition \cite[Definition 2.1.14]{davis19} of the blow down morphism $\tbun_G \to Y$, and the fact that the boundary of the stack $\Deg_S(E)$ of prestable degenerations of $E$ is a divisor with normal crossings relative to $S$ \cite[Proposition 2.1.7]{davis19}. So it is enough to show that each $D_\lambda(Z)$ has no self-intersections. But a point in such a self-intersection would have to be given by a stable map with at least two rational components both of degree $\geq \alpha_i^\vee$ or both of degree $\geq \alpha_j^\vee$. But this is forbidden by Lemma \ref{lem:subregularemptydivisors}, so we are done.
\end{proof}

\begin{lem} \label{lem:subregularemptydivisors}
For $\lambda \in \mb{X}_*(T)_+$, we have $D_\lambda(Z) \neq \emptyset$ if and only if $\lambda \in \{\alpha_i^\vee, \alpha_j^\vee, \alpha_i^\vee + \alpha_j^\vee\}$.
\end{lem}
\begin{proof}
For simplicity, we can assume without loss of generality that $S = \spec k$ for $k$ an algebraically closed field. We first show that $D_{\alpha_i^\vee}(Z) \neq \emptyset$ and $D_{\alpha_j^\vee}(Z) \neq \emptyset$.

If $(G, P, \mu)$ is of type $A$, then $\mu$ is the image of $-\alpha_i^\vee - \alpha_j^\vee$ under the homomorphism $\mb{X}_*(T) \to \mb{X}_*(Z(L)^\circ)_{\mb{Q}}$ and $\langle \alpha, \alpha_i^\vee + \alpha_j^\vee \rangle \leq 0$ for all $\alpha \in \Phi_+$ a root of $P$. So by \cite[Proposition 3.6.4]{davis19a}, the morphism
\[ \mrm{KM}_{B, G}^{-\alpha_i^\vee - \alpha_j^\vee} \longrightarrow \mrm{KM}_{P, G}^\mu \]
is surjective. In particular, for every $z \in Z_0$, there exists a section of $\xi_{L, z} \times^L P/B \subseteq \xi_{L, z} \times^L G/B$ with degree $-\lambda_0 \leq -\alpha_i^\vee - \alpha_j^\vee$. So we must have $D_{\lambda_0}(Z) \neq \emptyset$, and hence $D_{\alpha_i^\vee}(Z) \neq \emptyset$ and $D_{\alpha_j^\vee}(Z) \neq \emptyset$, since we can always add rational tails to such a section to produce a stable map in each of these divisors.

On the other hand, if $(G, P, \mu)$ is not of type $A$, then $\mu$ is the image of $-\alpha_i^\vee$ in $\mb{X}_*(Z(L)^\circ)_\mb{Q}$, and $\langle \alpha, \alpha_i^\vee \rangle \leq 0$ for $\alpha \in \Phi_+$ a root of $P$. So
\[ \mrm{KM}_{B, G}^{-\alpha_i^\vee} \longrightarrow \mrm{KM}_{P, G}^\mu \]
is surjective by \cite[Proposition 3.6.4]{davis19a}, so we deduce that $D_{\alpha_i^\vee}(Z) \neq \emptyset$. For $D_{\alpha_j^\vee}(Z)$, note that since $\alpha_j \in \Delta$ is the unique special root, \cite[Proposition 4.2.3]{davis19} implies that the Harder-Narasimhan locus $\bun_Q^{ss, -\alpha_j^\vee} \subseteq \bun_G$ is dense in the locus of unstable $G$-bundles, where $Q$ is the standard parabolic with $t(Q) = \{\alpha_j^\vee\}$. So $\bun_{Q, rig}^{ss, -\alpha_j^\vee} \times_{\bun_{G, rig}} Z \neq \emptyset$, and hence $D_{\alpha_j^\vee}(Z) \neq \emptyset$ by \cite[Proposition 4.3.8]{davis19}.

Conversely, suppose that $\lambda \in \mb{X}_*(T)$ and that $D_\lambda(Z) \neq \emptyset$. Then for any $\alpha_k \in \Delta$ with corresponding maximal parabolic $P_k$, there exists a point in $Z$ and a section of the corresponding $G/P_k$-bundle with degree $\nu_k = -\langle \varpi_k, \lambda \rangle/\langle \varpi_k, \varpi_k^\vee \rangle \varpi_k^\vee$ (the image of $\lambda$ in $\mb{X}_*(T_{P_k})$). So by Lemma \ref{lem:inducedsectionbound1} and \cite[Lemma 3.3.2]{friedman-morgan00}, we must have
\[ (l + 1)\langle \varpi_k, \lambda \rangle \leq \frac{\langle 2 \rho, \varpi_k^\vee \rangle}{\langle \varpi_k, \varpi_k^\vee\rangle} \langle \varpi_k, \lambda \rangle = -\langle 2 \rho, \nu_k \rangle\leq - \langle 2 \rho, \mu \rangle \leq l + 3.\]
So
\[ \langle \varpi_k, \lambda \rangle \leq \frac{l + 3}{l + 1} < 2,\]
since $l > 1$. So $\langle \varpi_k, \lambda \rangle = 0$ or $1$ for all $k$.

Now assume for a contradiction that there exists $\lambda \in \mb{X}_*(T)_+ \setminus \{\alpha_i^\vee, \alpha_j^\vee, \alpha_i^\vee + \alpha_j^\vee \}$ such that $D_\lambda(Z) \neq \emptyset$. Since the divisor $D(Z) = \tilde{\chi}_Z^{-1}(0_{\Theta_Y^{-1}})$ is connected by Lemma \ref{lem:connectedspringerfibres} below, we can choose $\lambda$ so that $D_\lambda(Z)$ has nonempty intersection with one of $D_{\alpha_i^\vee}(Z)$, $D_{\alpha_j^\vee}(Z)$ or $D_{\alpha_i^\vee + \alpha_j^\vee}(Z)$. Choose a point in such an intersection over $z \in Z$, and let $-\lambda' \in \mb{X}_*(T)_-$ denote the degree of the corresponding stable map restricted to the irreducible component of genus $1$. Then we have $D_{\lambda'}(Z) \neq \emptyset$, $\lambda' \geq \lambda$ and $\lambda' \geq \alpha_r^\vee$ for some $\alpha_r \in \{\alpha_i, \alpha_j\}$. By the bound proved above, we must have $\langle \varpi_k, \lambda \rangle = 1$ for some $\alpha_k \in \Delta \setminus \{\alpha_i, \alpha_j\}$, and hence $\lambda' \geq \alpha_r^\vee + \alpha_k^\vee$. So adding rational tails to the degree $-\lambda'$ section if necessary, we deduce $D_{\alpha_r^\vee + \alpha_k^\vee}(Z) \neq \emptyset$ and $D_{\alpha_k^\vee}(Z) \neq \emptyset$.

Assume first that $G$ is not of type $A$. Since $D_{\alpha_k^\vee}(Z) \neq \emptyset$, there exists $z \in Z$ and a section of $\xi_{G, z}/B$ with degree $-\alpha_k^\vee$, and hence a section of $\xi_{G, z}/P_k$ with slope $-\varpi_k^\vee/\langle \varpi_k, \varpi_k^\vee \rangle$. So by Lemma \ref{lem:inducedsectionbound1}, there exists $z' \in Z$ such that $\xi_{G, z'}$ has Harder-Narasimhan reduction to $P_k$ with slope $-\varpi_k^\vee/\langle \varpi_k, \varpi_k^\vee \rangle$. Since $P_k \neq P$, we have $z' \in Z \setminus Z_0$ so in particular $z$ is not fixed under the $Z(L)_{rig}$-action. Comparing codimensions in $\bun_{G, rig}/E$ and in $Z$, we deduce that $\xi_{G, z'}$ must be regular unstable, which is a contradiction since it has the wrong Harder-Narasimhan type, as $\alpha_k$ is not a special root.

Assume on the other hand that $G$ is of type $A$. We have $k \notin\{i, i + 1\}$ and $r \in \{i, i + 1\}$ such that $D_{\alpha_r^\vee + \alpha_k^\vee}(Z) \neq \emptyset$. So there exists $z \in Z$ and a section of $\xi_{G, z}/P_{r, k}$ of slope $\nu \in \mb{X}_*(Z(L_{r, k})^\circ)_\mb{Q}$ satisfying $\langle \varpi_r, \nu\rangle = \langle \varpi_k, \nu\rangle = -1$, where $P_{r, k} \subseteq G$ is the standard parabolic of type $\{\alpha_r, \alpha_k\}$ and $L_{r, k}$ its standard Levi factor. But $\nu$ is a Harder-Narasimhan vector for $P_{r, k}$, so by Lemma \ref{lem:inducedsectionbound1}, there exists $z' \in Z$ such that $\xi_{G, z'}$ has Harder-Narasimhan reduction to $P_{r, k}$ with slope $\nu$. Since $P_{r, k} \neq P$, we have $z \in Z \setminus Z_0$. Again this implies that $\xi_{G, z'}$ is regular unstable, giving a contradiction.

So $D_\lambda(Z) = \emptyset$ for $\lambda \notin \{\alpha_i^\vee, \alpha_j^\vee, \alpha_i^\vee + \alpha_j^\vee\}$, and $D_{\alpha_i^\vee}(Z), D_{\alpha_j^\vee}(Z) \neq \emptyset$. This implies that $D_{\alpha_i^\vee + \alpha_j^\vee}(Z) \neq \emptyset$, for if this were not the case, we would have $D_{\alpha_i^\vee}(Z) \cap D_{\alpha_j^\vee}(Z) = \emptyset$ and hence $\tilde{\chi}_Z^{-1}(0_{\Theta_Y^{-1}})$ would be disconnected, contradicting Lemma \ref{lem:connectedspringerfibres}.
\end{proof}

\begin{lem} \label{lem:inducedsectionbound1}
In the setup of Proposition \ref{prop:subregularresolutions1}, fix some $z \in Z$ with corresponding $G$-bundle $\xi_{G, z}$. If there exists a section of $\xi_{G, z}/Q$ of degree $\nu$, where $Q$ is any standard parabolic with Harder-Narasimhan vector $\nu$ and $(Q, \nu) \neq (P, \mu)$, then
\begin{enumerate}[(1)]
\item \label{itm:inducedsectionbound1} there exists $z' \in Z$ such that the corresponding $G$-bundle $\xi_{G, z'}$ has Harder-Narasimhan reduction to $Q$ with degree $\nu$, and
\item \label{itm:inducedsectionbound2} $-\langle 2 \rho, \nu \rangle \leq l + 2$.
\end{enumerate}
\end{lem}
\begin{proof}
The assumptions imply that the stack $Z \times_{\bun_{G, rig}} \bun_{Q, rig}^\nu$ is nonempty. Since $Z \to \bun_{G, rig}/E$ is smooth, the preimage $Z \times_{\bun_{G, rig}} \bun_{Q, rig}^{ss, \nu}$ of $\bun_{Q, rig}^{ss, \nu}/E$ under the morphism
\[ Z \times_{\bun_{G, rig}} \bun_{Q, rig}^\nu = Z \times_{\bun_{G, rig}/E} \bun_{Q, rig}/E \longrightarrow \bun_{Q, rig}^\nu/E \]
is dense, hence nonempty. This proves \eqref{itm:inducedsectionbound1}. Since $(Q, \nu) \neq (P, \mu)$, by uniqueness of Harder-Narasimhan reductions, the $Z(L)_{rig}$-invariant locally closed substack $Z \times_{\bun_{G, rig}} \bun_{Q, rig}^{ss, \nu} \subseteq Z$ is disjoint from the $Z(L)_{rig}$-fixed locus $Z_0 \subseteq Z$. Since $Z_0 \to S$ has finite relative stabilisers, $Z \times_{\bun_{G, rig}}\bun_{Q, rig}^\nu \to S$ is therefore flat of relative dimension at least $1$, and hence has codimension at most
\[ \dim_S Z - 1 = l + 2.\]
But this codimension is equal to the codimension $- \langle 2\rho, \nu \rangle$ of $\bun_{Q, rig}^{ss, \nu}/E$ in $\bun_{G, rig}/E$, so \eqref{itm:inducedsectionbound2} follows.
\end{proof}

\begin{lem} \label{lem:connectedspringerfibres}
The morphisms
\begin{equation} \label{eq:connectedspringerfibres1}
 \tbun_G \longrightarrow \bun_G \times_{(\hat{Y}\sslash W)/\mb{G}_m} \Theta_Y^{-1}/\mb{G}_m
\end{equation}
and
\begin{equation} \label{eq:connectedspringerfibres2}
\tilde{\chi}_Z^{-1}(0_{\Theta_Y^{-1}}) \longrightarrow Y
\end{equation}
have connected fibres.
\end{lem}
\begin{proof}
Note that the target of \eqref{eq:connectedspringerfibres1} is a local complete intersection, hence Cohen-Macaulay. Moreover, \eqref{eq:connectedspringerfibres1} is an isomorphism over the open substack $\bun_G^{reg} \times_{(\hat{Y}\sslash W)/\mb{G}_m} \Theta_Y^{-1}/\mb{G}_m$, where $\bun_G^{reg} \subseteq \bun_G$ is the open substack of regular bundles \cite[\S 4.4]{davis19}. This open substack is big (i.e., the complement has codimension at least $2$) by \cite[Proposition 4.4.6]{davis19}, so the target is normal and the pushforward of $\mc{O}$ is $\mc{O}$. Connectedness of the fibres now follows from Zariski's connectedness theorem \cite[Theorem 11.3]{olsson07}. 

We can write \eqref{eq:connectedspringerfibres2} as a composition
\begin{equation} \label{eq:connectedspringerfibres3}
 \tilde{\chi}_Z^{-1}(0_{\Theta_Y^{-1}}) \longrightarrow \chi_Z^{-1}(0) \times_S 0_{\Theta_Y^{-1}} \longrightarrow 0_{\Theta_Y^{-1}} = Y.
\end{equation}
The first factor is a pullback of \eqref{eq:connectedspringerfibres1}, so has connected fibres. The morphism $\chi_Z^{-1}(0) \to S$ also has connected fibres, since the $H = Z(L)_{rig}$-action contracts $\chi_Z^{-1}(0)$ onto $Z_0$ and $Z_0 \to S$ has connected fibres. So both factors of \eqref{eq:connectedspringerfibres3} have connected fibres, and the first is proper, so their composition has connected fibres also.
\end{proof}

\subsection{Digression: Bruhat cells for $P$-bundles} \label{subsection:bruhat}

In this subsection, we consider $G$ an arbitrary reductive group. (The examples of interest will be our simply connected simple group $G$ from the rest of the paper, and $G = GL_n$.) The material presented here is a brief recap of \cite[\S 3.7]{davis19a}.

Given two parabolic subgroups $P, P' \subseteq G$, which we may as well assume standard, for each $w$ in the Weyl group $W$ of $G$, there is an associated \emph{Bruhat cell}
\[ C^w_{P, P'} \subseteq \bun_P \times_{\bun_G} \bun_{P'}.\]
Thinking of the stack on the right as the stack of pairs $(\xi_P, \sigma)$, where $\xi_P$ is a $P$-bundle and $\sigma$ is a section of the partial flag variety bundle $\xi_P \times^P G/P'$, we can define $C^w_{P, P'}$ as the locally closed substack of pairs such that $\sigma$ factors through the Bruhat cell $\xi_P \times^P PwP'/P'$. Since $PwP'/P' \cong P/P \cap wP'w^{-1}$ by the orbit-stabiliser theorem, we have
\[ C^w_{P, P'} \cong \bun_{P \cap wP'w^{-1}}.\]
Under this identification, the map $C^w_{P, P'} \to \bun_P$ (resp.\ $C^w_{P, P'} \to \bun_{P'}$) sends a $P \cap wP'w^{-1}$-bundle of degree $\mu \in \mb{X}_*(T_{P \cap wP'w^{-1}})$ to a $P$-bundle of degree $i_w(\mu)$ (resp.\ a $P'$-bundle of degree $j_w(\mu)$), where $i_w\colon \mb{X}_*(T_{P \cap wP'w^{-1}}) \to \mb{X}_*(T_P)$ and $j_w \colon \mb{X}_*(T_{P \cap wP'w^{-1}}) \to \mb{X}_*(T_{P'})$ are induced by the two inclusions
\[ P \longhookleftarrow P \cap wP'w^{-1} {\lhook\joinrel\relbar\joinrel\xrightarrow{w^{-1}(-)w}} P'.\]

The cell $C^w_{P, P'}$ depends only on the double coset $W_PwW_{P'}$, where $W_P$ and $W_{P'}$ are the Weyl groups of (the Levi subgroups of) $P$ and $P'$. A particularly nice choice of double coset representatives is given by
\begin{equation} \label{eq:cosetreps}
 W^0_{P, P'} = \{w \in W \mid w^{-1}\alpha_i \in \Phi_+\;\text{and}\; w\alpha_j \in \Phi_+\; \text{for}\;\alpha_i \in \Delta\setminus t(P)\; \text{and}\; \alpha_j \in \Delta \setminus t(P')\}.
\end{equation}
These are the coset representatives of minimal length. Note that $W_P$ (resp., $W_{P'}$) are generated by the simple reflections $s_i$ in the roots $\alpha_i \in \Delta \setminus t(P)$ (resp., $\alpha_i \in \Delta \setminus t(P')$).

When we want to keep track of the associated $P$-bundle and the degree of the associated $P'$-bundle, we write
\[ C^{w, \lambda}_{P, P'} = C^w_{P, P'} \times_{\bun_{P'}} \bun_{P'}^\lambda \quad \text{and} \quad C^{w, \lambda}_{P, P', \xi_P} = \{\xi_P\} \times_{\bun_P} C^{w, \lambda}_{P, P'} \]
for $\lambda \in \mb{X}_*(T_{P'})$ and $\xi_P \in \bun_P$.

The following proposition gives an extremely useful criterion for determining when the Bruhat cells cover a given fibre of $\bun_P \times_{\bun_G} \bun_{P'}^\lambda \to \bun_P$.

\begin{prop}[{\cite[Proposition 3.7.6]{davis19a}}] \label{prop:levidegreebound}
Let $\xi_P \to E_s$ be a $P$-bundle on a geometric fibre of $E \to S$ and suppose that there exists a point in $\bun_P \times_{\bun_G} \bun_{P'}^\lambda$ over $\xi_P$ that does not lie in any Bruhat cell. Then there exists $w \in W^0_{P, P'} \setminus \{1\}$ and $\lambda' < \lambda$ such that $C^{w, \lambda'}_{P, P', \xi_L \times^L P} \neq \emptyset$, where $L$ is the Levi factor of $P$ and $\xi_L = \xi_P \times^P L$ is the associated $L$-bundle.
\end{prop}

\subsection{Digression: some Bruhat cells for unstable vector bundles} \label{subsection:glnbruhat}

The aim of this subsection is to describe certain spaces of stable maps to partial flag variety bundles associated to particular minimally unstable $GL_n$-bundles on $E$. The spaces considered here will crop up again and again in the subregular part of the elliptic Grothendieck-Springer resolution $\tbun_G$.

Recall the notation for the root datum of $GL_n$ and the standard parabolic subgroups $Q^n_k \subseteq GL_n$ given in \S\ref{subsection:notation} We also consider the standard parabolic subgroup
\begin{align*}
R_n &= \{(a_{p,q})_{1 \leq p, q \leq n} \in GL_n \mid a_{p, q} = 0\;\text{for}\; q > \mrm{max}(p, n - 1)\} = \left\{\left(\begin{matrix} * & * & \cdots & * & 0 \\ \vdots & \vdots & & \vdots & \vdots \\ * & * & \cdots & * & 0 \\ * & * & \cdots & * & * \end{matrix}\right)\right\}
\end{align*}
of type $\{\beta_{n - 1}\}$. For $1 \leq k \leq n$, let
\[ X_k^n = Y_{Q^n_n}^{-e_n^*} \times_{Y_{Q^n_k}^{-e_n^*}} \mrm{KM}^{-e_n^*}_{Q^n_k, GL_n} \times_{\bun_{GL_n}^{-1}} \bun_{R_n}^{ss, -e_1^*},\]
where, for any standard parabolic subgroup $P \supseteq Q^n_n$, we use the same notation for a cocharacter $\lambda \in \mb{X}_*(Q^n_n)$ and for its image in $\mb{X}_*(P/[P, P])$. In words, a point of the stack $X_k^n$ over $s \in S$ consists of a tuple $(y, \sigma \colon C \to \xi_{R_n} \times^{R_n} GL_n/Q^n_k, \xi_{R_n})$, where $\xi_{R_n} \to E_s$ is a semistable $R_n$-bundle of degree $-e_1^*$ (which is the Harder-Narasimhan reduction of the unstable $GL_n$-bundle of the subsection title), $\sigma$ is a stable section of degree $-e_n^*$, and $y$ is a lift of (the isomorphism class of) the associated $T_{Q^n_k}$-bundle to a $T_{Q^n_n}$-bundle of degree $-e_n^*$.

For $1 \leq p \leq n - 1$, let $w_p \in W_{GL_n} = S_n$ be the cyclic permutation
\[ w_p = (n, n - 1, \ldots, p + 1, p) = s_{n - 1}s_{n - 2} \cdots s_p\]
and let $w_n = 1$ be the identity, where $W_{GL_n}$ is the Weyl group of $GL_n$, and $s_i = (i, i + 1)$ is the reflection in the root $\beta_i$. For $1 \leq p, k \leq n$, we write $C_{k, p}^{GL_n} \subseteq X^n_k$ for the locally closed substack of tuples $(y, \sigma, \xi_{R_n})$ such that the restriction of $\sigma$ to the genus $1$ component factors through the Bruhat cell
\[ \xi_{R_n} \times^{R_n} R_nw_pQ^n_k/Q^n_k \subseteq \xi_{R_n} \times^{R_n} GL_n/Q^n_k.\]

\begin{prop} \label{prop:glndecomposition}
For $1 \leq k \leq n$, there is a decomposition
\[ X_k^n = \bigcup_{1 \leq p < k} C_{k, p}^{GL_n} \cup C_{k, n}^{GL_n}\]
into disjoint locally closed substacks.
\end{prop}

We break the proof of Proposition \ref{prop:glndecomposition} into several lemmas.

\begin{lem} \label{lem:glndegrees}
Assume that $\xi_{R_n} \to E_s$ is a semistable $R_n$-bundle on a geometric fibre of $E \to S$ of degree $-e_1^*$ and that $\sigma \colon E_s \to \xi_{R_n} \times^{R_n} GL_n/Q^n_n$ is a section of degree $\lambda \leq -e_n^*$. Then $\lambda \in \{-e_n^*, -e_{n - 1}^*\}$.
\end{lem}
\begin{proof}
The section $\sigma$ corresponds to a complete flag
\[ 0 = V_n \subsetneq V_{n - 1} \subsetneq \cdots \subsetneq V_0 = V,\]
where $V$ is the vector bundle associated to the $GL_n$-bundle $\xi_{GL_n} = \xi_{R_n} \times^{R_n} GL_n$, such that $V_{i - 1}/V_i$ is a line bundle of degree $\langle e_i, \lambda \rangle$ for $i = 1, \ldots, n$. Since $\xi_{R_n}$ is the Harder-Narasimhan reduction of $\xi_{GL_n}$, $V$ has Harder-Narasimhan decomposition $V = M \oplus U$, where $U$ is a semistable vector bundle of rank $n - 1$ and degree $-1$ and $M$ is a line bundle of degree $0$. In particular, any quotient bundle of $V$ has slope $\geq -1/(n - 1)$, so we deduce that
\begin{equation} \label{eq:glndegrees1}
\langle e_1 + \cdots + e_i, \lambda \rangle = \deg V/V_i \geq \frac{-i}{n - 1}
\end{equation}
for $i = 1, \ldots, n - 1$.

Since $\lambda \leq -e_n^*$ by assumption, we have
\[ \lambda = -e_n^* - \sum_{i = 1}^{n - 1}d_i\beta_i^\vee \]
for some $d_i \in \mb{Z}_{\geq 0}$, where $\beta_i^\vee = e_i^* - e_{i + 1}^*$. Applying \eqref{eq:glndegrees1}, we have $d_i = 0$ for $1 \leq i \leq n - 2$ and $d_{n - 1} \in \{0, 1\}$, which implies the lemma.
\end{proof}

In what follows, we will write
\[ C^{w, \lambda}_k = \bun_{R_n}^{ss, -e_1^*} \times_{\bun_{R_n}} C^{w, \lambda}_{R_n, Q^n_k} \subseteq \bun_{R_n} \times_{\bun_{GL_n}} \bun_{Q^n_k}\]
for $w \in W^0_{R_n, Q^n_k}$ and $\lambda \in \mb{X}_*(T_{Q^n_k})$. Here $C^{w, \lambda}_{R_n, Q^n_k}$ is the Bruhat cell of \S\ref{subsection:bruhat}.

\begin{lem} \label{lem:glnbruhat}
Assume that $w \in W^0_{R_n, Q^n_n}$ and $\lambda \in \mb{X}_*(T_{Q^n_n})$ with $C^{w, \lambda}_n \neq \emptyset$ and $\lambda \leq -e_n^*$. Then
\[ (w, \lambda) \in \{(1, -e_{n - 1}^*)\} \cup \{(w_p, -e_n^*) \mid 1 \leq p < n\}.\]
\end{lem}
\begin{proof}
First note that by Lemma \ref{lem:glndegrees}, we know that $\lambda \in \{-e_n^*, -e_{n - 1}^*\}$. Moreover, we have from the definition \eqref{eq:cosetreps} that
\[ W^0_{R_n, Q^n_n} = \{w \in S_n \mid w^{-1}(i) < w^{-1}(i + 1) \;\text{for}\; 1 \leq i < n - 1\} = \{w_p \mid 1 \leq p \leq n\}.\]
Since $Q^n_n \subseteq GL_n$ is the standard Borel subgroup, the homomorphism
\[ j_w \colon \mb{X}_*(T_{Q^n_n}) = \mb{X}_*(T_{R_n \cap Q^n_n}) = \mb{X}_*(T_{R_n \cap wQ^n_nw^{-1}}) \longrightarrow \mb{X}_*(T_{Q^n_n}) \]
defined in \S\ref{subsection:bruhat} is just the isomorphism given by $w^{-1}$. So by nonemptiness of $C^{w, \lambda}_n$ there exists a semistable $L_n$-bundle $\xi_{L_n} \to E_s$ on a geometric fibre of $E \to S$ of degree $-e_1^*$, where $L_n \cong GL_{n - 1} \times \mb{G}_m$ is the standard Levi factor of $R_n$ and a section $\sigma_L \colon E_s \to \xi_{L_n}/(L_n \cap Q^n_n)$ of degree $w\lambda$. In particular, since $e_n \in \mb{X}^*(L_n)$, $\langle e_n, w\lambda \rangle = \langle e_n, -e_1^*\rangle = 0$ and $w\lambda$ is the degree of a section
\[ E_s \xrightarrow{\sigma_L} \xi_{L_n}/(L_n \cap Q^n_n) \longhookrightarrow \xi_{L_n} \times^{L_n} GL_n/Q^n_n.\]

If $\lambda = -e_n^*$ and $w = w_p$, then
\[ w\lambda = \begin{cases} -e_{n - 1}^*, & \text{if}\;\; p < n, \\ -e_n^*, & \text{if}\;\; p = n,\end{cases}\]
so from the above discussion we must have $p \in \{1, \ldots, n - 1\}$. If $\lambda = -e_{n - 1}^*$, on the other hand, then
\[ w\lambda = \begin{cases} -e_{n - 2}^*, & \text{if}\;\; p < n - 1, \\ -e_n^*, & \text{if}\;\; p = n - 1, \\ -e_{n - 1}^*, & \text{if}\;\; p = n,\end{cases} \]
so the above discussion and Lemma \ref{lem:glndegrees} imply that $p = n$. Combining these two cases gives that $(w, \lambda)$ is in the desired set.
\end{proof}

\begin{lem} \label{lem:glnbruhatsurjective}
For all $\lambda \in \mb{X}_*(T_{Q^n_n})$ with $\lambda \leq -e_n^*$, we have
\[ \bigcup_{w \in W^0_{R_n, Q^n_n}} C^{w, \lambda}_n = \bun_{Q^n_n}^\lambda \times_{\bun_{GL_n}^{-1}} \bun_{R_n}^{ss, -e_1^*}.\]
\end{lem}
\begin{proof}

Assume for a contradiction that this fails for some $\lambda \leq -e_n^*$. Then by Proposition \ref{prop:levidegreebound} there exist $w \in W^0_{R_n, Q^n_n} \setminus \{1\}$ and $\lambda' < \lambda$ such that $C^{w, \lambda'}_n \neq \emptyset$. So Lemmas \ref{lem:glndegrees} and \ref{lem:glnbruhat} imply that $\lambda' = -e_n^*$ and $\lambda \in \{-e_n^*, -e_{n - 1}^*\}$. But this contradicts $\lambda' < \lambda$ so we are done.
\end{proof}

\begin{lem} \label{lem:glnpartialbruhat}
Let $1 \leq k < n$. Then
\[ W^0_{R_n, Q^n_k} = \{w_p \mid 1 \leq p < k\} \cup \{w_n\} \]
and
\begin{equation} \label{eq:glnpartialbruhat1}
\bun_{Q^n_k}^{-e_n^*} \times_{\bun_{GL_n}} \bun_{R_n}^{ss, -e_1^*} = \bigcup_{w \in W^0_{R_n, Q^n_k}} C^{w, -e_n^*}_k.
\end{equation}
\end{lem}
\begin{proof}
From the definition,
\[ W^0_{R_n, Q^n_k} = \{ w \in W^0_{R_n, Q^n_n} \mid w(i) < w(i + 1) \; \text{for}\; k \leq i \leq n - 1\} = \{w_p \mid 1 \leq p < k\} \cup \{w_n\}\]
as claimed. Next, note that by \cite[Proposition 3.6.4]{davis19a} the natural morphism
\[ \mrm{KM}_{Q^n_n, GL_n}^{-e_n^*} \longrightarrow \mrm{KM}_{Q^n_k, GL_n}^{-e_n^*}\]
is surjective. So any geometric point of $\bun_{Q^n_k}^{-e_n^*} \times_{\bun_{GL_n}} \bun_{R_n}^{ss, -e_1^*}$ lifts to a point of $\bun_{Q^n_n}^\lambda \times_{\bun_{GL_n}} \bun_{R_n}^{ss, -e_1^*}$ for some $\lambda \leq -e_n^*$, and hence  $\lambda \in \{-e_n^*, -e_{n - 1}^*\}$ by Lemma \ref{lem:glndegrees}. So by Lemma \ref{lem:glnbruhatsurjective}, the morphism
\[ \coprod_{\substack{w \in W^0_{R_n, Q^n_n} \\ \lambda \in \{-e_n^*, -e_{n - 1}^*\}}} C^{w, \lambda}_n \longrightarrow \coprod_{w \in W^0_{R_n, Q^n_k}} C^{w, -e_n^*}_k \longrightarrow \bun_{Q^n_k}^{-e_n^*} \times_{\bun_{GL_n}} \bun_{R_n}^{ss, -e_1^*} \]
is surjective, which proves \eqref{eq:glnpartialbruhat1}.
\end{proof}

\begin{proof}[Proof of Proposition \ref{prop:glndecomposition}]
Suppose first that $k < n$. Since any $Q^n_k$-bundle of degree $\leq -e_n^*$ can be reduced to a $Q^n_n$-bundle of degree $\leq -e_n^*$ by \cite[Proposition 3.6.4]{davis19a}, Lemma \ref{lem:glndegrees} implies that
\[ \mrm{KM}_{Q_k^n, GL_n}^{-e_n^*} \times_{\bun_{GL_n}} \bun_{R_n}^{ss, -e_1^*} = \bun_{Q^n_k}^{-e_n^*} \times_{\bun_{GL_n}} \bun_{R_n}^{ss, -e_1^*}, \]
since $-e_n^*$ and $-e_{n - 1}^*$ have the same image in $\mb{X}_*(T_{Q^n_k})$. So we have the desired decomposition of $X_k^n$ into locally closed substacks by Lemma \ref{lem:glnpartialbruhat} (note that $C^{GL_n}_{k, p}$ is the preimage of $C^{w_p, \lambda}_k$ in $X^n_k$ in this case).

On the other hand, if $k = n$, then Lemma \ref{lem:glndegrees} implies that $X^n_n$ decomposes as a disjoint union
\[ X_n^n = (\bun_{Q^n_n}^{-e_n^*} \times_{\bun_{GL_n}} \bun_{R_n}^{ss, -e_1^*}) \cup (\bun_{Q^n_n}^{-e_{n - 1}^*} \times_{\bun_{GL_n}} \bun_{R_n}^{ss, -e_1^*} \times_S E)\]
of locally closed substacks, where the first factor is the locus of stable sections with irreducible domain and the second factor is the locus of stable sections with a single rational component of degree $\beta_{n - 1}^\vee = e_{n - 1}^* - e_n^*$. By Lemmas \ref{lem:glnbruhat} and \ref{lem:glnbruhatsurjective}, this decomposes further as the desired decomposition
\[ X_n^n = \bigcup_{1 \leq p < n} C_{n, p}^{GL_n} \cup C_{n, n}^{GL_n}\]
so we are done.
\end{proof}

From the proof of Proposition \ref{prop:glndecomposition}, we have that
\[ C^{GL_n}_{n, n} \cong C^{1, -e_{n - 1}^*}_n \times_S E = \bun_{Q^n_n}^{-e_{n - 1}^*} \times_{\bun_{GL_n}} \bun_{R_n}^{ss, -e_1^*} \times_S E \]
is the locus of stable maps with a single rational component of degree $\beta_{n - 1}^\vee$. The natural projection to $E$ keeps track of the point of attachment of the rational component, and the projection to the other factors keeps track of the restriction to the elliptic component. Note that the projection to $E$ agrees with composition of $C_{n, n}^{GL_n} \to C_{1, n}^{GL_n}$ with the morphism
\begin{align}
C_{1, n}^{GL_n} \longrightarrow Y_{Q^n_n}^{-e_n^*} \times_{\mrm{Pic}^{-1}_S(E)} Y_{R_n}^{-e_1^*} &\longrightarrow \mrm{Pic}^1_S(E) = E \label{eq:glnbruhatpoint2}\\
(y, y')&\longmapsto e_n(y') - e_n(y). \nonumber
\end{align}

For $1 \leq p < n$, we let
\[ M_p^{GL_n} \subseteq C_{1, n}^{GL_n}\]
be the closed substack given by the fibre product
\[
\begin{tikzcd}
M_p^{GL_n} \ar[r] \ar[d] & C_{1, n}^{GL_n} \ar[d] \\
Y_{Q^n_n}^{-e_n^*} \ar[r, "\theta_p^{GL_n}"] & Y_{Q^n_n}^{-e_n^*} \times_S E,
\end{tikzcd}
\]
where the morphism $C_{1, n}^{GL_n} \to E$ is \eqref{eq:glnbruhatpoint2}, and the morphism $Y_{Q^n_n}^{-e_n^*} \to Y_{Q^n_n}^{-e_n^*} \times_S \mrm{Pic}^1_S(E)$ is given by
\begin{align*}
\theta^{GL_n}_p \colon Y_{Q^n_n}^{-e_n^*} &\longrightarrow Y_{Q^n_n}^{-e_n^*} \times_S \mrm{Pic}^1_S(E) = Y_{Q^n_n}^{-e_n^*} \times_S E \\
y &\longmapsto (y, e_p(y) - e_n(y)).
\end{align*}

\begin{prop} \label{prop:glnbruhatcomparison}
For all $1 \leq k < n$, the morphism $X^n_{k + 1} \to X^n_k$ restricts to isomorphisms
\[ C^{GL_n}_{k + 1, n} \overset{\sim}\longrightarrow C^{GL_n}_{k, n} \quad \text{and} \quad C^{GL_n}_{k + 1, p} \overset{\sim}\longrightarrow C^{GL_n}_{k, p} \]
for $1 \leq p < k$, and a morphism
\[ C^{GL_n}_{k + 1, k} \longrightarrow M_k^{GL_n} \subseteq C^{GL_n}_{k, n} \cong C^{GL_n}_{1, n}\]
that exhibits $C_{k + 1, k}^{GL_n}$ as an $\mb{A}^1$-bundle over $M_k^{GL_n}$.
\end{prop}
\begin{proof}
If $k < n - 1$, then the morphism $C_{k + 1, n}^{GL_n} \to C_{k, n}^{GL_n}$ can be identified with
\[ Y_{Q^n_n}^{-e_n^*} \times_{Y_{Q^n_{k + 1}}^{-e_n^*}} \bun_{R_n \cap Q^n_{k + 1}}^{-e_{n - 1}^*} \times_{\bun_{R_n}^{-e_1^*}} \bun_{R_n}^{ss, -e_1^*} \longrightarrow Y_{Q^n_n}^{-e_n^*} \times_{Y_{Q^n_{k}}^{-e_n^*}} \bun_{R_n \cap Q^n_{k}}^{-e_{n - 1}^*} \times_{\bun_{R_n}^{-e_1^*}} \bun_{R_n}^{ss, -e_1^*}. \]
This is a pullback of
\[ \bun_{Q^{n - 1}_{k + 1}}^{-e_{n - 1}^*} \times_{\bun_{GL_{n - 1}}} \bun_{GL_{n - 1}}^{ss, -1} \longrightarrow Y^{-e_{n - 1}^*}_{Q^{n - 1}_{k + 1}} \times_{Y^{-e_{n - 1}^*}_{Q^{n - 1}_{k}}} \bun_{Q^{n - 1}_{k}}^{-e_{n - 1}^*} \times_{\bun_{GL_{n - 1}}} \bun_{GL_{n - 1}}^{ss, -1} \]
under the morphism $R_n \to GL_{n - 1}$ forgetting the last row and column, hence an isomorphism by \cite[Lemma 4.3.7]{davis19}.

If $k = n - 1$, then we can identify $C^{GL_n}_{k + 1, n} \to C^{GL_n}_{k, n}$ with the morphism
\[ \bun_{Q^n_n \cap R_n}^{-e_{n - 1}^*} \times_{\bun_{R_n}} \bun_{R_n}^{ss, -e_1^*} \times_S E \longrightarrow Y_{Q^n_n}^{-e_n^*} \times_{Y_{Q^n_{n - 1}}^{-e_n^*}} \bun_{R_n \cap Q^n_{n - 1}}^{-e_{n - 1}^*} \times_{\bun_{R_n}} \bun_{R_n}^{ss, -e_1^*}.\]
Since $R_n \cap Q^n_n = R_n \cap Q^n_{n - 1} = Q^n_n$, this is naturally a pullback of the isomorphism
\begin{align*}
Y_{Q^n_n}^{-e_{n - 1}^*} \times_S E &\overset{\sim}\longrightarrow Y_{Q^n_n}^{-e_n^*} \times_{Y_{Q^n_{n - 1}}^{-e_n^*}} Y_{Q_n^n}^{-e_{n - 1}^*} \\
(y, x) &\longmapsto (y + \beta_{n - 1}^\vee(x), y),
\end{align*}
hence an isomorphism itself.

If $k \leq n$ and $1 \leq p < k$, then $L_n \cap w_pQ^n_kw_p^{-1} = L_n \cap Q^n_{k - 1}$, where $L_n \subseteq R_n$ is the standard Levi factor. One easily checks that, in the notation of \S\ref{subsection:bruhat}, the morphism
\[ (i_{w_p}, j_{w_p}) \colon \mb{X}_*(T_{R_n \cap w_pQ^n_kw_p^{-1}}) \longrightarrow \mb{X}_*(T_{R_n}) \oplus \mb{X}_*(T_{Q^n_k}) \]
is injective and sends $-e_{n - 1}^*$ to $(-e_1^*, -e_n^*)$. So
\[ C_k^{w_p, -e_n^*} = \bun_{R_n \cap w_pQ^n_kw_p^{-1}}^{-e_{n - 1}^*} \times_{\bun_{R_n}} \bun_{R_n}^{ss, -e_1^*}.\]
By general nonsense, the right hand side is the relative space of sections of
\[ \eta_{k, p} \times^{(L_n \cap w_pQ^n_kw_p^{-1})R_u(R_n)} \frac{R_u(R_n)}{R_u(R_n) \cap w_pQ^n_kw_p^{-1}} \longrightarrow \bun_{L_n \cap w_pQ^n_kw_p^{-1}}^{-e_{n - 1}^*} \times_{\bun_{L_n}} \bun_{R_n}^{ss, -e_1^*} \times_S E\]
over
\[ \bun_{L_n \cap w_pQ^n_kw_p^{-1}}^{-e_{n - 1}^*} \times_{\bun_{L_n}} \bun_{R_n}^{ss, -e_1^*} \subseteq \bun_{(L_n \cap w_pQ^n_kw_p^{-1})R_u(R_n)}, \]
where $\eta_{k, p}$ is the universal $(L_n \cap w_pQ^n_kw_p^{-1})R_u(R_n)$-bundle. By Lemma \ref{lem:glnbruhatbasecomparison} below, we can therefore identify $C_{k, p}^{GL_n} = Y_{Q^n_n}^{-e_n^*} \times_{Y_{Q^n_k}^{-e_n^*}} C_k^{w_p, -e_n^*}$ with the relative space of sections of
\begin{equation} \label{eq:glnbruhatcomparison1}
 \bar{\eta}_{k, p} \times^{(L_n \cap w_pQ^n_kw_p^{-1})R_u(R_n)} \frac{R_u(R_n)}{R_u(R_n) \cap w_pQ^n_kw_p^{-1}} \longrightarrow M_p^{GL_n} \times_S E
\end{equation}
over $M_p \subseteq C_{1, n}^{GL_n}$, where $\bar{\eta}_{k, p}$ is a pullback of $\eta_{k, p}$. Note that by Lemma \ref{lem:glnunipotent} below, there is an isomorphism
\[ \frac{R_u(R_n)}{R_u(R_n) \cap w_pQ^n_kw_p^{-1}} \cong U_{k, p}^\vee \otimes \mb{Z}_{e_n},\]
of $L_n \cap w_pQ^n_kw_p^{-1}$-varieties, where $U_{k, p}$ is the representation described immediately before Lemma \ref{lem:glnunipotent}. So after pulling back along the smooth surjection $\bun_{L_n}^{ss, -e_1^*} \to \bun_{R_n}^{ss, -e_1^*}$, \eqref{eq:glnbruhatcomparison1} becomes a family of stable vector bundles on $E$ of degree $1$.

If $k \leq n - 1$, then by the above discussion, the morphism $C^{GL_n}_{k + 1, p} \to C^{GL_n}_{k, p}$ becomes the pushforward of a surjective morphism between families of stable vector bundles of degree $1$ over $M_p^{GL_n}$ after pulling back along $\bun_{L_n}^{ss, -e_1^*} \to \bun_{R_n}^{ss, -e_1^*}$, and is therefore an isomorphism as claimed. On the other hand, the morphism $C_{k + 1, k}^{GL_n} \to C_{k, n}^{GL_n}$ becomes the relative space of sections over $M_k^{GL_n} \subseteq C_{1, n}^{GL_n} \cong C_{k, n}^{GL_n}$ of a family of stable vector bundles of degree $1$, and is therefore an $\mb{A}^1$-bundle over $M_k^{GL_n}$.
\end{proof}

\begin{lem} \label{lem:glnbruhatbasecomparison}
If $p < k \leq n$, then the morphism
\begin{equation} \label{eq:glnbruhatbasecomparison1}
Y_{Q^n_n}^{-e_n^*} \times_{Y_{Q^n_k}^{-e_n^*}}(\bun_{L_n \cap w_pQ^n_kw_p^{-1}}^{-e_{n - 1}^*} \times_{\bun_{L_n}^{-e_1^*}} \bun_{R_n}^{ss, -e_1^*}) \longrightarrow Y_{Q^n_n}^{-e_n^*} \times_{\mrm{Pic}^{-1}_S(E)} \bun_{R_n}^{ss, -e_1^*} = C^{GL_n}_{1, n} = X^n_1
\end{equation}
induced by the inclusion $L_n \cap w_pQ^n_kw_p^{-1} \subseteq L_n$ factors through an isomorphism onto $M_p^{GL_n}$. Here the morphisms to $\mrm{Pic}^{-1}_S(E)$ in the fibre product in the right hand side of \eqref{eq:glnbruhatbasecomparison1} are both given by the determinant.
\end{lem}
\begin{proof}
First note that the morphism
\[ \bun_{L_n \cap w_pQ^n_kw_p^{-1}}^{-e_{n - 1}^*} \times_{\bun_{L_n}^{-e_1^*}} \bun_{R_n}^{ss, -e_1^*} \longrightarrow Y_{L_n \cap w_pQ^n_kw_p^{-1}}^{-e_{n - 1}^*} \times_{Y_{R_n}^{-e_1^*}} \bun_{R_n}^{ss, -e_1^*} \]
is a pullback of
\[ \bun_{Q^{n - 1}_{k - 1}}^{-e_{n - 1}^*} \times_{\bun_{GL_{n - 1}}} \bun_{GL_{n - 1}}^{ss, -1} \longrightarrow Y_{Q^{n - 1}_{k - 1}}^{-e_{n - 1}^*} \times_{\mrm{Pic}^{-1}_S(E)} \bun_{GL_{n - 1}}^{ss, -1} \]
and hence an isomorphism by \cite[Lemma 4.3.7]{davis19}. Composing with the isomorphism
\[ j_{w_p} \colon Y_{L_n \cap w_pQ^n_kw_p^{-1}}^{-e_{n - 1}^*} \overset{\sim}\longrightarrow Y_{Q^n_k}^{-e_n^*} \]
allows us to identify \eqref{eq:glnbruhatbasecomparison1} with the closed immerison 
\[ Y_{Q^n_n}^{-e_n^*} \times_{Y_{R_n}^{-e_1^*}} \bun_{R_n}^{ss, -e_1^*} \longrightarrow Y_{Q^n_n}^{-e_n^*} \times_{\mrm{Pic}^{-1}_S(E)} \bun_{R_n}^{ss, -e_1^*},\]
where the morphism $Y_{Q^n_n}^{-e_n^*} \to Y_{R_n}^{-e_1^*}$ is the composition
\[ Y_{Q^n_n}^{-e_n^*} \longrightarrow Y_{Q^n_k}^{-e_n^*} \xrightarrow{j_{w_p}^{-1}} Y_{L_n \cap w_pQ^n_kw_p^{-1}}^{-e_{n - 1}^*} \xrightarrow{i_{w_p}} Y_{R_n}^{-e_1^*}.\]
Chasing through the various definitions now shows that the source of this morphism is precisely $M_p^{GL_n}$, so we are done.
\end{proof}

In the following lemma, we write $U_{k, p}$ for the $L_n \cap w_pQ^n_kw_p^{-1}$-representation induced by the homomorphism
\[ L_n \cap w_pQ^n_kw_p^{-1} = Q^{n - 1}_{k - 1} \times \mb{G}_m \longrightarrow Q^{n - 1}_{k - 1} \longrightarrow GL_{n - p} \]
given by deleting the last row and column and the first $p - 1$ rows and columns.

\begin{lem} \label{lem:glnunipotent}
If $p < k$, then there is an $L_n \cap w_pQ^n_kw_p^{-1}$-equivariant isomorphism
\begin{equation} \label{eq:glnunipotent1}
 R_u(R_n)/(R_u(R_n) \cap w_pQ^n_kw_p^{-1}) \overset{\sim}\longrightarrow U_{k, p}^\vee \otimes \mb{Z}_{e_n}.
\end{equation}
\end{lem}
\begin{proof}

If $\beta$ is a root of $R_u(R_n)$, then the root subgroup $U_\beta \cong \mb{G}_a \subseteq R_u(R_n)$ maps injectively into $R_u(R_n)/(R_u(R_n) \cap w_pQ^n_kw_p^{-1})$ if and only if $w_p^{-1}\beta$ is not a root of $Q^n_k$. In particular, this implies that $\beta$ is a negative root and $w_p^{-1}\beta$ is a positive root, and hence that
\[ \beta \in \Sigma = \{-\beta_{n - 1}, -\beta_{n - 1}- \beta_{n - 2}, \ldots, - \beta_{n - 1} - \beta_{n - 2} - \cdots - \beta_p\},\]
and
\[ w_p^{-1}\beta \in \{\beta_{n - 1} + \beta_{n - 2} + \cdots + \beta_p, \beta_{n - 2} + \cdots + \beta_p, \ldots, \beta_p\}.\]
Note that if $\beta \in \Sigma$, then $U_\beta \subseteq R_u(P)$, and $w_p^{-1}\beta$ is not a root of $Q^n_k$, so $\Sigma$ is precisely the set of roots appearing in $R_u(R_n)/(R_u(R_n) \cap w_pQ^n_kw_p^{-1})$.

It is clear from the above that $R_u(R_n)/(R_u(R_n) \cap w_pQ^n_kw_p^{-1})$ is isomorphic to an $L_n \cap w_pQ^n_kw_p^{-1}$-representation. The isomorphism \eqref{eq:glnunipotent1} follows by inspection of the weights of this representation.
\end{proof}

\subsection{The divisor $D_{\alpha_j^\vee}(Z)$} \label{subsection:dalphaj}

The purpose of this subsection is to prove Proposition \ref{prop:subregularresolutions2} below, which refines Theorem \ref{thm:introsubregularresolutions}, \eqref{itm:introsubregularresolutions2}. For the statement, recall Notation \ref{notation:dynkindecomposition}. For $1 \leq k \leq n_0 + 1$, we write $\theta_k$ for the section
\begin{align*}
\theta_k \colon Y &\longrightarrow Y \times_S \mrm{Pic}^0_S(E) \\
y &\longmapsto \begin{cases} (y, \varpi_j(y) - \varpi_i(y) - \varpi_{c_0, 1}(y)), & \text{if}\;\; k = 1, \\ (y, \varpi_j(y) - \varpi_i(y) - \varpi_{c_0, k}(y) + \varpi_{c_0, k - 1}(y)), & \text{if}\;\; 1 < k \leq n_0, \\ (y, 0), & \text{if}\;\; k = n_0 + 1.\end{cases}
\end{align*}

\begin{prop} \label{prop:subregularresolutions2}
Assume we are in the setup of Proposition \ref{prop:subregularresolutions1}. Then there is a sequence of $n_0 + 1$ morphisms
\[ D_{\alpha_j^\vee}(Z) = D_{n_0 + 2} \longrightarrow D_{n_0 + 1} \longrightarrow \cdots \longrightarrow D_1\]
over $Y \times_S Z$ such that $D_1$ is a line bundle over $Y \times_S \mrm{Pic}^0_S(E)$ and $D_{k + 1} \to D_k$ is the blowup along the section $\theta_k \colon Y \to Y \times_S \mrm{Pic}^0_S(E) \subseteq D_k$ of the proper transform of the zero section of $D_1$.
\end{prop}
\begin{proof}
The spaces $D_k$ are defined as follows. For $1 \leq k \leq n_0$, let $P_k \subseteq G$ be the standard parabolic with type $t(P_k) = \Delta \setminus \{\alpha_{c_0, k}, \ldots, \alpha_{c_0,n_0}\} = \Delta \setminus \{\alpha_{c_0,k}, \ldots, \alpha_{c_0, n_0 - 1}, \alpha_i\}$, and let $P_{n_0 + 1} = B$. Then for $1 \leq k \leq n_0 + 1$, we define
\begin{align*}
D_k &= Y_B^{-\alpha_j^\vee} \times_{Y_{P_k}^{-\alpha_j^\vee}} \mrm{KM}_{P_k, G, rig}^{-\alpha_j^\vee} \times_{\bun_{G, rig}} Z \times_S E\\
&\cong Y \times_{Y_{P_k}} (\mrm{KM}_{P_k, G, rig}^{-\alpha_j^\vee} \times_{\bun_{G, rig}} Z \times_S E),
\end{align*}
where the morphism to $Y_{P_k}$ in the last fibre product is given by the composition
\begin{align*}
\mrm{KM}_{P_k, G/S, rig}^{-\alpha_j^\vee} \times_{\bun_{G, rig}} Z \times_S E \xrightarrow{\mrm{Bl}_{P_k}} Y_{P_k}^{-\alpha_j^\vee} \times_S E &\longrightarrow Y_{P_k} \\
(y, x) &\longmapsto y + \alpha_j^\vee(x).
\end{align*}
For $k \leq n_0$, the morphism $D_{k + 1} \to D_k$ is the obvious one induced by the inclusion $P_{k + 1} \subseteq P_k$. To describe the morphism $D_{\alpha_j^\vee}(Z) \to D_{n_0 + 1}$, note that every stable map parametrised by a point in $D_{\alpha_j^\vee}(Z)$ has a unique rational irreducible component of degree $\alpha_j^\vee$. Deleting this rational component and recording the point of $E$ over which it was attached defines the morphism
\[ D_{\alpha_j^\vee}(Z) \longrightarrow \mrm{KM}_{B, G, rig}^{-\alpha_j^\vee} \times_{\bun_{G, rig}} Z \times_S E = D_{n_0 + 1}.\]

For $k \leq n_0 + 1$, the spaces $D_k$ can be decomposed into locally closed subsets as follows. First, by Proposition \ref{prop:subregp1bundles}, for $z \in Z_0 \subseteq Z$, every stable section of $\xi_{G, z}/P_1 = \xi_{P, z} \times^P G/P_1$ of degree $-\alpha_j^\vee$ is in fact a genuine section of the subvariety $\xi_{P, z} \times^P PP_1/P_1 \cong \xi_{P, z}/(P \cap P_1)$. So there is an isomorphism
\begin{equation} \label{eq:subregularresolutions2:3}
\mrm{KM}_{P_k, G, rig}^{-\alpha_j^\vee} \times_{\bun_{G, rig}} Z_0 \cong \mrm{KM}_{P_k, P_1, rig}^{-\alpha_j^\vee} \times_{\bun_{P_1, rig}} \bun_{P \cap P_1, rig}^{-\alpha_i^\vee - \alpha_j^\vee} \times_{\bun_{P, rig}} Z_0.
\end{equation}
Explicitly, the right hand side is tautologically identified with the space of pairs $(z, \sigma)$ where $z \in Z_0$ and $\sigma$ is a stable section of $\xi_{P, z} \times^{P \cap P_1} P_1/P_k$ of appropriate degree such that the image in $\xi_{P, z}/(P \cap P_1)$ is a genuine section, while the left hand side is the space of pairs $(z, \sigma)$ where $z \in Z_0$ and $\sigma$ is a stable section of $\xi_{P, z} \times^P G/P_k$. The isomorphism \eqref{eq:subregularresolutions2:3} is obtained in this description by applying the isomorphism $P \times^{P \cap P_1} P_1/P_k \cong PP_1/P_k$ and the inclusion $PP_1/P_k \hookrightarrow G/P_k$. The homomorphism $\pi_{P_1} \colon P_1 \to GL_{n_0 + 1}$ of Proposition \ref{prop:subregp1hom} therefore induces a morphism
\begin{equation} \label{eq:subregularresolutions2:1}
 D_k \times_Z Z_0 \longrightarrow X_{k, rig}^{n_0 + 1},
\end{equation}
where $X_{k, rig}^{n_0 + 1}$ is the rigidification of the space $X_k^{n_0 + 1}$ of \S\ref{subsection:glnbruhat} with respect to the image of $Z(G)$ in $Z(GL_{n_0 + 1})$. We therefore get a decomposition
\begin{equation} \label{eq:subregularresolitions2:2}
D_k = (D_k \times_Z (Z \setminus Z_0)) \cup \bigcup_{1 \leq p < k} C_{k, p} \cup C_{k, n_0 + 1},
\end{equation}
where $C_{k, p}$ is the preimage of $C_{k, p}^{GL_{n_0 + 1}} \subseteq X^{n_0 + 1}_k$ under \eqref{eq:subregularresolutions2:1}. The behaviour of these locally closed subsets under the morphisms $D_{k + 1} \to D_k$ is described by Proposition \ref{prop:subregbruhatcomparison}.

There is a morphism $C_{1, n_0 + 1} \to Y \times_S \mrm{Pic}^0_S(E)$ \eqref{eq:bruhatcellcurvecomparison1}, which is defined so that the image of a stable section in $D_{\alpha_j^\vee}(Z)$ over $y \in Y$ with two rational components is sent to $(y, x_j - x_i)$, where $x_j \in E$ (resp.\ $x_i \in E$) is the point of attachment of the rational component of degree $\alpha_j^\vee$ (resp.\ $\alpha_i^\vee$). This morphism is an isomorphism by Lemma \ref{lem:bruhatcellcurvecomparison}.

Since $\mrm{KM}_{P_k, G, rig}^{-\alpha_j^\vee} \times_{\bun_{G, rig}} Z$ is smooth over $Y_{P_k}^{-\alpha_j^\vee}$, each space $D_k$ is smooth over $Y$. Proposition \ref{prop:subregbruhatcomparison} implies that they are all isomorphic to $D_{n_0 + 1}$, and hence to $D_{\alpha_j^\vee}(Z)$, outside $Z_0$. So the spaces $D_k$ are all smooth surfaces over $Y$.

In particular, $C_{1, n_0 + 1} = D_1 \times_Z Z_0$ is a Cartier divisor on $D_1$. Moreover, choosing any cocharacter of the torus $Z(L)_{rig}$ whose negative is a Harder-Narasimhan vector for the parabolic $P^+$ opposite to $P$, we get compatible actions of $\mb{G}_m$ on $Z$ and $D_1$ acting trivially on $Z_0$ and $D_1 \times_Z Z_0$, such that $\mb{G}_m$ acts on the fibres of the affine space bundle $Z \to Z_0$ with positive weights. Since the normal cone of $D_1 \times_Z Z_0$ in $D_1$ is a line bundle and $\mb{G}_m$ acts nontrivially on it, $\mb{G}_m$ acts on it with a single nonzero weight. So \cite[Lemma 4.3.11]{davis19} shows that $D_1$ is isomorphic to a line bundle over $C_{1, n_0 + 1} = Y \times_S \mrm{Pic}^0_S(E)$ as claimed.

It remains to show that $D_{k + 1} \to D_k$ is the blowup along the proper transform of $\theta_k$ for $1 \leq k \leq n_0 + 1$. If $k \leq n_0$, this follows from Proposition \ref{prop:subregbruhatcomparison} and Lemma \ref{lem:blowuprecognition}. For $k = n_0 + 1$, note that $D_{n_0 + 2} = D_{\alpha_j^\vee}(Z) \to D_{n_0 + 1}$ is an isomorphism outside the proper transform of $\theta_{n_0 + 1}$ (the locus of curves with a degree $\alpha_i^\vee$ rational component over the point of attachment of the degree $\alpha_j^\vee$ rational curve), and the fibre over any point in this proper transform is an irreducible curve. The claim now follows by Lemma \ref{lem:blowuprecognition} again.
\end{proof}

The rest of the subsection is devoted to the various lemmas and propositions quoted in the proof of Proposition \ref{prop:subregularresolutions2}.

\begin{prop} \label{prop:subregp1bundles}
Let $z \in Z_0 \subseteq Z$ and let $\xi_{P, z}$ and $\xi_{G, z} = \xi_{P, z} \times^P G$ be the corresponding $P$ and $G$-bundles. Then any stable section of $\xi_{G, z}/P_1 = \xi_{P, z} \times^G G/P_1$ of degree $-\alpha_j^\vee$ is a genuine section, and factors through $\xi_{P, z} \times^P PP_1/P_1$.
\end{prop}
\begin{proof}
The proposition is equivalent to the claim that
\begin{equation} \label{eq:subregularp1bundles3}
C^{1, -\alpha_j^\vee}_{P_1}(Z_0) \longhookrightarrow \bun_{P_1, rig}^{-\alpha_j^\vee} \times_{\bun_{G, rig}} Z_0 \longhookrightarrow \mrm{KM}_{P_1, G, rig}^{-\alpha_j^\vee} \times_{\bun_{G, rig}} Z_0
\end{equation}
is surjective, where, for $1 \leq k \leq n_0 + 1$, $w \in W^0_{P, P_k}$ and $\lambda \in \mb{X}_*(T_{P_k})$, we write
\[ C^{w, \lambda}_{P_k}(Z_0) = C^{w, \lambda}_{P, P_k, rig} \times_{\bun_{P, rig}} Z_0,\]
where $C^{w, \lambda}_{P, P_k, rig}$ is the rigidification of the Bruhat cell $C^{w, \lambda}_{P, P_k}$ of \S\ref{subsection:bruhat}. Lemma \ref{lem:subregularborelcells} below and Proposition \ref{prop:levidegreebound} imply that the morphism
\[ \coprod_{\substack{w \in W^0_{P, B} \cap W_{L_1}\\ \lambda =- w^{-1}(\alpha_i^\vee + \alpha_j^\vee)}} C^{w, \lambda}(Z_0) \longrightarrow \bun_{B, rig}^\lambda \times_{\bun_{G, rig}} Z_0 \]
is surjective for all $\lambda \leq -\alpha_j^\vee$, where $W_{L_1}$ is the Weyl group of the Levi factor $L_1 \subseteq P_1$ and $C^{w, \lambda}(Z_0) = C^{w, \lambda}_{P_{n_0 + 1}}(Z_0)$. Since the morphism
$\mrm{KM}_{B, G}^{-\alpha_j^\vee} \to \mrm{KM}_{P_1, G}^{-\alpha_j^\vee}$ is also surjective by \cite[Proposition 3.6.4]{davis19a}, and maps sections coming from $C^{w, \lambda}(Z_0)$ to $C^{1, -\alpha_j^\vee}_{P_1}(Z_0)$, surjectivity of \eqref{eq:subregularp1bundles3} now follows.
\end{proof}

\begin{lem} \label{lem:subregularborelcells}
Assume that $w \in W^0_{P, B}$, $\lambda \leq -\alpha_j^\vee$ and $C^{w, \lambda}(Z_0) \neq \emptyset$. Then $w \in W_{L_1}$ and $\lambda = -w^{-1}(\alpha_i^\vee + \alpha_j^\vee) \in \{-\alpha_j^\vee, -\alpha_i^\vee - \alpha_j^\vee\}$, where $L_1 \subseteq P_1$ is the standard Levi subgroup.
\end{lem}
\begin{proof}
It is immediate from Lemma \ref{lem:subregularemptydivisors} that $\lambda \in \{-\alpha_j^\vee, -\alpha_i^\vee - \alpha_j^\vee\}$. If $C^{w, \lambda}(Z_0) \neq \emptyset$, then there exists a geometric point $z \colon \spec k \to Z_0$ over $s \colon \spec k \to S$ and a section $\sigma_L \colon E_s \to \xi_{L, z}/(L \cap B)$ of degree $w\lambda \in \mb{X}_*(T)$. Since $\xi_{L, z}$ has slope $\mu$, we must have
\[ \langle \varpi_i, w\lambda \rangle = \langle \varpi_i, \mu \rangle = -1. \]
Since $\lambda$ and hence $w \lambda$ is a coroot, we therefore have $w\lambda \in \Phi^\vee_- \subseteq \mb{X}_*(T)_-$. Since composing $\sigma_L$ with the inclusion $\xi_{L, z}/(L \cap B) \to \xi_{G, z}/B$ defines a section of degree $w\lambda$, we deduce that $D_{-w\lambda}(Z) \neq \emptyset$, and hence that $w \lambda \in \{-\alpha_i^\vee, -\alpha_i^\vee - \alpha_j^\vee\}$.

If $w\lambda = -\alpha_i^\vee$, then $w^{-1}\alpha_i^\vee \in \Phi^\vee_+$, so $w = 1$ since $w \in W^0_{P, B}$. So $\lambda = -\alpha_i^\vee$, contradicting $\lambda \leq - \alpha_j^\vee$. So we must have $w\lambda = -\alpha_i^\vee - \alpha_j^\vee$, and in particular $w^{-1}(\alpha_i^\vee + \alpha_j^\vee) \in \Phi_+^\vee$.

If $(G, P, \mu)$ is not of type $A$, then $w^{-1}(\alpha_k^\vee) \in \Phi_+^\vee$ for $\alpha_k \neq \alpha_i$ (since $w \in W^0_{P, B}$ and $t(P) = \{\alpha_i\}$) so Lemma \ref{lem:weylgroupelements} implies that $w \in W_{L_1}$. If $(G, P, \mu)$ is of type $A$, then $w^{-1}(\alpha_k^\vee) \in \Phi_+^\vee$ for $\alpha_k \neq \alpha_i, \alpha_j$. If $w^{-1}(\alpha_j^\vee) \in \Phi_+^\vee$ then $w \in W_{L_1}$ by Lemma \ref{lem:weylgroupelements} again. Otherwise, we must have $w^{-1}(\alpha_i^\vee) \in \Phi_+^\vee$ and hence
\[ w \in \{s_{i + 1} s_{i + 2} \cdots s_k \mid i < k \leq l \}\]
by Lemma \ref{lem:weylgroupelements}. But this implies that $\lambda = w^{-1}(-\alpha_i^\vee - \alpha_j^\vee) = w^{-1}(-\alpha_i^\vee - \alpha_{i + 1}^\vee) = -\alpha_i^\vee$, contradicting $\lambda \leq -\alpha_j^\vee$, so we are done.
\end{proof}

\begin{lem} \label{lem:weylgroupelements}
Let $(M, \Psi, M^\vee, \Psi^\vee)$ be a root datum with Weyl group $W(\Psi)$, and let $\Gamma \subseteq \Psi$ be a complete set of positive simple roots. Let $\beta_j \in \Gamma$ be a simple root, and let $c \in \pi_0(\Gamma\setminus \{\beta_j\})$ be a connected component of the Dynkin diagram of $\Gamma \setminus \{\beta_j\}$ of type $A_n$ such that $\beta_j$ is adjacent to one end of $c$. Let $\beta_{c, 1}, \ldots, \beta_{c, n} \in \Gamma$ denote the nodes of $c$, labelled so that $\beta_{c, k}$ is adjacent to $\beta_{c, k + 1}$ for all $k$ and $\beta_{c, n}$ is adjacent to $\beta_j$, and let
\[ \Sigma = \{ w \in W(\Psi) \mid w^{-1}\beta_k^\vee \in \Psi^\vee_+ \;\text{for all}\; \beta_k \in \Gamma \setminus \{\beta_{c, n}\}\; \text{and} \; w^{-1}(\beta_{c, n}^\vee + \beta_j^\vee) \in \Psi^\vee_+\},\]
Then
\[ \Sigma = \{1\} \cup \{s_{c, n}s_{c, n - 1} \cdots s_{c, k} \mid 1 \leq k \leq n\}\]
where $s_{c, k} \in W(\Psi)$ is the reflection in the root $\beta_{c, k}$
\end{lem}
\begin{proof}
First note that an easy inspection shows that
\[ \{1\} \cup \{s_{c, n}s_{c, n - 1} \cdots s_{c, k}\mid 1 \leq k \leq n\} \subseteq \Sigma,\]
so it suffices to prove the reverse inclusion.

We prove the claim by induction on $n \geq 1$. Suppose that $w \in \Sigma$. Then either $w = 1$ or $w^{-1}\beta_{c, n} \in \Psi_-$. In the second case, we see that $(s_{c, n}w)^{-1}\beta_k^\vee \in \Psi^\vee_+$ for $\beta_k \in \Gamma \setminus \{\beta_{c, n - 1}\}$ and $(s_{c, n}w)^{-1}(\beta_{c, n - 1}^\vee + \beta_{c, n}^\vee) \in \Psi^\vee_+$ if $n > 1$. So either $n = 1$ and $w \in \{1, s_{c, n}\}$, or $n > 1$ and by induction we have
\[ s_{c, n} w \in \{s_{c, n - 1}\cdots s_{c, k} \mid 1 \leq k \leq n - 1\},\]
and hence
\[ w \in \{1\} \cup \{s_{c, n}s_{c, n - 1} \cdots s_{c, k} \mid 1 \leq k \leq n\}.\]
This proves the lemma.
\end{proof}

\begin{prop} \label{prop:subregp1hom}
There exists a surjective homomorphism
\[ \pi_{P_1} \colon P_1 \longrightarrow GL_{n_0 + 1} \]
such that $\pi_{P_1}^{-1}(R_{n_0 + 1}) = P \cap P_1$ and $\pi_{P_1}^{-1}(Q^{n_0 + 1}_k) = P_k$ for $1 \leq k \leq n_0 + 1$, and such that the induced map $T = T_{P_{n_0 + 1}} \to Q^{n_0 + 1}_{n_0 + 1}$ is given on cocharacters by
\begin{align*}
\mb{X}_*(T) &\longrightarrow \mb{X}_*(T_{Q^{n_0 + 1}_{n_0 + 1}}) \\
\alpha_{c_0, k}^\vee &\longmapsto e_k^* - e_{k + 1}^* \\
\alpha_j^\vee &\longmapsto e_{n_0 + 1}^* \\
\alpha_p^\vee &\longmapsto 0, \quad \text{if}\;\; \alpha_p \notin \{\alpha_{c_0, 1}, \ldots, \alpha_{c_0, n_0}, \alpha_j\}.
\end{align*}
\end{prop}
\begin{proof}
Since the Dynkin diagram $\Delta \setminus t(P_1)$ has exactly one connected component of type $A_{n_0}$, Proposition \ref{prop:typealevi} gives an embedding
\begin{equation} \label{eq:subregularlevirep1}
L_1 \longhookrightarrow GL_{n_0 + 1} \times \mb{G}_m^{n_1}.
\end{equation}
Let $\pi_{L_1}$ be the composition of \eqref{eq:subregularlevirep1} with the projection to the first factor, and let $\pi_{P_1}$ be the composition of $\pi_{L_1}$ with the quotient $P_1 \to L_1$. The remaining claims can now be checked routinely using the explicit isomorphism of Proposition \ref{prop:typealevi}.
\end{proof}

By construction, the morphism $C_{1, n_0 + 1} \to C_{1, n_0 + 1}^{GL_{n_0 + 1}}$ factors through a morphism
\[ C_{1, n_0 + 1} \longrightarrow Y^{-\alpha_j^\vee} \times_{Y_{Q^{n_0 + 1}_{n_0 + 1}}^{-e_{n_0 + 1}^*}} (C_{1, n_0 + 1}^{GL_{n_0 + 1}} \times_S E) = Y \times_{Y_{Q^{n_0 + 1}_{n_0 + 1}}} (C_{1, n_0 + 1}^{GL_{n_0 + 1}} \times_S E),\]
where the morphism $C_{1, n_0 + 1}^{GL_{n_0 + 1}} \times_S E \to Y_{Q^{n_0 + 1}_{n_0 + 1}}$ is given by the natural morphism to $Y_{Q^{n_0 + 1}_{n_0 + 1}}^{-e_{n_0 + 1}^*} \times_S E$ composed with $(y, x) \mapsto y + e_{n_0 + 1}^*(x)$. Composing with the morphism \eqref{eq:glnbruhatpoint2} gives a morphism $C_{1, n_0 + 1} \to Y \times_S E \times_S E$ and hence a morphism
\begin{align} 
C_{1, n_0 + 1} \longrightarrow Y \times_S E \times_S E &\longrightarrow Y \times_S \mrm{Pic}^0_S(E) \label{eq:bruhatcellcurvecomparison1} \\
(y, x_i, x_j) &\longmapsto (y, x_j - x_i) \nonumber
\end{align}
over $Y$. We remark that for the image of a stable map with two rational components of degree $\alpha_i^\vee$ and $\alpha_j^\vee$, $x_i$ and $x_j$ above are just the points of attachment of the two rational curves.

For $1 \leq p \leq n_0 + 1$, we let
\[ M_p \subseteq C_{1, n_0 + 1}\]
be the closed substack given by the fibre product
\[
\begin{tikzcd}
M_p \ar[r] \ar[d] & C_{1, n_0 + 1} \ar[d, "\eqref{eq:bruhatcellcurvecomparison1}"] \\
Y \ar[r, "\theta_p"] & Y \times_S \mrm{Pic}^0_S(E).
\end{tikzcd}
\]

\begin{prop} \label{prop:subregbruhatcomparison}
For all $1 \leq k \leq n_0$, the morphism $D_{k + 1} \to D_k$ restricts to isomorphisms
\[ D_{k + 1} \times_Z (Z \setminus Z_0) \overset{\sim}\longrightarrow D_k \times_Z (Z \setminus Z_0), \quad C_{k + 1, n_0 + 1} \overset{\sim}\longrightarrow C_{k, n_0 + 1} \quad \text{and} \quad C_{k + 1, p} \overset{\sim}\longrightarrow C_{k, p} \]
for $1 \leq p < k$, and a morphism
\[ C_{k + 1, k} \longrightarrow M_k \subseteq C_{k, n_0 + 1} \cong C_{1, n_0 + 1}\]
that realises $C_{k + 1, k}$ as an $\mb{A}^1$-bundle over $M_k$.
\end{prop}
\begin{proof}
Chasing through the definitions, we have
\[ M_k = C_{1, n_0 + 1} \times_{C_{1, n_0 + 1}^{GL_{n_0 + 1}}} M_k^{GL_{n_0 + 1}}.\]
Since the diagram
\[
\begin{tikzcd}
D_{k + 1} \times_Z Z_0 \ar[r] \ar[d] & X^{n_0 + 1}_{k + 1, rig} \ar[d] \\
D_k \times_Z Z_0 \ar[r] & X^{n_0 + 1}_{k, rig}
\end{tikzcd}
\]
is Cartesian, Proposition \ref{prop:glnbruhatcomparison} implies everything except the claim that
\begin{equation} \label{eq:subregbruhatcomparison1}
 D_{k + 1} \times_Z (Z\setminus Z_0) \longrightarrow D_k \times_Z (Z \setminus Z_0)
\end{equation}
is an isomorphism. To prove this, first note that every $G$-bundle in the image of $D_k \times_Z (Z \setminus Z_0)$ is regular unstable, as follows from comparing the codimensions of its $Z(L)_{rig}$-orbit in $Z$ and in $\bun_{G, rig}/E$. Since $\mrm{KM}_{B, G}^{-\alpha_j^\vee} \to \mrm{KM}_{P_k, G}^{-\alpha_j^\vee}$ is surjective for all $k$ by \cite[Proposition 3.6.4]{davis19a}, all such bundles necessarily have Harder-Narasimhan reduction to the parabolic $Q$ of type $t(Q) = \{\alpha_j\}$ by \cite[Lemma 4.3.4]{davis19}. So the morphism to $\bun_{G, rig}$ factors as
\[ D_k \times_Z (Z \setminus Z_0) \longrightarrow \bun_{Q, rig}^{ss, -\alpha_j^\vee} \longhookrightarrow \bun_{G, rig}.\]
The argument of the proof of \cite[Proposition 4.3.8]{davis19}, shows that we have isomorphisms
\[ D_k \times_Z (Z \setminus Z_0) \cong Y \times_{Y_{P_k}}(\bun_{M \cap P_k, rig}^{-\alpha_j^\vee} \times_{\bun_{M, rig}^{-\alpha_j^\vee}} \bun_{Q, rig}^{ss, -\alpha_j^\vee} \times_{\bun_{G, rig}} (Z \setminus Z_0) \times_S E) \]
for all $k$, where $M$ is the Levi factor of $Q$ (note that the first two factors in the parentheses above are just the Bruhat cell $C^{1, -\alpha_j^\vee}_{P_k, Q}$). So Proposition \ref{prop:typealevi} and \cite[Lemma 4.3.7]{davis19} show that \eqref{eq:subregbruhatcomparison1} is an isomorphism as claimed.
\end{proof}

\begin{lem} \label{lem:bruhatcellcurvecomparisonsmooth}
The morphism \eqref{eq:bruhatcellcurvecomparison1} is smooth with connected fibres.
\end{lem}
\begin{proof}
From the construction, we have
\[ C_{1, n_0 + 1} = D_1 \times_Z Z_0 = Y^{-\alpha_j^\vee} \times_{Y_{P_1}^{-\alpha_j^\vee}} \bun_{L \cap P_1, rig}^{-\alpha_i^\vee - \alpha_j^\vee} \times_{\bun_{L, rig}} Z_0 \times_S E.\]
There is an isomorphism
\begin{align}
Y \times_{Y_{P_1}} Y_{L \cap P_1} &\overset{\sim}\longrightarrow Y \times_S \mrm{Pic}^0_S(E) \label{eq:bruhatcellcurvecomparisonsmooth1} \\
(y_1, y_2) &\longmapsto (y_1, \varpi_i(y_2) - \varpi_i(y_1)). \nonumber
\end{align}
Chasing through the definitions of the various morphisms involved, we deduce that there is a pullback
\[
\begin{tikzcd}
C_{1, n_0 + 1} \ar[rr, "\eqref{eq:bruhatcellcurvecomparison1}"] \ar[d] & & Y \times_S \mrm{Pic}^0_S(E) \ar[d] \\
\bun_{L \cap P_1, rig}^{-\alpha_i^\vee - \alpha_j^\vee} \times_{\bun_{L, rig}} Z_0 \times_S E \ar[r] & Y_{L \cap P_1}^{-\alpha_i^\vee -\alpha_j^\vee} \times_S E \ar[r] & Y_{L \cap P_1},
\end{tikzcd}
\]
where the morphisms $Y \times_S \mrm{Pic}^0_S(E) \to Y_{L \cap P_1}$ is the composition of the inverse to \eqref{eq:bruhatcellcurvecomparisonsmooth1} with the natural projection.

It therefore suffices to show that the composition $f$ of the first two morphisms in the bottom row is smooth with connected fibres. Note that the morphism
\[ Y_{L \cap P_1}^{-\alpha_i^\vee - \alpha_j^\vee} \times_S E \longrightarrow Y_{L \cap P_1}\]
naturally identifies $Y_{L \cap P_1}$ with the quotient $(Y_{L \cap P_1}^{-\alpha_i^\vee - \alpha_j^\vee} \times_S E)/E$ by the diagonal action of $E$ by translations. So we can identify $f$ with the composition of the middle vertical arrows in the diagram
\begin{equation} \label{eq:bruhatcellcurvecomparison3}
\begin{tikzcd}
& \bun_{L \cap P_1, rig}^{-\alpha_i^\vee - \alpha_j^\vee} \times_{\bun_{L, rig}^\mu} Z_0 \times_S E \ar[r] \ar[d] & Z_0 \ar[d] \\
\bun_{L \cap P_1, rig}^{-\alpha_i^\vee - \alpha_j^\vee} \times_S E \ar[r] \ar[d] & (\bun_{L \cap P_1, rig}^{-\alpha_i^\vee - \alpha_j^\vee} \times_S E)/E \ar[r] \ar[d] & \bun_{L, rig}^\mu/E \\
Y_{L \cap P_1}^{-\alpha_i^\vee - \alpha_j^\vee} \times_S E \ar[r] & (Y_{L \cap P_1}^{-\alpha_i^\vee - \alpha_j^\vee} \times_S E)/E.
\end{tikzcd}
\end{equation}
The vertical arrow on the left in \eqref{eq:bruhatcellcurvecomparison3} is smooth, and has connected fibres since the semisimple part of $L \cap P_1$ is simply connected. The vertical arrow on the right in \eqref{eq:bruhatcellcurvecomparison3} is smooth with connected fibres by assumption. Since both squares are Cartesian, and the horizontal arrows in the square on the left are faithfully flat, it follows that both vertical arrows in the middle are smooth with connected fibres, and hence so is their composition $f$.
\end{proof}

\begin{lem} \label{lem:bruhatcellcurvecomparison}
The morphism \eqref{eq:bruhatcellcurvecomparison1} is an isomorphism.
\end{lem}
\begin{proof}
Observe that the cell
\[ C_{n_0 + 1, n_0 + 1} \subseteq D_{n_0 + 1} = \mrm{KM}_{B, G, rig}^{-\alpha_j^\vee} \times_{\bun_{G, rig}} Z \times_S E \]
is equal to the locus of singular domain curves, and is therefore a divisor in $D_{n_0 + 1}$ flat over $Y$. Since $D_{n_0 + 1} \to Y$ has relative dimension $2$, $C_{n_0 + 1, n_0 + 1} \to Y$ therefore has relative dimension $1$. So by Lemma \ref{lem:bruhatcellcurvecomparisonsmooth}, \eqref{eq:bruhatcellcurvecomparison1} is a smooth proper morphism with connected fibres and finite relative stabilisers between smooth stacks of the same dimension over $S$. Since $C_{n_0 + 1, n_0 + 1} \to S$ is representable over the dense open substack where $Z_0 \to S$ is representable (note that $\tilde{Z} \to Z$ is representable, and even projective since all stable maps involved in the definition have trivial automorphism group), so is $C_{n_0 + 1, n_0 + 1} \to Y \times_S \mrm{Pic}^0_S(E)$. Since $Y \times_S \mrm{Pic}^0_S(E) \to S$ has irreducible fibres, \eqref{eq:bruhatcellcurvecomparison1} is therefore surjective, so by Lemma \ref{lem:smoothconfibiso} below, it is an isomorphism as claimed.
\end{proof}

\begin{lem} \label{lem:smoothconfibiso}
Let $X$ and $X'$ be stacks that are smooth and of the same dimension over $S$, and let $f \colon X \to X'$ be a smooth surjective proper morphism with connected fibres and finite relative stabilisers. Assume that there exists some open set $U \subseteq X$ that is dense in every fibre of $X \to S$ such that $f|_U$ is representable. Then $f$ is an isomorphism.
\end{lem}
\begin{proof}
First note that $f|_U \colon U \to X'$ is \'etale and representable with connected fibres, and hence an open immersion. Moreover, the morphism $X \times_{X'} X \to X$ is smooth with connected fibres, so the preimage of $U$ under either projection is dense. So the diagonal $X \to X \times_{X'} X$, which is finite by assumption, is an isomorphism over the dense open subset $U$, and hence surjective. Since $X \times_{X'} X$ is smooth over $S$, and hence normal, it follows that $X \to X \times_{X'} X$ is an isomorphism. Since $f$ is smooth and surjective, by flat descent it follows that $f \colon X \to X'$ is also an isomorphism as claimed.
\end{proof}

\begin{lem} \label{lem:blowuprecognition}
Let $U$ be a regular stack, let $X \to U$ and $X' \to U$ be smooth representable morphisms of relative dimension $2$, and let $f \colon X \to X'$ be a projective morphism over $U$. Suppose that there exists a section $g \colon U \to X'$ such that $f^{-1}(X' \setminus g(U)) \to X' \setminus g(U)$ is an isomorphism, and such that every fibre of $f$ over a point in $g(U)$ is an irreducible curve. Then $f$ is the blowup of $X'$ along $g(U)$.
\end{lem}
\begin{proof}
Since the claim is local in the smooth topology on $U$ and in the \'etale topology on $X'$, we can reduce to the case where $X' \to U$ is a smooth morphism of schemes with $U$ connected and regular.

First note that the underlying reduced scheme $D$ of the exceptional locus $f^{-1}(g(U))$ is an integral closed subscheme of codimension $1$ in a regular scheme, and hence a Cartier divisor. Since $X$ and $X'$ are smooth over $U$ and $f$ is an isomorphism outside $D$, we therefore have $K_{X/U} = f^*K_{X'/U}(nD)$ for some $n > 0$. If $k$ is any field and $u \colon \spec k \to U$ is a $k$-point, we have $D|_{X_u} = m_u C_u$ for some $m_u > 0$, where $C_u \subseteq X_u$ is the irreducible curve contracted under $f$, and hence, by adjunction
\[ -2 \leq \deg K_{C_u} = (m_u n + 1)C_u^2.\]
Since $C_u^2 < 0$, we deduce that $m_u = n = 1$, $C_u^2 = -1$, $\deg K_{C_u} = -2$, and hence that $C_u$ is a smooth rational curve. In particular, by Castelnuovo's theorem, $f_u \colon X_u \to X'_u$ is the blowup at $g(u)$.

Now let $\eta \colon \spec K \to U$ be the generic point of $U$. We have shown that on the generic fibre $f_\eta \colon X_\eta \to X'_\eta$ is the blowup along $g(\eta)$, so the same must be true on some dense open set $U$ with complement $V$. So we get an isomorphism
\[h \colon X \setminus f^{-1}(g(V)) \overset{\sim}\longrightarrow \tilde{X}' \setminus \pi^{-1}(g(V))\]
over $X'$, where $\pi \colon \tilde{X}' \to X'$ is the blowup of $X'$ along $g(U)$. Since $f$ is projective and is an isomorphism outside $D$, it follows that either $D$ or $-D$ is $f$-ample. Since $D\cdot C_u = (C_u^2)_{X_u} = -1$ for all points $u \colon \spec k \to X'$, it follows that $-D$ is $f$-ample. But $h$ is an isomorphism in codimension $1$ between regular schemes projective over $X'$, $h(D \setminus f^{-1}(g(V))) = \pi^{-1}(g(U)) \setminus \pi^{-1}(g(V))$, and $-\pi^{-1}(g(U))$ is $f$-ample, so
\[ X \overset{\sim}\longrightarrow \mrm{Proj}_{X'} \bigoplus_{d \geq 0} f_*\mc{O}(-dD) \cong \mrm{Proj}_{X'} \bigoplus_{d \geq 0} \pi_*\mc{O}(-d\pi^{-1}(g(U))) \overset{\sim}\longleftarrow \tilde{X}',\]
which proves that $X$ is the blowup as claimed.
\end{proof}

\subsection{The divisor $D_{\alpha_i^\vee + \alpha_j^\vee}(Z)$} \label{subsection:dalphaij}

In this subsection, we prove the following proposition, which is essentially Theorem \ref{thm:introsubregularresolutions} \eqref{itm:introsubregularresolutions3}.

\begin{prop} \label{prop:subregularresolutions3}
Assume we are in the setup of Proposition \ref{prop:subregularresolutions1}. Then every fibre of the morphism
\[ D_{\alpha_i^\vee + \alpha_j^\vee}(Z) \longrightarrow Y \]
is isomorphic to the Hirzebruch surface $\mb{F}_{d - 1}$.
\end{prop}
\begin{proof}
By Proposition \ref{prop:subregularresolutions1}, we have $D_\lambda(Z) = \emptyset$ for all $\lambda > \alpha_i^\vee + \alpha_j^\vee$, so any stable map parametrised by a point in $D_{\alpha_i^\vee + \alpha_j^\vee}(Z)$ must be the union of a section of the relevant $G/B$-bundle of degree $-\alpha_i^\vee - \alpha_j^\vee$ and a single connected stable map of genus $0$ and degree $\alpha_i^\vee + \alpha_j^\vee$ to a fibre of the $G/B$-bundle. We deduce that
\[ D_{\alpha_i^\vee + \alpha_j^\vee}(Z) \cong \eta \times^{B/Z(G)} \bar{M}_{0, 1}^+(G/B, \alpha_i^\vee + \alpha_j^\vee),\]
where
\[ \eta \longrightarrow \bun_{B, rig}^{-\alpha_i^\vee - \alpha_j^\vee} \times_{\bun_{G, rig}} Z \times_S E \]
is the pullback of the universal $B/Z(G)$-bundle on $\bun_{B, rig} \times_S E$, and
\[ \bar{M}_{0, 1}^+(G/B, \alpha_i^\vee + \alpha_j^\vee) \]
is the moduli space of $1$-pointed stable maps of genus $0$ and degree $\alpha_i^\vee + \alpha_j^\vee$ sending the marked point to the base point $B/B \in G/B$.The morphism $D_{\alpha_i^\vee + \alpha_j^\vee}(Z) \to Y$ factors through
\begin{equation} \label{eq:subregularresolutions3:1}
 \bun_{B, rig}^{-\alpha_i^\vee - \alpha_j^\vee} \times_{\bun_{G, rig}} Z \times_S E \xrightarrow{\Delta_E} \bun_{B, rig}^{-\alpha_i^\vee - \alpha_j^\vee} \times_{\bun_{G, rig}} Z \times_S E \times_S E = C_{n_0 + 1, n_0 + 1} \longrightarrow Y, \end{equation}
where the last morphism is the map $C_{n_0 + 1, n_0 + 1} \to C_{1, n_0 + 1}$ composed with \eqref{eq:bruhatcellcurvecomparison1} and the projection to $Y$. Proposition \ref{prop:subregbruhatcomparison} and Lemma \ref{lem:bruhatcellcurvecomparison} identify the last morphism with $Y \times_S \mrm{Pic}^0_S(E) \to Y$ and the first with the zero section. So \eqref{eq:subregularresolutions3:1} is an isomorphism, so we can identify $D_{\alpha_i^\vee + \alpha_j^\vee}(Z)$ with the morphism
\[ \eta \times^{B/Z(G)} \bar{M}_{0, 1}^+(G/B, \alpha_i^\vee + \alpha_j^\vee) \to Y \]
for some $B/Z(G)$-bundle $\eta \to Y$. The proposition now follows from Proposition \ref{prop:subregqflagcurves} below.
\end{proof}

\begin{prop} \label{prop:subregqflagcurves}
With notation as in the proof of Proposition \ref{prop:subregularresolutions3}, there is an isomorphism
\[ \bar{M}^+_{0, 1}(G/B, \alpha_i^\vee + \alpha_j^\vee) \cong \mb{F}_{d - 1}.\]
such that the closure of the locus of stable maps with dual graph
\[
\begin{tikzpicture}[every label/.style=dlabel]
\draw (0, 0) node [dnode, label=below:{\alpha_i^\vee}] {} -- (1, 0) node [dnode, label=below:{\alpha_j^\vee}] {} -- (1.5, 0);
\draw (3, 0) node {(resp.};
\draw (4, 0) node [dnode, label=below:{\alpha_j^\vee}] {} -- (5, 0) node [dnode, label=below:{\alpha_i^\vee}] {} -- (5.5, 0);
\draw (6, 0) node {)};
\end{tikzpicture}
\]
is a fibre of $\mb{F}_{d - 1} \to \mb{P}^1$ (resp.\ a section $\mb{P}^1 \to \mb{F}_{d - 1}$ with self-intersection $1 - d$).
\end{prop}

An important role in the proof of Proposition \ref{prop:subregqflagcurves} is played by the Schubert varieties in $G/B$. Given $w \in W$, recall that the \emph{Schubert variety associated to $w$} is the closed subvariety
\[ X_w = \overline{BwB/B} \subseteq G/B.\]
In what follows, we write $Q_i, Q_j \subseteq G$ for the standard minimal parabolics of types $t(Q_i) = \Delta \setminus \{\alpha_i\}$ and $t(Q_j) = \Delta \setminus \{\alpha_j\}$.

\begin{lem} \label{lem:2dimschubert}
There are isomorphisms
\[ X_{s_is_j} \cong \mb{F}_d, \quad \text{(resp.} \quad X_{s_js_i} \cong \mb{F}_1\;\;\text{)}\]
such that $X_{s_j}$ is identified with a fibre of $\mb{F}_d \to \mb{P}^1$ (resp., the unique section $\mb{P}^1 \to \mb{F}_1$ of self-intersection $-1$) and $X_{s_i}$ is identified with the unique section $\mb{P}^1 \to \mb{F}_d$ of self-intersection $-d$ (resp., a fibre of $\mb{F}_1 \to \mb{P}^1$).
\end{lem}
\begin{proof}
We prove the claim for $X_{s_is_j}$; the proof for $X_{s_js_i}$ is identical after noting that $\langle \alpha_i, \alpha_j^\vee \rangle = -1$.

There is an isomorphism
\[ SL_2 \times^{B_{SL_2}, \rho_{\alpha_i}} Q_j/B = Q_i \times^B Q_j/B \overset{\sim}\longrightarrow X_{s_is_j},\]
given by multiplication, where $B_{SL_2} \subseteq SL_2$ is the Borel subgroup of lower triangular matrices, and $\rho_{\alpha_i}\colon SL_2 \to G$ is the root homomorphism corresponding to $\alpha_i$. We also have an isomorphism of $Q_j$-varieties $Q_j/B \cong \mb{P}(V^\vee)$, where $V$ is the $Q_j$-representation $V = \mrm{Ind}_B^{Q_j}(\mb{Z}_{\varpi_j})$, and an exact sequence
\[ 0 \longrightarrow \mb{Z}_{\varpi_j - \alpha_j} \longrightarrow V \longrightarrow \mb{Z}_{\varpi_j} \longrightarrow 0\]
of $B$-representations, which splits uniquely as an exact sequence of $B_{SL_2}$-representations. So we have
\[ X_{s_is_j} = SL_2 \times^{B_{SL_2}} \mb{P}(V^\vee) = \mb{P}_{\mb{P}^1}(\mc{O}(-\langle \varpi_j, \alpha_i^\vee \rangle) \oplus \mc{O}(-\langle \varpi_j - \alpha_j, \alpha_i^\vee\rangle)) = \mb{P}_{\mb{P}^1}(\mc{O} \oplus \mc{O}(-d))) = \mb{F}_{d}.\]
(Recall that $\mb{P}(-)$ denotes the projective space of $1$-dimensional \emph{sub}spaces or rank $1$ \emph{sub}bundles of a vector space or vector bundle.) The identifications of $X_{s_i} = Q_i/B$ and $X_{s_j} = Q_j/B$ under this isomorphism follow immediately.
\end{proof}

\begin{lem} \label{lem:2dimschubertblowdown}
The partial Schubert variety $X_{s_is_j}/Q_i = \overline{Bs_is_jQ_i/Q_i} \subseteq G/Q_i$ is isomorphic to the projective cone $\hat{\mb{P}}^1_d$ on $\mb{P}^1$ of degree $d$, and the morphism
\begin{equation} \label{eq:2dimschubertblowdown1}
X_{s_is_j} \longrightarrow X_{s_is_j}/Q_i
\end{equation}
is the blowup of $X_{s_is_j}/Q_i$ at the origin $Q_i/Q_i$.
\end{lem}
\begin{proof}
First note that the morphisms $Bs_is_jB/B \to Bs_is_jQ_i/Q_i$ and $Bs_jB/B \to Bs_jQ_i/Q_i$ are isomorphisms. So \eqref{eq:2dimschubertblowdown1} is birational and finite outside $Q_i/Q_i$, and hence an isomorphism outside $Q_i/Q_i$ since partial Schubert varieties are always normal. Since the preimage of $Q_i/Q_i$ under \eqref{eq:2dimschubertblowdown1} is $Q_i/B = X_{s_i}$, using normality of $X_{s_is_j}/Q_i$ and of $\hat{\mb{P}}^1_d$, we can conclude from Lemma \ref{lem:2dimschubert} that \eqref{eq:2dimschubertblowdown1} can be identified with the morphism
\[ \mb{F}_{d} \longrightarrow \hat{\mb{P}}^1_{d} \]
contracting the curve of self-intersection $-d$. But this is indeed the blowup at the cone point, so we are done.
\end{proof}

\begin{lem} \label{lem:subregqpartialschubertcurves}
There is a $Q_i$-equivariant isomorphism
\[ \bar{M}^+_{0, 1}(X_{s_is_j}/Q_i, \alpha_j^\vee) \cong Q_i/B \cong \mb{P}^1,\]
identifying the universal stable map with
\begin{equation} \label{eq:subregpartialschubertcurves1}
 Q_i \times^B Q_j/B \longrightarrow X_{s_is_j} \longrightarrow X_{s_is_j}/Q_i.
\end{equation}
\end{lem}
\begin{proof}
Assume that $U$ is a scheme and $(f \colon C \to X_{s_is_j}/Q_i, x\colon U \to C)$ is a $1$-pointed stable map over $U$ of degree $\alpha_j^\vee$ sending $x$ to the base point. We need to show that there is a unique morphism $U \to Q_i/B$ such that $f$ and $x$ are the pullbacks of \eqref{eq:subregpartialschubertcurves1} and the canonical section $Q_i/B = Q_i \times^B B/B \to Q_i \times^B Q_j/B$.

We first claim that $C \to U$ is smooth and that every geometric fibre of $f^{-1}(Q_i/Q_i) \to U$ is a reduced point. Since $f^{-1}(Q_i/Q_i) \to U$ has a section $x$, it then follows that it is an isomorphism.

To prove the claim, fix a geometric point $u \colon \spec k \to U$, and consider the stable map $f_u \colon C_u \to (X_{s_is_j}/Q_i)_k$. Since $\alpha_j^\vee$ is not the sum of two nonzero effective curve classes, it follows that $C_u$ is irreducible, hence smooth over $\spec k$. Since this holds for all geometric points, $C \to U$ is smooth as claimed, and $f_u^{-1}(Q_i/Q_i)$ is a Cartier divisor on $C_u$. So by Lemmas \ref{lem:2dimschubert} and \ref{lem:2dimschubertblowdown}, $f_u$ lifts to a morphism $\bar{f}_u \colon C_u \to (X_{s_is_j})_k \cong (\mb{F}_{d})_k$ such that $C_u \cdot X_{s_i} > 0$ and $C_u \cdot (dX_{s_j} + X_{s_i}) = 1$. (Note that $dX_{s_j} + X_{s_i}$ is linearly equivalent to the pullback of $\mc{L}_{\varpi_j}$.) Since $d > 0$, it follows that $C_u \cdot X_{s_i} = 1$ and $C_u \cdot X_{s_j} = 0$. In particular, $f_u^{-1}(Q_i/Q_i) = C_u \cap X_{s_i}$ is a reduced closed point on $C_u$, so $f_u^{-1}(Q_i/Q_i) \cong \spec k$ as claimed.

Since $f^{-1}(Q_i/Q_i) \subseteq C$ is a section of the smooth curve $C \to U$, it is a Cartier divisor, so by Lemma \ref{lem:2dimschubertblowdown}, $f$ lifts uniquely to a morphism $\bar{f} \colon C \to X_{s_is_j}$. Since the above argument shows that the composition $\bar{f} \colon C \to X_{s_is_j} = Q_i \times^B Q_j/B \to Q_i/B$ has degree $0$ on every fibre, this descends to a unique morphism $U \to Q_i/B$. The induced morphism
\begin{equation} \label{eq:subregqpartialschubertcurves1}
C \longrightarrow U \times_{Q_i/B} (Q_i \times^B Q_j/B)
\end{equation}
has degree $1$ on every fibre and is therefore an isomorphism. Since \eqref{eq:subregqpartialschubertcurves1} sends the section $x$ to the section $Q_i/B \to Q_i \times^B Q_j/B$ (as both are the preimage of $Q_i/Q_i \subseteq X_{s_is_j}/Q_i$), this proves the lemma.
\end{proof}

\begin{proof}[Proof of Proposition \ref{prop:subregqflagcurves}]
For the sake of brevity, write
\[ M = \bar{M}_{0, 1}^+(G/B, \alpha_i^\vee + \alpha_j^\vee).\]
We first claim that $M$ is connected. To see this, observe that $B$ acts on $M$, that any $B$-fixed point corresponds to a stable map factoring through $X_{s_i} \cup X_{s_j} \subseteq G/B$, and that there is a unique such pointed stable map of class $\alpha_i^\vee + \alpha_j^\vee$ defined over $k$ for any algebraically closed field $k$. Since every connected component of $M$ must have at least one $B$-fixed point over every algebraically closed field, connectedness of $M$ follows immediately.

We now compute the closed subscheme
\[ M' = \bar{M}_{0, 1}^+(X_{s_is_js_i}, \alpha_i^\vee + \alpha_j^\vee) \subseteq M \]
consisting of stable maps factoring through the Schubert variety $X_{s_is_js_i}$. We will show that $M' \cong \mb{F}_d$ is smooth and projective of relative dimension $2$ over $\spec \mb{Z}$. Since the same is true for $M$ and $M$ is connected, it follows that $M' = M$.

Since $X_{s_is_js_i}/Q_i = X_{s_is_j}/Q_i$, by Lemma \ref{lem:subregqpartialschubertcurves} we have a morphism
\[ M' \longrightarrow \bar{M}_{0, 1}^+(X_{s_is_j}/Q_i, \alpha_j^\vee) \cong Q_i/B = \mb{P}^1\]
sending a stable map to the stabilisation of its composition with $G/B \to G/Q_i$. The pullback of the universal domain curve of $\bar{M}_{0, 1}^+(X_{s_is_j}/Q_i, \alpha_j^\vee)$ along $X_{s_is_js_i} \to X_{s_is_j}/Q_i$ is
\[ X_{s_is_js_i} \times_{X_{s_is_j}/Q_i}(Q_i \times^B Q_j/B) = G/B \times_{G/Q_i} (Q_i \times^B Q_j/B),\]
which is identified with the Bott-Samelson variety $\tilde{X}_{s_is_js_i}$ via
\begin{align*}
\tilde{X}_{s_is_js_i} = Q_i \times^B Q_j \times^B Q_i/B &\overset{\sim}\longrightarrow G/B \times_{G/Q_i} (Q_i \times^B Q_j/B) \\
(g_1, g_2, g_3B) &\longmapsto (g_1g_2g_3B, (g_1, g_2B)).
\end{align*}
So we can identify $M'$ with the relative space of stable maps
\[ M' \cong \bar{M}_{0, 1, Q_i/B}^+(\tilde{X}_{s_is_js_i}, \alpha_i^\vee + \alpha_j^\vee),\]
where $\bar{M}_{0, 1, Q_i/B}^+(\tilde{X}_{s_is_js_i}, \alpha_i^\vee + \alpha_j^\vee)$ is the fibre product
\[
\begin{tikzcd}
\bar{M}_{0, 1, Q_i/B}^+(\tilde{X}_{s_is_js_i}, \alpha_i^\vee + \alpha_j^\vee) \ar[r] \ar[d] & Q_i/B \ar[d, "\sigma"] \\
\bar{M}_{0, 1, Q_i/B}(\tilde{X}_{s_is_js_i}, \alpha_i^\vee + \alpha_j^\vee) \ar[r] &  \tilde{X}_{s_is_js_i}.
\end{tikzcd}
\]
Here $\sigma$ is the section defined by $Q_i/B \cong m^{-1}(B/B) \to \tilde{X}_{s_is_js_i}$, for $m \colon \tilde{X}_{s_is_js_i} \to G/B$ the natural morphism given by multiplication. Note that $\bar{M}_{0, 1, Q_i/B}(\tilde{X}_{s_is_js_i}, \alpha_i^\vee + \alpha_j^\vee)$ is naturally identified with the universal domain curve over the space $\bar{M}_{0, Q_i/B}(\tilde{X}_{s_is_js_i}, \alpha_i^\vee + \alpha_j^\vee)$ of unpointed stable maps.

By Lemma \ref{lem:2dimschubert}, every fibre of $\tilde{X}_{s_is_js_j} \to Q_i/B$ is isomorphic to $\mb{F}_1 = X_{s_js_i} = Q_j \times^B Q_i/B$, and $\alpha_i^\vee + \alpha_j^\vee$ is the class $X_{s_i} + X_{s_j}$ of the $(-1)$-curve plus a fibre of $\mb{F}_1 \to \mb{P}^1$. Unpointed stable maps of class $\alpha_i^\vee + \alpha_j^\vee$ are the same things as closed subschemes with ideal sheaf $\mc{O}(-X_{s_i} - X_{s_j}) = m^*\mc{L}_{-\varpi_i}$. So we can identify $M_{0, Q_i/B}(\tilde{X}_{s_is_js_i}, \alpha_i^\vee + \alpha_j^\vee)$ with the Hilbert scheme $\mb{P}_{Q_i/B}(\pi_*m^*\mc{L}_{\varpi_i})$ and $M'$ with the closed subscheme
\[ M' = \mb{P}_{Q_i/B}(\ker \pi_*m^*\mc{L}_{\varpi_i} \to \sigma^*m^*\mc{L}_{\varpi_i})\]
of curves meeting $\sigma(Q_i/B)$, where $\pi \colon \tilde{X}_{s_is_js_i} \to Q_i/B$ is the natural projection.

It therefore remains to identify the vector bundle $\pi_*m^*\mc{L}_{\varpi_i}$ on $Q_i/B \cong \mb{P}^1$ and the morphism $\pi_*m^*\mc{L}_{\varpi_i} \to \sigma^*m^*\mc{L}_{\varpi_i} = \mc{O}$. It is clear from the identification $\tilde{X}_{s_is_js_i} = Q_i \times^B Q_j \times^B Q_i/B$ that $\pi_*m^*\mc{L}_{\varpi_i}$ is the $Q_i$-linearised vector bundle associated to the $B$-representation
\[ V = \mrm{Ind}_B^{Q_j} \mrm{Ind}_B^{Q_i} \mb{Z}_{\varpi_i}.\]
The representation $V$ has rank $3$, with weights $\varpi_i$, $\varpi_i - \alpha_i$ and $\varpi_i - \alpha_i - \alpha_j$, and restricting $V$ to a $B_{SL_2}$-representation via the root homomorphism $\rho_{\alpha_i}\colon SL_2 \to Q_i \subseteq G$, we have
\[ V = U \oplus \mb{Z}_{\langle \varpi_i - \alpha_i - \alpha_j, \alpha_i^\vee \rangle} = U \oplus \mb{Z}_{d - 1},\]
where $U$ is the standard representation of $SL_2$ and $\mb{Z}_{d - 1}$ is the rank $1$ $B_{SL_2}$-module of weight $d - 1$. So we get
\[ \pi_*m^*\mc{L}_{\varpi_i} = U \otimes \mc{O}_{\mb{P}^1} \oplus \mc{O}(d - 1).\]
Since $d > 0$, the kernel of
\[ \pi_*m^*\mc{L}_{\varpi_i} = \mc{O} \oplus \mc{O} \oplus \mc{O}(d - 1) \longrightarrow \mc{O} = \sigma^*m^*\mc{L}_{\varpi_i}\]
must be isomorphic to $\mc{O} \oplus \mc{O}(d - 1)$, which gives the desired isomorphism $M = M' \cong \mb{F}_{d - 1}$. 

Finally, to identify the loci of stable maps with given dual graphs in the statement of the proposition, notice that each closure is isomorphic to $\mb{P}^1$ (since there are unique curves of classes $\alpha_i^\vee$ and $\alpha_j^\vee$ through every point in $G/B$), and that the closure of curves with dual graph
\[
\begin{tikzpicture}[every label/.style=dlabel]
\draw (0, 0) node [dnode, label=below:{\alpha_i^\vee}] {} -- (1, 0) node [dnode, label=below:{\alpha_j^\vee}] {} -- (1.5, 0);
\end{tikzpicture} 
\]
is contracted under the map to $M_{0,1}^+(G/Q_i, \alpha_j^\vee)$, and is hence a fibre of $\mb{F}_{d - 1} \to \mb{P}^1$ as claimed. For the other statement, note that the map
\[ \pi_*m^*\mc{L}_{\varpi_i} \longrightarrow \pi'_*(\sigma')^*m^*\mc{L}_{\varpi_i} \]
is just the quotient map $U \otimes \mc{O}_{\mb{P}^1} \oplus \mc{O}(d - 1) \to U \otimes \mc{O}_{\mb{P}^1}$, where $\sigma'$ is the morphism $Q_i \times^B Q_i/B = Q_i \times^B B \times^B Q_i/B \to \tilde{X}_{s_is_js_i}$ and $\pi' \colon Q_i \times^B Q_i/B \to Q_i/B$ is the natural projection onto the first factor. So the subscheme
\[ \mb{P}_{Q_i/B}(\ker \pi_* m^* \mc{L}_{\varpi_i} \to \pi'_*(\sigma')^*m^*\mc{L}_{\varpi_i}) \subseteq \mb{P}_{Q_i/B}(\ker \pi_* m^* \mc{L}_{\varpi_i} \to \sigma^*m^*\mc{L}_{\varpi_i}) = M'\]
is the canonical section of $\mb{F}_{d - 1}$ of degree $1 - d$. But this parametrises curves of class $\alpha_i^\vee + \alpha_j^\vee$ containing some curve of class $\alpha_i^\vee$, so this must be the closure of the locus of curves with dual graph
\[
\begin{tikzpicture}[every label/.style=dlabel]
\draw (0, 0) node [dnode, label=below:{\alpha_j^\vee}] {} -- (1, 0) node [dnode, label=below:{\alpha_i^\vee}] {} -- (1.5, 0);
\end{tikzpicture} 
\]
as claimed.
\end{proof}

\subsection{The divisor $D_{\alpha_i^\vee}(Z)$} \label{subsection:dalphai}

In this subsection, we complete the proof of Theorem \ref{thm:introsubregularresolutions} by proving Proposition \ref{prop:subregularresolutions4} below.

For the statement, we let
\[ N = \begin{cases} n_1 + 1, & \text{in type}\;\; A, \\ n_1 - 1, & \text{in type}\;\; F, \\ n_1, & \text{otherwise}.\end{cases}\]
We let $\theta_N' \colon Y \to Y \times_S \mrm{Pic}^0_S(E)$ be the section $\theta_N'(y) = (y, 0)$, and for $1 \leq k < N$, we let $\theta_k' \colon Y \to Y \times_S \mrm{Pic}^0_S(E)$ be the section given in type $A$ by
\[ \theta_k'(y) = \begin{cases} (y, -\varpi_{i}(y) + \varpi_{i + 1}(y) + \varpi_l(y)), & \text{if}\;\; k = 1, \\ (y, -\varpi_{i}(y) + \varpi_{i + 1}(y) + \varpi_{l - k + 1}(y) - \varpi_{l - k + 2}(y)), & \text{if}\;\; 1 < k \leq l - i + 1 = N - 1, \end{cases}\]
and in types $B$, $D$ and $E$ by
\[ \theta_k'(y) = \begin{cases} (y, \alpha_{l - 1}(y)), & \text{in type}\; B, \\ (y, \alpha_{l - 2}(y) + \cdots + \alpha_{l - k}(y)), & \text{in type}\; D, \\ (y, \alpha_k(y) + \alpha_{k + 1}(y) + \cdots + \alpha_3(y)), & \text{in type}\; E.\end{cases}\]
Note that $N = 1$ in types $C$, $F$ and $G$.

\begin{prop} \label{prop:subregularresolutions4}
Assume we are in the setup of Proposition \ref{prop:subregularresolutions1}, and moreover assume for simplicity of notation that $i = l - 3$ if $(G, P, \mu)$ is of type $D$, and that $i = 5$ if $(G, P, \mu)$ is of type $E$. Then there is a sequence of $N$ morphisms
\[ D_{\alpha_i^\vee}(Z) = D_{N + 1}' \longrightarrow D_N' \longrightarrow \cdots \longrightarrow D_1'\]
over $Y \times_S Z$ such that $D_1'$ is a family of smooth surfaces over $Y$ containing $Y \times_S \mrm{Pic}^0_S(E)$ as a closed substack, and $D_{k + 1}' \to D_k'$ is the blowup along the section $\theta_k' \colon Y \to Y \times_S \mrm{Pic}^0_S(E) \subseteq D_k'$ of the proper transform of $Y \times_S \mrm{Pic}^0_S(E) \subseteq D_1'$. Moreover, we have the following descriptions of $D_1'$ in each type.
\begin{enumerate}[(1)]
\item \label{itm:subregularresolutions4:1} In type $A$, $D_1' \to Y \times_S Z_0 = Y \times_S \mrm{Pic}^0(E)$ is a line bundle.
\item \label{itm:subregularresolutions4:2} In type $B$, the morphism $D_1' \to Y \times_S Z_0$ is a $\mb{P}^1$-bundle such that the fibre of $D_1' \to Y$ over a point $y \in Y$ is isomorphic to the stacky Hirzebruch surface
\[ (D_1')_y \cong \begin{cases} \mb{P}_{\mb{P}(1, 2)}(\mc{O} \oplus \mc{O}(1)), & \text{if}\;\; \varpi_l(y) \neq 0, \\ \mb{P}_{\mb{P}(1, 2)}(\mc{O} \oplus \mc{O}(3)), & \text{if}\;\; \varpi_l(y) = 0. \end{cases}\]
\item \label{itm:subregularresolutions4:3} In types $C$ and $D$, the morphism $D_1' \to Y \times_S Z_0$ is a $\mb{P}^1$-bundle such that the fibre of $D_1' \to Y$ over a point $y \in Y$ is isomorphic to the Hirzebruch surface
\[(D_1')_y \cong \begin{cases} \mb{F}_0, & \text{if}\;\; \varpi_l(y) \neq 0, \\ \mb{F}_2, & \text{if}\;\; \varpi_l(y) = 0. \end{cases}\]
\item \label{itm:subregularresolutions4:4} In types $E$ and $G$, the morphism $D_1' \to Y \times_S Z_0 = Y$ is a $\mb{P}^2$-bundle.
\item \label{itm:subregularresolutions4:5} In type $F$, the morphism $D_1' \to Y \times_S Z_0 = Y$ factors as a sequence of two $\mb{P}^1$-bundles $D_1' \to D_1'' \to Y$, and the fibre over a point $y \in Y$ is isomorphic to the Hirzebruch surface
\[ (D_1')_y \cong \begin{cases} \mb{F}_0, & \text{if}\;\; \alpha_1(y) \neq 0, \\ \mb{F}_2, & \text{if}\;\; \alpha_1(y) = 0. \end{cases}\]
\end{enumerate}
\end{prop}
\begin{proof}
First note that in type $A$, the roots $\alpha_i$ and $\alpha_j = \alpha_{i + 1}$ play completely symmetric roles. So applying Proposition \ref{prop:subregularresolutions2} with the vertices of the Dynkin diagram $A_l$ labelled in reverse order gives contractions
\[ D_{\alpha_i^\vee}(Z) = D_{l - i + 2}' \longrightarrow D_{l - i + 1}' \longrightarrow \cdots \longrightarrow D_1'\]
with the desired properties, where to get the correct blowup loci we have composed the identification of $D_1'$ with a line bundle over $Y \times_S \mrm{Pic}^0_S(E)$ given by Proposition \ref{prop:subregularresolutions2} with the isomorphism
\begin{align*}
Y \times_S \mrm{Pic}^0_S(E) &\overset{\sim}\longrightarrow Y \times_S \mrm{Pic}^0_S(E) \\
(y, x) &\longmapsto (y, -x).
\end{align*}

If $G$ is not of type $A$, then we define
\begin{equation} \label{eq:subregularresolutions4:1}
D_{\alpha_i^\vee}(Z) \longrightarrow D_N' := \mrm{KM}_{B, G, rig}^{-\alpha_i^\vee} \times_{\bun_{G, rig}} Z \times_S E
\end{equation}
to be the map given by deleting the unique degree $\alpha_i^\vee$ rational component of a stable section and recording its image in $E$. 

Let
\[ (D_N')_0 = \bun_{B, rig}^{-\alpha_i^\vee} \times_{\bun_{G, rig}} Z \times_S E \subseteq D_N'\]
and let $(D_N')_1 = D_N' \setminus (D_N')_0$. Then $(D_N')_1$ is a smooth divisor in $D_N'$ isomorphic to
\[ (D_N')_1 \cong \bun_{B, rig}^{-\alpha_i^\vee - \alpha_j^\vee} \times_{\bun_{G, rig}} Z \times_S E \times_S E,\]
where the first (resp., second) factor of $E$ above keeps track of the point of attachment of an $\alpha_j^\vee$ curve. There is a morphism
\begin{equation} \label{eq:subregularresolutions4:2}
(D_N')_1 \longrightarrow Y \times_S \mrm{Pic}^0_S(E)
\end{equation}
given on the first factor by the morphism $(D_N')_1 \to D_N' \to Y$ and on the second by the morphism
\begin{align*}
(D_N')_1 \longrightarrow E \times_S E &\longrightarrow \mrm{Pic}^0_S(E) \\
(x_j, x_i) &\longmapsto x_j - x_i.
\end{align*}
Using the fact that $D_{\alpha_i^\vee}(Z)$ is naturally identified with the pullback of the universal domain curve over $\mrm{KM}_{B, G, rig}^{-\alpha_i^\vee} \times_{\bun_{G, rig}} Z$, one can deduce from \cite[Proposition 2.1.7]{davis19} that \eqref{eq:subregularresolutions4:1} is the blow up at the preimage of the section $\theta_N \colon Y \to Y \times_S \mrm{Pic}^0_S(E)$ under \eqref{eq:subregularresolutions4:2}. It follows that the strict transform $D_{\alpha_i^\vee}(Z) \cap D_{\alpha_j^\vee}(Z)$ of $(D_N')_1$ maps isomorphically to it. By construction, the composition
\[ D_{\alpha_i^\vee}(Z) \cap D_{\alpha_j^\vee}(Z) \overset{\sim}\longrightarrow (D_N')_1 \xrightarrow{\eqref{eq:subregularresolutions4:2}} Y \times_S \mrm{Pic}^0_S(E) \]
agrees with the composition
\[ D_{\alpha_i^\vee}(Z) \cap D_{\alpha_j^\vee}(Z) \overset{\sim}\longrightarrow C_{1, n_0 + 1} \xrightarrow{\eqref{eq:bruhatcellcurvecomparison1}} Y \times_S \mrm{Pic}^0_S(E),\]
and is therefore an isomorphism by Lemma \ref{lem:bruhatcellcurvecomparison}.

The next step is to construct the spaces $D_k'$ for $1 \leq k < N$. This is vacuous for types $C$, $F$ and $G$ (since $N = 1$ in these cases). In the remaining types, we define standard parabolics $P_k' \subseteq G$ for $1 \leq k < N$ and set
\[ D_k' = Y \times_{Y_{P_k'}}(\mrm{KM}_{P_k', G, rig} \times_{\bun_{G, rig}} Z \times_S E).\]
In type $B$, $N = 2$, and we let $P_1'$ be the standard parabolic with type $t(P_1') = \{\alpha_i, \alpha_l\} = \{\alpha_{l - 2}, \alpha_l\}$. In type $D$, $N = 3$, and we let $t(P_1') = \{\alpha_i, \alpha_l \} = \{\alpha_{l - 3}, \alpha_l\}$ and $t(P_2') = \{\alpha_{l - 3}, \alpha_{l - 1}, \alpha_l\}$. Finally, in type $E$, $N = 4$, and we let $t(P_1') = \{\alpha_i, \alpha_4\} = \{\alpha_4, \alpha_5\}$, $t(P_2') = \{\alpha_1, \alpha_4, \alpha_5\}$ and $t(P_3') = \{\alpha_1, \alpha_2, \alpha_4, \alpha_5\}$. Note that in each case, we have a sequence of morphisms
\[ D_N' \longrightarrow D_{N - 1}' \longrightarrow \cdots \longrightarrow D_1'\]
coming from the inclusions of the parabolics.

We prove below in Proposition \ref{prop:basesurfaces} that the spaces $D_1'$ are as described in the statement of the proposition. This completes the proof of the proposition in types $C$, $F$ and $G$. In types $B$, $D$ and $E$, we still need to show that $D_{k + 1}' \to D_k'$ is the blowup at the desired section for $1 \leq k < N$. As in the proof of Proposition \ref{prop:subregularresolutions2}, the proof relies on a decomposition into locally closed substacks coming from the Bruhat cells of \S\ref{subsection:glnbruhat}.

We define representations $\pi_{P_1'} \colon P_1' \to GL_{n_1}$ of $P_1'$ as follows. In type $B$, we let $\pi_{P_1'}$ be given by
\[ P_1' \longrightarrow P_1'/R_u(P) = L \cap P_1' \overset{\rho_L}\longrightarrow GSp_4 \cap R_4,\]
composed with the homomorphism
\begin{align*}
GSp_4 \cap R_4 &\longrightarrow GL_2 \\
\left(\begin{array}{c|cc|c} \lambda^{-1}\det A & 0 & 0 & 0 \\ \hline 0 & & & 0 \\ 0 & \multicolumn{2}{c|}{\smash{\raisebox{.5\normalbaselineskip}{$A$}}} & 0 \\ \hline \\[-\normalbaselineskip] 0 & 0 & 0 & \lambda \end{array}\right) &\longmapsto A,
\end{align*}
where $\rho_L$ is the representation defined in \S\ref{subsection:sliceexistence}. In types $D$ and $E$, we let $\pi_{P_1'} \colon P_1' \to GL_{n_1}$ be the composition
\[ \pi_{P_1'} \colon P_1' \longrightarrow L \cap P_1' \overset{\rho_L}\longrightarrow R_{n_1 + 1} \longrightarrow GL_{n_1},\]
where the last homomorphism is given by deleting the last row and column, and $\rho_L$ is the composition of the isomorphism of Lemma \ref{lem:typecdefleviiso} with the projection to the second factor.

In each of types $B$, $D$ and $E$, we have $P_k' = (\pi_{P_1'})^{-1}(Q_k^{n_1})$ for $1 \leq k \leq N$, where we set $P_N' = P \cap P_1$ in the notation of \S\ref{subsection:dalphaj}. Note that the morphism
\begin{equation} \label{eq:subregularresolutions4:3}
 D_N' \longrightarrow Y \times_{Y_{P_N'}}(\mrm{KM}_{P_N', G, rig}^{-\alpha_i^\vee} \times_{\bun_{G, rig}} Z \times_S E)
\end{equation}
is an isomorphism by \cite[Lemma 4.3.7]{davis19}, since there is an isomorphism $P_N'/B \cong GL_{n_0}/Q^{n_0}_{n_0}$ identifying sections of degree $-\alpha_i^\vee$ with sections of degree $-e_{n_0}^*$. So we have a sequence of pullback squares
\begin{equation} \label{eq:subregularresolutions4:4}
\begin{tikzcd}
D_{k + 1}' \ar[r] \ar[d] & Y_{Q^{n_1}_{n_1}}^{-e_{n_1}^*} \times_{Y_{Q_{k + 1}^{n_1}}^{-e_{n_1}^*}} \mrm{KM}_{Q^{n_1}_{k + 1}, GL_{n_1}, rig}^{-e_{n_1}^*} \ar[d] \\
D_{k}' \ar[r] & Y_{Q^{n_1}_{n_1}}^{-e_{n_1}^*} \times_{Y_{Q_{k}^{n_1}}^{-e_{n_1}^*}} \mrm{KM}_{Q^{n_1}_{k}, GL_{n_1}, rig}^{-e_{n_1}^*},
\end{tikzcd}
\end{equation}
where the subscript $(-)_{rig}$ denotes the rigidification with respect to the image of $Z(G)$ in $Z(GL_{n_1})$.

By Lemma \ref{lem:subregsections}, there is a stable section of $\xi_{G, z} \times^G G/P$ of degree $-\alpha_i^\vee$ if and only if $z \in Z_0 \subseteq Z$, and for such $z$, the unique such section is the canonical (Harder-Narasimhan) one of $\xi_{G, z} \times^G G/P = \xi_{L, z} \times^L G/P$. Since $P_k' \subseteq P$, one can use this fact, the definition of the slice $Z_0$ and elementary slope arguments (see e.g., \cite[Lemma 6.6.11]{davis19a}) to show in each case that any unstable $GL_{n_1}$-bundle in the image of $D_1' \to \bun_{GL_{n_1}, rig}^{-1}$ has Harder-Narasimhan reduction to $R_{n_1}$ of degree $-e_1^*$. By Proposition \ref{prop:glndecomposition}, we therefore have a decomposition
\[ D_k' = (D_k' \times_{\bun_{GL_{n_1}, rig}^{-1}} \bun_{GL_{n_1}, rig}^{ss, -1}) \cup \bigcup_{1 \leq p < k} C_{k, p}' \cup C_{k, n_1}' \]
into disjoint locally closed substacks for $1 \leq k \leq n_1 = N$, where $C_{k, p}' \subseteq D_k'$ is the preimage of $C_{k, p, rig}^{GL_{n_1}} \subseteq X^{n_1}_{k, rig}$ in $D_k'$. We remark that $C_{N, n_1}' = (D'_N)_1 \cong Y \times_S \mrm{Pic}^0_S(E)$.

Using Proposition \ref{prop:glnbruhatcomparison}, Lemma \ref{lem:blowuprecognition}, \cite[Lemma 4.3.7]{davis19} and the pullback squares \eqref{eq:subregularresolutions4:4}, one can now check that $D_{k + 1}' \to D_k'$ is the blowup along the desired section $\theta_k'$ of $C_{k, n_1}' \cong C_{N, n_1}' = Y \times_S \mrm{Pic}^0_S(E)$ exactly as in the proof of Proposition \ref{prop:subregularresolutions2}.
\end{proof}

In the rest of this subsection, we will establish the propositions and lemmas quoted in the proof of Proposition \ref{prop:subregularresolutions4}. We will assume from now on that $(G, P, \mu)$ is not of type $A$.

\begin{lem} \label{lem:subregsections}
Assume $z \in Z$ is such that there exists a section of $\xi_{G, z} \times^G G/P$ of degree $\leq -\alpha_i^\vee$. Then $z \in Z_0 \subseteq Z$, and the only such section is the canonical (Harder-Narasimhan) one of $\xi_{G, z} \times^G G/P = \xi_{P, z} \times^P G/P$.
\end{lem}
\begin{proof}
First note that $\mrm{KM}_{B, G}^{-\alpha_i^\vee} \to \mrm{KM}_{P, G}^{-\alpha_i^\vee}$ is surjective by \cite[Proposition 3.6.4]{davis19a}. Since $z \notin Z_0$ implies that $\xi_{G, z}$ is either semistable or regular unstable, we must therefore have $z \in Z_0$ by \cite[Lemma 4.3.4]{davis19} as $\alpha_i$ is not a special root. Given a section $\sigma$ of $\xi_{G, z} \times^G G/P$ of degree $\leq -\alpha_i^\vee$, any lift to a section of $\xi_{G, z} \times^G G/B$ of degree $\leq -\alpha_i^\vee$ must factor through $\xi_{P, z} \times^P P/B$ by Lemma \ref{lem:typenotacells} and Proposition \ref{prop:levidegreebound}, so $\sigma$ must be the canonical section as claimed.
\end{proof}

\begin{lem} \label{lem:typenotacells}
Assume $w \in W^0_{P, B}$, $\lambda \leq -\alpha_i^\vee$ and $C^{w, \lambda}(Z_0) \neq \emptyset$. Then $w = 1$ and $\lambda \in \{-\alpha_i^\vee, -\alpha_i^\vee - \alpha_j^\vee\}$.
\end{lem}
\begin{proof}
From the proof of Lemma \ref{lem:subregularborelcells}, we have either $w\lambda = -\alpha_i^\vee$ and $w = 1$, or $w\lambda = -\alpha_i^\vee - \alpha_j^\vee$ and
\[ w \in \{1\} \cup \{s_{c_0, n_0} s_{c_0, n_0 - 1} \cdots s_{c_0, k} \mid 1 \leq k \leq n_0\}.\]
If $w \neq 1$, then this implies that $\lambda = -w^{-1}(\alpha_i^\vee + \alpha_j^\vee) = -\alpha_j^\vee$, contradicting $\lambda \leq - \alpha_i^\vee$. So this proves the lemma.
\end{proof}

It now remains only to describe the maps $D_1' \to Y \times_S Z_0$. We do this in Proposition \ref{prop:basesurfaces} after a few preparations.

Since the statement is local on $S$, we will assume from now on that the initial section $S \to \bun_{L, rig}^{ss, \mu}$ (resp.\ $\B_S \mb{G}_m \to \bun_{L', rig}^{ss, \mu'}$) used in the construction of the slice $Z_0$ in types $E$, $F$ and $G$ (resp.\ $B$, $C$ and $D$) lifts to a section $S \to \bun_L^{ss, \mu}$ (resp.\ $S \to \bun_{L'}^{ss, \mu'}$). We will also write $Z_1 = Z_0 = S$ in types $E$, $F$ and $G$ and $Z_1 = \mrm{Ind}_{L'}^L(S) \setminus S$ in types $B$, $C$ and $D$; our assumption implies that $Z_0 \to \bun_{L, rig}^{ss, \mu}$ lifts to $Z_1 \to \bun_L^{ss, \mu}$.

We first relate $D_1' \to Y \times_S Z_0$ to the projectivisation of a vector bundle. Let $\rho_L$ be the representation of $L$ given by the isomorphism of Lemmas \ref{lem:typecdefleviiso} and \ref{lem:typegleviiso} composed with the projection to the second factor in types $C$, $D$, $E$, $F$ and $G$, and given by the isomorphism of Lemma \ref{lem:typebleviiso} composed with the projection to the second factor and the inclusion $GSp_4 \subseteq GL_4$ in type $B$. We will write $W$ for the vector bundle on $Z_1 \times_S E$ induced by $Z_1 \to \bun_L^{ss, \mu}$ and $\rho_L$. We will also write $\lambda \in \mb{X}^*(T)$ for the character
\[ \lambda = \begin{cases} \varpi_l, & \text{in types}\; B, C, D, \\ \varpi_4, &\text{in type}\; E, \\ \varpi_2, & \text{in type}\;G.\end{cases}\]

\begin{lem} \label{lem:typebcdegbasesurfaceprojectivisation}
In types $B$, $C$, $D$, $E$ and $G$, there is an isomorphism
\[ D_1' \times_{Z_0} Z_1 \cong \mb{P}_{Y \times_S Z_1}\pi_*(M_\lambda \otimes \mc{O}(dO_E) \otimes W),\]
where $\pi \colon Y \times_S Z_1 \times_S E \to Y \times_S Z_1$ is the natural projection and $M_\lambda$ is the line bundle on $Y \times_S Z_1 \times_S E$ classified by the morphism
\[ Y \times_S Z_1 \longrightarrow Y \overset{\lambda}\longrightarrow \mrm{Pic}^0_S(E).\]
\end{lem}
\begin{proof}
We first prove the lemma in types $B$, $D$ and $E$. Let
\[ X = Y \times_{Y_{P_1'}} (\bun_{L \cap P_1'}^{-\alpha_i^\vee} \times_{\bun_L} Z_1 \times_S E) \subseteq D_1' \times_{Z_0} Z_1,\]
where we note that Lemma \ref{lem:subregsections} implies that
\[ D_1' = Y \times_{Y_{P_1'}} (\mrm{KM}_{L \cap P_1', L, rig}^{-\alpha_i^\vee} \times_{\bun_{L, rig}} Z_0 \times_S E).\]
Lemmas \ref{lem:typecdefleviiso} and \ref{lem:typebcdparabolicbundles} show that $X$ is the stack of tuples $(y, z, M_{\lambda, y}^{-1} \otimes \mc{O}(-O_E) \subseteq W_z)$, where $y \in Y$, $z \in Z_1$, $M_{\lambda, y}$ is the line bundle on $E$ corresponding to $\lambda(y) \in \mrm{Pic}^0_S(E)$, and $W_z$ is the restriction of $W$ to the fibre over $z \in Z_1$. Since the vector bundle $W_z$ is semistable of slope $< 0$, any nonzero morphism $M_{\lambda, y}^{-1} \otimes \mc{O}(-O_E) \to W_z$ must be a subbundle, so we have an isomorphism
\[ X \cong \mb{P}_{Y \times_S Z_1}\pi_*(M_\lambda \otimes \mc{O}(O_E) \otimes W).\]
Since this implies in particular that $X$ is already proper over $Y \times_S Z_1 = Y \times_{Y_{P_1'}}(Z_1 \times_S E)$, we conclude that $X = D_1' \times_{Z_0} Z_1$ and the claim is proved.

In types $C$ and $G$, we argue instead as follows. Observe that there is a pullback
\begin{equation} \label{eq:typecgbasesurface1}
\begin{tikzcd}
D_1' \times_{Z_0} Z_1 \ar[r] \ar[d] & \mrm{KM}_{Q^2_2, GL_2}^{-de_2^*} \times_{\bun_{GL_2}} \bun_{GL_2}^{ss, -d} \ar[d] \\
Y \times_S Z_1 \ar[r] & \mrm{Pic}^{-d}_S(E) \times_S \bun_{GL_2}^{ss, -d},
\end{tikzcd}
\end{equation}
where the bottom morphism is given by
\[ (y, z) \longmapsto (M_{\lambda, y}^{-1} \otimes \mc{O}(-dO_E), W_z)\]
and the right morphism is given on the first factor by the blow down to $T_{Q^2_2}$-bundles composed with the character $e_2$. If $(y, z) \in Y \times_S Z_1$ lies over a geometric point $s \colon \spec k \to S$, then any stable map to the $GL_2$ flag variety bundle $\mb{P}(W_z^\vee)$ corresponding to a point in $D_1' \times_{Z_0} Z_1$ over $(y, z)$ is a closed immersion with ideal sheaf $p^*(M_{\lambda, y}^{-1} \otimes \mc{O}(-d)O_E)) \otimes \mc{O}(-1)$, where $p \colon \mb{P}(W_z^\vee) \to E_s$ is the structure morphism. So we deduce that
\[ D_1' \times_{Z_0} Z_1 = \mb{P}_{Y \times_S Z_1} \pi_*p_*(p^*(M_\lambda \otimes \mc{O}(d O_E)) \otimes \mc{O}(1)) = \mb{P}_{Y \times_S Z_1} \pi_*(M_\lambda \otimes \mc{O}(d O_E) \otimes W_z)\]
as claimed.
\end{proof}

The situation in type $F$ is similar. In this case, we let $P_1'' \subseteq L$ be the standard parabolic subgroup of type $t(P_1'') = \{\alpha_1\}$, and define
\[ D_1'' = Y \times_{Y_{P_1''}}(\mrm{KM}_{P_1'', L, rig}^{-\alpha_i^\vee} \times_{\bun_{L, rig}^{\mu}} Z_0 \times_S E).\]

\begin{lem} \label{lem:typefbasesurfaceprojectivisation}
In type $F$, there are isomorphisms
\[ D_1'' \cong \mb{P}_{Y \times_S Z_1}\pi_*(M_{\varpi_1} \otimes W^\vee) \]
and 
\[ D_1' \cong \mb{P}_{D_1''}\pi'_*(p^*M_{\varpi_2} \otimes \mc{O}(2O_E) \otimes \ker(p^*W \to p^*M_{\varpi_1} \otimes \mc{O}_{D_1''}(1))),\]
where $\pi \colon Y \times_S Z_1 \times_S E \to Y \times_S Z_1$ and $\pi' \colon D_1' \times_S Z_1 \times_S E \to Y \times_S Z_1$ are the natural projections, and $p \colon D_1'' \to Y \times_S Z_1$ is the structure morphism.
\end{lem}
\begin{proof}
Recall that $\alpha_i = \alpha_3$ and $Z_1 = S$ in this case and let
\[
X = Y \times_{Y_{P_1''}}(\bun_{P_1''}^{-\alpha_3^\vee} \times_{\bun_L^{\mu}} Z_1 \times_S E) \subseteq D_1''.
\]
Then Lemma \ref{lem:typecdefleviiso} shows that $X$ is the stack of tuples $(y, z, W_z \twoheadrightarrow M_{\varpi_1, y})$, where $y \in Y$ and $z \in Z_1$. Since the vector bundle $W_z$ is semistable of slope $> -1$, any nonzero morphism $W_z \to M_{\varpi_1, y}$ is surjective, so we have an isomorphism
\[ X \cong \mb{P}_{Y \times_S Z_1}(\pi_*(M_{\varpi_1} \otimes W^\vee)).\]
Since this shows that $X$ is already proper over $Y \times_S Z_1 = Y \times_{Y_{P_1''}}(Z_1 \times_S E)$, it follows that $X = D_1''$, so this gives the first of the desired isomorphisms.

For the second isomorphism, there is a pullback
\[
\begin{tikzcd}
D_1' \ar[r] \ar[d] & \mrm{KM}_{Q^2_2, GL_2}^{-2e_2^*} \times_{\bun_{GL_2}^{-2}} \bun_{GL_2}^{ss, -2} \ar[d] \\
D_1'' \ar[r] & \mrm{Pic}^{-2}_S(E) \times_S \bun_{GL_2}^{ss, -2}
\end{tikzcd}
\]
where the bottom horizontal morphism is classified by the pair $(p^*M_{\varpi_2}^{-1} \otimes \mc{O}(-2O_E), \ker(p^*W \to p^*M_{\varpi_1} \otimes \mc{O}_{D_1''}(1)))$ of line bundle and vector bundle on $D_1'' \times_S E$. Since any stable map to the associated flag variety bundle appearing in $D_1'$ is again a closed immersion, the argument used in the proof of Lemma \ref{lem:typebcdegbasesurfaceprojectivisation} for types $C$ and $G$ gives the desired isomorphism
\[ D_1' \cong \mb{P}_{D_1''}\pi'_*(p^*M_{\varpi_2} \otimes \mc{O}(2O_E) \otimes \ker(p^*W \to p^*M_{\varpi_1} \otimes \mc{O}_{D_1''}(1))).\]
\end{proof}

\begin{prop} \label{prop:basesurfaces}
The descriptions given in Proposition \ref{prop:subregularresolutions4} for the maps $D_1' \to Y \times_S Z_0$ are correct.
\end{prop}
\begin{proof}
First observe that in types $E$ and $G$, $M_\lambda \otimes \mc{O}(d O_E) \otimes W$ is a family of semistable vector bundles of degree $3$, so Lemma \ref{lem:typebcdegbasesurfaceprojectivisation} shows that $D_1' \to Y \times_S Z_1 = Y$ is a $\mb{P}^2$-bundle, which proves \eqref{itm:subregularresolutions4:4}.

In types $B$, $C$ and $D$, $M_\lambda \otimes \mc{O}((d + 1)O_E) \otimes W$ is a family of semistable vector bundles of degree $2$, so Lemma \ref{lem:typebcdegbasesurfaceprojectivisation} shows that $D_1' \times_{Z_0} Z_1 \to Y \times_S Z_1$ is a $\mb{P}^1$-bundle, and hence that $D_1' \to Y \times_S Z_0$ is also.

To complete the proof of \eqref{itm:subregularresolutions4:2}, note that in type $B$, we have a canonical $Z(L')$-invariant subbundle $\mc{O}(-O_E) \subseteq W$ and a $Z(L')$-equivariant exact sequence
\[ 0 \longrightarrow U \longrightarrow W/\mc{O}(-O_E) \longrightarrow \mc{O} \longrightarrow 0,\]
where $U$ is a family of stable vector bundles on $E$ of rank $2$ and determinant $\mc{O}(-O_E)$. So if we fix a geometric point $y \colon \spec k \to Y$ over $s \colon \spec k \to S$, we have $Z(L')$-equivariant exact sequences
\begin{align} \label{eq:typebbasesurface1}
0 \longrightarrow \pi_*(M_{\varpi_l, y}) \longrightarrow \pi_*&(M_{\varpi_l, y} \otimes \mc{O}(O_E) \otimes W_s) \\
& \longrightarrow \pi_*(M_{\varpi_l, y} \otimes \mc{O}(O_E) \otimes (W_s/\mc{O}(-O_E))) \longrightarrow \mb{R}^1\pi_*(M_{\varpi_l, y}) \longrightarrow 0, \nonumber
\end{align}
and
\begin{align} \label{eq:typebbasesurface2}
0 \longrightarrow \pi_*(M_{\varpi_l, y} \otimes \mc{O}(O_E) \otimes U_s) \longrightarrow \pi_*(M_{\varpi_l, y} \otimes \mc{O}(O_E)&\otimes (W_s/\mc{O}(-O_E)))\\
& \longrightarrow \pi_*(M_{\varpi_l, y} \otimes \mc{O}(O_E)) \longrightarrow 0 \nonumber
\end{align}
of $Z(L')$-linearised vector bundles on $(Z_1)_s$. Note that $\pi_*(M_{\varpi_l, y})$, $\mb{R}^1\pi_*(M_{\varpi_l, y})$, $\pi_*(M_{\varpi_l, y} \otimes \mc{O}(O_E) \otimes U_s)$ and $\pi_*(M_{\varpi_l, y} \otimes \mc{O}(O_E))$ are each either a trivial line bundle or zero, with $Z(L')$-weights $f_4$, $f_4$, $f_2 = f_3$ and $f_1$ respectively, where we use the notation of the proof of Lemma \ref{lem:typebleviiso}. So after tensoring with the character $-f_1$ of $Z(L')$, $Z(G)$ acts trivially on \eqref{eq:typebbasesurface1} and \eqref{eq:typebbasesurface2}, so they descend to exact sequences of vector bundles on $(Z_0)_s = (Z_1)_s/\mb{G}_m \cong \mb{P}(1, 2)$. Examining the $\mb{G}_m$-weights, the sequence \eqref{eq:typebbasesurface2} descends to a sequence of the form
\[ 0 \longrightarrow \mc{O}(1) \longrightarrow W' \longrightarrow \mc{O} \longrightarrow 0.\]
Since any such sequence splits, we must have $W' \cong \mc{O} \oplus \mc{O}(1)$ as vector bundles on $\mb{P}(1, 2)$.

If $\varpi_l(y) \neq 0$, then $\pi_*(M_{\varpi_l, y}) = \mb{R}^1 \pi_*(M_{\varpi_l, y}) = 0$, so we have
\[ \pi_*(M_{\varpi_l, y} \otimes \mc{O}(O_E) \otimes W_s) \cong \pi_*(M_{\varpi_l, y} \otimes \mc{O}(O_E) \otimes W_s/\mc{O}(-O_E)),\]
and hence $(D_1')_y = \mb{P}_{\mb{P}(1, 2)}(W') = \mb{P}_{\mb{P}(1, 2)}(\mc{O} \oplus \mc{O}(1))$. Otherwise, \eqref{eq:typebbasesurface1} tensored with $-f_1$ descends to an exact sequence
\[ 0 \longrightarrow \mc{O}(2) \longrightarrow W'' \longrightarrow W' = \mc{O} \oplus \mc{O}(1) \longrightarrow \mc{O}(2) \longrightarrow 0\]
such that $(D_1')_y = \mb{P}_{\mb{P}(1, 2)}(W'')$. But since the kernel of any surjection $\mc{O} \oplus \mc{O}(1) \to \mc{O}(2)$ on $\mb{P}(1, 2)$ must be isomorphic to $\mc{O}(-1)$, this means that we must have $W'' = \mc{O}(-1) \oplus \mc{O}(2)$, so
\[ (D_1')_y = \mb{P}_{\mb{P}(1, 2)}(\mc{O}(-1) \oplus \mc{O}(2)) = \mb{P}_{\mb{P}(1, 2)}(\mc{O} \oplus \mc{O}(3)).\]
This proves \eqref{itm:subregularresolutions4:2}.

Similarly, to prove \eqref{itm:subregularresolutions4:3}, note that in types $C$ and $D$ we have a canonical $Z(L')$-equivariant exact sequence
\[ 0 \longrightarrow \mc{O}(-d O_E) \longrightarrow W \longrightarrow U \longrightarrow 0,\]
where $U$ is semistable and $Z(L')$ acts on $\mc{O}(-d O_E)$ and $\mc{O}$ respectively with weights
\[ e_{n_1 + 1} = -\varpi_l + (d + 1) \varpi_i = \begin{cases} -\varpi_l + 2 \varpi_{l - 1}, & \text{in type}\;C, \\ -\varpi_l + \varpi_{l - 3}, & \text{in type}\; D, \end{cases} \quad \text{and} \quad e_1 = \begin{cases} \varpi_l, & \text{in type}\; C, \\ \varpi_{l - 1}, & \text{in type}\; D. \end{cases}\]
So over any geometric point $y \colon \spec k \to Y$ over $s \colon \spec k \to S$, we have an exact sequence
\begin{align} \label{eq:typecdbasesurface1}
 0 \longrightarrow \pi_*(M_{\varpi_l, y}) \longrightarrow \pi_*&(M_{\varpi_l, y} \otimes \mc{O}(d O_E) \otimes W_s) \\
& \longrightarrow \pi_*(M_{\varpi_l, y} \otimes \mc{O}(d O_E) \otimes U_s) \longrightarrow \mb{R}^1\pi_*(M_{\varpi_l, y}) \longrightarrow 0, \nonumber
\end{align}
of $Z(L')$-linearised vector bundles on $(Z_1)_s$, which descends to an exact sequence of vector bundles on $\mb{P}^1 = (Z_0)_s = (Z_1)_s/\mb{G}_m$ after tensoring with $-e_1$. Note that in both cases $M_{\varpi_l, y} \otimes \mc{O}((d + 1)O_E) \otimes U_s$ is a semistable vector bundle of degree $2$ on which $Z(L')$ acts with the single weight $e_1$, so $\pi_*(M_{\varpi_l, y} \otimes \mc{O}((d + 1)O_E) \otimes U_s) \otimes \mb{Z}_{-e_1}$ descends to a trivial rank $2$ vector bundle $\mc{O} \oplus \mc{O}$ on $\mb{P}^1$.

If $\varpi_l(y) \neq 0$, then $\pi_*(M_{\varpi_l, y}) = \mb{R}^1\pi_*(M_{\varpi_l, y}) = 0$, so
\[ \pi_*(M_{\varpi_l, y} \otimes \mc{O}((d + 1)O_E) \otimes W_s) \otimes \mb{Z}_{-e_1} = \pi_*(M_{\varpi_l, y} \otimes \mc{O}((d + 1)O_E) \otimes U_s) \otimes \mb{Z}_{-e_1} \]
descends to $\mc{O} \oplus \mc{O}$ on $\mb{P}^1$, which together with Lemma \ref{lem:typebcdegbasesurfaceprojectivisation} shows that $(D_1')_y = \mb{P}_{\mb{P}^1}(\mc{O} \oplus \mc{O}) = \mb{F}_0$. Otherwise, \eqref{eq:typecdbasesurface1} descends to an exact sequence
\[ 0 \longrightarrow \mc{O}(1) \longrightarrow W' \longrightarrow \mc{O} \oplus \mc{O} \longrightarrow \mc{O}(1) \longrightarrow 0\]
such that $(D_1')_y \cong \mb{P}_{\mb{P}^1}(W')$. Since the kernel of any surjection $\mc{O} \oplus \mc{O} \to \mc{O}(1)$ must be isomorphic to $\mc{O}(-1)$, this implies that $W' \cong \mc{O}(-1) \oplus \mc{O}(1)$ and hence that
\[ (D_1')_y \cong \mb{P}_{\mb{P}^1}(\mc{O}(-1) \oplus \mc{O}(1)) \cong \mb{F}_2. \]
This proves \eqref{itm:subregularresolutions4:3}.

Finally, in type $F$, we have already constructed the morphisms $D_1' \to D_1'' \to Y = Y \times_S Z_0$. Since $M_{\varpi_1} \otimes W^\vee$ is a family of semistable vector bundles of degree $2$, Lemma \ref{lem:typefbasesurfaceprojectivisation} shows that $D_1'' \to Y$ is a $\mb{P}^1$-bundle as claimed. Moreover, any rank $2$ degree $-2$ subbundle of $W$ is necessarily also semistable, so Lemma \ref{lem:typefbasesurfaceprojectivisation} also shows that $D_1' \to D_1''$ is a $\mb{P}^1$-bundle.

If $y \colon \spec k \to Y$ is a geometric point over $s \colon \spec k \to S$, then by Lemma \ref{lem:typefbasesurfaceprojectivisation} we have an exact sequence
\begin{equation} \label{eq:typefbasesurface1}
0 \longrightarrow U \longrightarrow q^*(M_{\varpi_2, y} \otimes \mc{O}(2O_E) \otimes W_s) \longrightarrow  q^*(M_{\varpi_1 + \varpi_2, y} \otimes \mc{O}(2O_E)) \otimes (\pi')^*\mc{O}(1) \longrightarrow 0
\end{equation}
of vector bundles on $\mb{P}^1 \times E_s$ such that $(D_1')_y = \mb{P}\pi'_*U$, where $\pi'$ and $q$ are the projections to the first and second factors respectively. Since $U$ is a vector bundle of rank $2$ and determinant $q^*(M_{-\varpi_1 + 2 \varpi_2, y} \otimes \mc{O}(2O_E)) \otimes (\pi')^*\mc{O}(-1)$, it follows that we have an isomorphism
\[ U \overset{\sim}\longrightarrow U^\vee \otimes \det U = U^\vee \otimes q^*(M_{-\varpi_1 + 2\varpi_2, y} \otimes \mc{O}(2O_E)) \otimes (\pi')^*\mc{O}(-1).\]
So the dual of \eqref{eq:typefbasesurface1} gives an exact sequence
\[
0 \longrightarrow q^*M_{-2\varpi_1 + \varpi_2, y} \otimes (\pi')^*\mc{O}(-2) \longrightarrow q^*(M_{-\varpi_1 + \varpi_2, y} \otimes W_s^\vee) \otimes (\pi')^*\mc{O}(-1) \longrightarrow U \longrightarrow 0,
\]
and hence an exact sequence
\begin{align}
0 \longrightarrow H^0(E_s, M_{-2\varpi_1 + \varpi_2, y}) \otimes \mc{O}(-2) &\longrightarrow H^0(E_s, M_{-\varpi_1 + \varpi_2, y} \otimes W_s^\vee) \otimes \mc{O}(-1) \nonumber \\
&\longrightarrow (\pi')_*U \longrightarrow H^1(E_s, M_{-2\varpi_1 + \varpi_2, y}) \otimes \mc{O}(-2) \longrightarrow 0. \label{eq:typefbasesurface2}
\end{align}

If $\alpha_1(y) = 2\varpi_1(y) - \varpi_2(y) \neq 0$, then $H^0(E_s, M_{-2\varpi_1 + \varpi_2, y}) = H^1(E_s, M_{-2\varpi_1 + \varpi_2, y}) = 0$, so \eqref{eq:typefbasesurface2} gives an isomorphism
\[ (\pi')_*U \cong H^0(E_s, M_{-\varpi_1 + \varpi_2, y} \otimes W_s^\vee) \otimes \mc{O}(-1) = \mc{O}(-1) \oplus \mc{O}(-1),\]
so $(D_1')_y \cong \mb{P}_{\mb{P}^1}(\mc{O}(-1) \oplus \mc{O}(-1)) = \mb{F}_0$. Otherwise, \eqref{eq:typefbasesurface2} gives an exact sequence
\[ 0 \longrightarrow \mc{O}(-2) \longrightarrow \mc{O}(-1) \oplus \mc{O}(-1) \longrightarrow (\pi')_*U \longrightarrow \mc{O}(-2) \longrightarrow 0.\]
Since the cokernel of the injective morphism $\mc{O}(-2) \to \mc{O}(-1) \oplus \mc{O}(-1)$ must be isomorphic to $\mc{O}$, we get $(\pi')_*U \cong \mc{O}(-2) \oplus \mc{O}$ and hence $(D_1')_y \cong \mb{F}_2$. This completes the proof of \eqref{itm:subregularresolutions4:5} and of the proposition.
\end{proof}

\section{Singularities} \label{section:singularities}

In this section, we apply the results of \S\ref{section:resolutions} to the study of the singularities of the unstable varieties $\chi_Z^{-1}(0)$ and their deformations $\chi_Z \colon Z \to \hat{Y}\sslash W$. We describe the singularities explicitly in \S\ref{subsection:singularities}, which are given in Theorem \ref{thm:subregularsingularities}. In \S\ref{subsection:deformations}, we briefly sketch a minor variation on standard deformation theory (in which all deformation rings are graded by the character lattice of a torus) before stating and proving weighted miniversality of the deformations $\chi_Z$ (Theorem \ref{thm:subregulardeformations}). Theorems \ref{thm:subregularsingularities} and \ref{thm:subregulardeformations} together include all the statements from Theorem \ref{thm:introsubregularsingularities} from the introduction.

\subsection{Codimension $2$ singularities of the locus of unstable bundles} \label{subsection:singularities}

Theorem \ref{thm:introsubregularresolutions} (and the more refined statements in Propositions \ref{prop:subregularresolutions2} and \ref{prop:subregularresolutions3}) give very explicit descriptions of the families of normal crossings surfaces $\tilde{\chi}_Z^{-1}(0_{\Theta_Y^{-1}}) \to Y$. We show in this section how these results can be used to give equally explicit descriptions of the unstable loci $\chi_Z^{-1}(0)$ (which give local models for the singularities of the unstable loci of $\bun_G$, since the maps $Z \to \bun_G$ are slices). For the sake of simplicity, we will assume always that $S = \spec k$ for some algebraically closed field $k$.

We will see below that there is a dichotomy between the classical types ($A_l$ for $l > 1$, $B$, $C$ and $D$) and the exceptional types ($E$, $F$, $G$ and $A_1$). In the exceptional types, the unstable varieties are always cones over elliptic curves, with unique isolated singularities. In the classical types, the unstable varieties have non-isolated singularities obtained via the following construction.

\begin{construction} \label{cons:gluing}
Let $\pi \colon X \to X'$ be a degree $2$ morphism between smooth, possibly stacky curves over $k$, and let $L$ be a line bundle on $X$. The \emph{surface obtained by gluing $L$ along $\pi$} is the affine stack over $X'$ given by the spectrum of the fibre product
\[
\begin{tikzcd}
R \ar[r] \ar[d] & \pi_*\bigoplus_{n \geq 0} L^{\otimes-n} \ar[d] \\
\mc{O}_{X'} \ar[r] & \pi_*\mc{O}_X,
\end{tikzcd}
\]
where the vertical arrow on the right is given by restriction of a function on the total space of $L$ to the zero section. Geometrically, $\spec_{X'} R$ is the surface obtained by identifying points in the zero section of the total space of $L$ with the same image under $\pi$.
\end{construction}

\begin{rmk}
Assume that the characteristic of $k$ is not $2$, let $X''$ be a surface obtained by Construction \ref{cons:gluing}, and let $p \in X''$ be a singular point. If $p$ does not lie over a branch point of $\pi$, then the singularity at $p$ is of type $A_\infty$, i.e., we can choose (formal) local coordinates $x$, $y$ and $z$ at $p$ so that $X''$ has local equation $xy = 0$. If $p$ does lie over a branch point, then the singularity is of type $D_\infty$, i.e., we can choose local coordinates so that $X''$ has equation $x^2 = y^2z$.
\end{rmk}

\begin{thm} \label{thm:subregularsingularities}
Assume that $S = \spec k$ for $k$ an algebraically closed field and let $(G, P, \mu)$ be a subregular Harder-Narasimhan class. Assume that $(G, P, \mu)$ is not of type $A_1$ (resp., $(G, P, \mu)$ is of type $A_1$ and $k$ does not have characteristic $2$) and let $Z = \mrm{Ind}_L^G(Z_0) \to \bun_{G, rig}$ be the equivariant slice constructed in the proof of Theorem \ref{thm:subregularsliceexistence} (resp., Remark \ref{rmk:sl2slice}). Then the stack $\chi_Z^{-1}(0) \subseteq Z$ can be constructed as follows.
\begin{enumerate}[(1)]
\item \label{itm:subregularsingularities1} If $(G, P, \mu)$ is of type $A$ (but not $A_1$), then then there are two line bundles $L_1$ and $L_2$ with $\deg L_1 + \deg L_2 = l + 1$ such that $\chi_Z^{-1}(0)$ is obtained by gluing the corresponding line bundle on $E \sqcup E$ along the canonical map $E \sqcup E \to E$. In particular, $\chi_Z^{-1}(0)$ has singularities of type $A_\infty$ only.
\item \label{itm:subregularsingularities2} If $(G, P, \mu)$ is of type $B$ (resp., $C$, $D$), then there exists a line bundle $L$ on $E$ of degree $l - 6$ (resp., $l - 4$, $l - 8$) such that $\chi_Z^{-1}(0)$ is obtained by gluing $L$ along a degree $2$ map $E \to \mb{P}(1, 2)$ (resp., $E \to \mb{P}^1$) branched over $3$ (resp., $4$) points. The singularities are of type $A_\infty$ at the non-branch points of $\mb{P}(1, 2)$ (resp., $\mb{P}^1$) and, if the characteristic of $k$ is not $2$, of type $D_\infty$ at the branch points.
\item \label{itm:subregularsingularities3} If $(G, P, \mu)$ is of type $E$ (resp., $F$, $G$, $A_1$), then $\chi_Z^{-1}(0)$ is the cone over $E$ obtained by contracting the zero section of a line bundle $L$ on $E$ of degree $l - 9$ (resp., $l - 5$, $l - 3 = -1$, $-4$) to a point. The singularity is simply elliptic of degree $9 - l$ (resp., $5 - l$, $3 - l$, $4$).
\end{enumerate}
\end{thm}
\begin{proof}
We first prove \eqref{itm:subregularsingularities3}. If $(G, P, \mu)$ is of type $A_1$, then the claim is proved in Proposition \ref{prop:sl2sing} below. So assume that $(G, P, \mu)$ is of type $E$ (resp., $F$, $G$).

By construction, $\chi_Z^{-1}(0)$ is affine, and the open subset
\[ \chi_Z^{-1}(0)^{reg} = \chi_Z^{-1}(0) \times_{\bun_{G, rig}} \bun_{G, rig}^{reg}\]
is big, where $\bun_{G, rig}^{reg} \subseteq \bun_{G, rig}$ is the open substack of regular bundles of \cite[Proposition 4.4.6]{davis19}. So choosing any $y \colon \spec k \to 0_{\Theta_Y^{-1}}$, we have
\[ \chi_Z^{-1}(0) = \spec H^0(\chi_Z^{-1}(0), \mc{O}) = \spec H^0(\chi_Z^{-1}(0)^{reg}, \mc{O}) = \spec H^0(\tilde{\chi}_Z^{-1}(y)^{reg}, \mc{O}),\]
where $\tilde{\chi}_Z^{-1}(y)^{reg} = \tilde{\chi}_Z^{-1}(y) \cap \psi_Z^{-1}(\chi_Z^{-1}(0)^{reg}) \cong \chi_Z^{-1}(y)^{reg}$. But by Proposition \ref{prop:subregularresolutions2} (and the fact that $\alpha_i$ is not a special root), $\tilde{\chi}_Z^{-1}(y)^{reg} = (D_1)_y \setminus E$ is the complement of the zero section in the line bundle $L^{-1} = (D_1)_y$ over $E = \{y\} \times \mrm{Pic}^0(E)$, which has (negative) degree $l - 9$ (resp., $l - 5$, $l - 3$). So
\[ \chi_Z^{-1}(0) = \spec H^0((D_1)_y \setminus E, \mc{O}) = \spec \bigoplus_{n \geq 0} H^0(E, L^{\otimes n})\]
is a cone over $E$ of the asserted degree.

To prove \eqref{itm:subregularsingularities1} and \eqref{itm:subregularsingularities2}, we argue as follows. Since $\psi_{Z, y}{\vphantom{p}}_*\mc{O} = \mc{O}$ by Proposition \ref{prop:steinfactorisation} below and $\chi_Z^{-1}(0) \to Z_0$ is affine, we have
\[ \chi_Z^{-1}(0) = \spec_{Z_0} \pi_*\mc{O}_{\bar{D}_y} = \spec_{Z_0} \pi_*\mc{O}_{D_y},\]
for any choice of $y \colon \spec k \to 0_{\Theta_Y^{-1}}$, where $\pi \colon \tilde{\chi}_Z^{-1}(y) \to Z_0$ is the natural morphism and we write
\[ D = D_{\alpha_i^\vee}(Z) + D_{\alpha_j^\vee}(Z) + D_{\alpha_i^\vee + \alpha_j^\vee}(Z) \quad \text{and} \quad \bar{D} = \tilde{\chi}_Z^{-1}(0_{\Theta_Y^{-1}}).\]
Using Theorem \ref{thm:introsubregularresolutions} and Propositions \ref{prop:subregularresolutions2} and \ref{prop:subregularresolutions4}, it is easy to see that
\[ \pi_* \mc{O}_{D_y} \cong \pi_*\mc{O}_{(D_1)_y} \times_{\pi_*\mc{O}_E} \pi_*\mc{O}_{(D_1')_y},\]
where we have identified $\{y\} \times \mrm{Pic}^0(E) = D_{\alpha_i^\vee}(Z)_y \cap D_{\alpha_j^\vee}(Z)_y$ with $E$ and by mild abuse of notation we have also written $\pi$ for the morphisms $(D_1)_y \to Z_0$, $(D_1')_y \to Z_0$ and $E \to Z_0$. In type $A$, $L_1 = (D_1)_y$ and $L_2 = (D_1')_y$ are line bundles satisfying $\deg L_1 + \deg L_2 = l + 1$ by Lemma \ref{lem:linebundledegrees} below, which proves \eqref{itm:subregularsingularities1}. In type $B$ (resp., $C$, $D$), $L = (D_1)_y$ is a line bundle on $E$ of the desired degree by Lemma \ref{lem:linebundledegrees}, $\pi_*\mc{O}_{(D_1')_y} = \mc{O}_{Z_0}$, and $E \to Z_0 = \mb{P}(1, 2)$ (resp., $\mb{P}^1$) has degree $2$ by Lemma \ref{lem:degree2}. Since any degree $2$ map $E \to \mb{P}(1, 2)$ (resp., $E \to \mb{P}^1$) is branched over $3$ (resp., $4$) points, \eqref{itm:subregularsingularities2} now follows.
\end{proof}

In order to prove the lemmas quoted in the proof of Theorem \ref{thm:subregularsingularities}, we will appeal to the following formula for the canonical bundle of $\tbun_{G, rig}$.

\begin{prop}
There exists a line bundle $M$ on $\bun_{G, rig}$ such that
\[ K_{\tbun_{G, rig}/\bun_{G, rig}} \cong \psi^*M \otimes \mc{O}\left(\sum_{\mu \in \mb{X}_*(T)_+} (-2 + \langle \rho, \mu \rangle) D_\mu \right).\]
\end{prop}
\begin{proof}
This is an immediate consequence of \cite[Theorem 4.6.1]{davis19a}.
\end{proof}

\begin{cor} \label{cor:canonicalformula}
For any slice $Z \to \bun_{G, rig}$, there exists a line bundle $M$ on $Z$ such that
\[ K_{\tilde{Z}/Z} = \psi_Z^*M \otimes \mc{O}\left(\sum_{\mu \in \mb{X}_*(T)_+} (-2 + \langle \rho, \mu \rangle)D_\mu(Z)\right).\]
\end{cor}

\begin{lem} \label{lem:degree2}
Assume that $(G, P, \mu)$ is of type $B$, $C$ or $D$. Then for any $y \colon \spec k \to Y = 0_{\Theta_Y^{-1}}$, the morphism $\mrm{Pic}^0(E) = \{y\} \times \mrm{Pic}^0(E) \subseteq \tilde{\chi}_Z^{-1}(y) \to Z_0$ has degree $2$.
\end{lem}
\begin{proof}
By Corollary \ref{cor:canonicalformula},
\[ K_{\tilde{Z}} = \psi_Z^*K_Z \otimes K_{\tilde{Z}/Z} = \psi_Z^* M \otimes \mc{O}(-D_{\alpha_i^\vee}(Z) - D_{\alpha_j^\vee}(Z))\]
for some line bundle $M$ on $Z$. So by adjunction, we have
\begin{equation} \label{eq:classicaldegrees1}
K_{D_{\alpha_i^\vee}(Z)_y} = (K_{\tilde{Z}} \otimes \mc{O}(D_{\alpha_i^\vee}(Z)))|_{D_{\alpha_i^\vee}(Z)_y} = \psi_Z^*M|_{D_{\alpha_i^\vee}(Z)_y} \otimes \mc{O}(-E),
\end{equation}
where we write $E = \{y\} \times \mrm{Pic}^0(E) \subseteq D_{\alpha_i^\vee}(Z)_y$. To compute the degree of the finite morphism $E \to Z_0$, choose a $k$-point $z \in Z_0$ disjoint from the images of $\theta_k'(y)$ and the stacky point in type $B$, and let $F_z \cong \mb{P}^1_k$ be the fibre of $D_{\alpha_i^\vee}(Z)_y \to Z_0$ over $z$. By \eqref{eq:classicaldegrees1} and adjunction, the degree is the intersection product
\[ E \cdot F_z = - K_{D_{\alpha_i^\vee}(Z)_y} \cdot F_z = -(K_{D_{\alpha_i^\vee}(Z)_y} + F_z) \cdot F_z = - \deg K_{F_z} = 2,\]
which proves the lemma.
\end{proof}

\begin{lem} \label{lem:linebundledegrees}
Assume we are in the setup of Propositions \ref{prop:subregularresolutions2} and \ref{prop:subregularresolutions4} and fix a geometric point $y \colon \spec k \to Y$. Then we have the following.
\begin{enumerate}[(1)]
\item \label{itm:linebundledegrees1} If $(G, P, \mu)$ is of type $A$ (not $A_1$), then sum of the degrees of the line bundles $(D_1)_y$ and $(D_1')_y$ on $\mrm{Pic}^0(E)$ is $l + 1$.
\item \label{itm:linebundledegrees2} If $(G, P, \mu)$ is not of type $A$, then the degree of the line bundle $(D_1)_y$ on $\mrm{Pic}^0(E)$ is given in Table \ref{tab:typebcdefgdegrees}.
\end{enumerate}
\begin{table}[h]
\centering
\begin{tabular}{c|c|c|c|c|c|c}
Type & $B$ & $C$ & $D$ & $E$ & $F$ & $G$ \\ \hline
$\deg(D_1)_y$ & $l - 6$ & $l - 4$ & $l - 8$ & $l - 9$ & $l - 5$ & $l - 3$
\end{tabular}
\vspace{1ex}
\caption{Degree of $(D_1)_y$}
\label{tab:typebcdefgdegrees}
\end{table}
\end{lem}
\begin{proof}
To simplify the notation, identify $\mrm{Pic}^0(E) \subseteq (D_1)_y$ with $E$. The degree of the line bundle $(D_1)_y$ is equal to the self-intersection number $(E^2)_{(D_1)_y}$ of $E$ on the surface $(D_1)_y$.

First note that by Proposition \ref{prop:subregularresolutions2}, $D_{\alpha_j^\vee}(Z)_y$ is the iterated blowup of $(D_1)_y$ at $n_0 + 1$ points on $E$, so we have
\begin{equation} \label{eq:linebundledegrees1}
(E^2)_{(D_1)_y} = (E^2)_{D_{\alpha_j^\vee}(Z)_y} + n_0 + 1.
\end{equation}
Next, observe that we have
\begin{equation} \label{eq:linebundledegrees2}
0 = \tilde{\chi}_Z^{-1}(0_{\Theta_Y^{-1}}) \cdot E = (dD_{\alpha_i^\vee}(Z) + D_{\alpha_j^\vee}(Z) + D_{\alpha_i^\vee + \alpha_j^\vee}(Z))\cdot E = d(E^2)_{D_{\alpha_j^\vee}(Z)_y} + (E^2)_{D_{\alpha_i^\vee}(Z)_y} + 1,
\end{equation}
where $d = \frac{1}{2}(\alpha_i^\vee \mmid \alpha_i^\vee)$ and we have used the fact that $D_{\alpha_i^\vee}(Z)_y \cap D_{\alpha_j^\vee}(Z)_y = E$ and that the exceptional curve of the final blowup $D_{\alpha_i^\vee + \alpha_j^\vee}(Z)_y \cap D_{\alpha_j^\vee}(Z)_y$ meets $E$ transversely in a single point. Since $D_{\alpha_j^\vee}(Z)_y$ is the iterated blowup of the smooth surface $(D_1')_y$ of Proposition \ref{prop:subregularresolutions4} at $N$ points on $E$, we have
\[ (E^2)_{D_{\alpha_i^\vee}(Z)} = (E^2)_{(D_1')_y} - N,\]
and hence \eqref{eq:linebundledegrees1} and \eqref{eq:linebundledegrees2} give
\begin{equation} \label{eq:linebundledegrees3}
(E^2)_{(D_1)_y} = \frac{1}{d}(N - (E^2)_{(D_1')_y} - 1) + n_0 + 1.
\end{equation}

In type $A$, $d = 1$ and $N = n_1 + 1$, so \eqref{eq:linebundledegrees3} is equivalent to
\[ \deg (D_1)_y + \deg (D_1')_y = (E^2)_{(D_1)_y} + (E^2)_{(D_1')_y} = n_0 + n_1 + 1 = l + 1,\]
which proves \eqref{itm:linebundledegrees1}.

In types $B$, $C$ and $D$, we argue as follows. By Lemma \ref{lem:degree2}, $E$ is a smooth elliptic curve contained in a (possibly stacky) Hirzebruch surface $(D_1')_y$ mapping with degree $2$ to the base $Z_0 = \mb{P}(1, 2)$ or $\mb{P}^1$. It follows from a straightforward adjunction calculation that $E$ is an anticanonical divisor on $(D_1')_y$ and satisfies $(E^2)_{D_1'} = 6$ in type $B$ and $(E^2)_{(D_1')_y} = 8$ in types $C$ and $D$. Substituting into \eqref{eq:linebundledegrees3} gives the degrees in Table \ref{tab:typebcdefgdegrees}.

Finally, in types $E$, $F$ and $G$, note that by Corollary \ref{cor:canonicalformula}, we have
\[ K_{\tilde{Z}/Z} = \psi_Z^*M \otimes \mc{O}(-D_{\alpha_i^\vee}(Z) - D_{\alpha_j^\vee}(Z))\]
for some line bundle on $M$ on $Z$. Since $Z \cong \mb{A}^{l + 3}$ is an affine space, every line bundle on $Z$ is trivial, so
\[ K_{\tilde{Z}} = K_{\tilde{Z}/Z} \otimes \psi_Z^*K_Z \cong \mc{O}(-D_{\alpha_i^\vee}(Z) - D_{\alpha_j^\vee}(Z)).\]
By adjunction, we therefore have a linear equivalence
\[ K_{D_{\alpha_i^\vee}(Z)_y} \sim (K_{\tilde{Z}} + D_{\alpha_i^\vee}(Z))|_{D_{\alpha_i^\vee}(Z)_y} = -D_{\alpha_i^\vee}(Z)_y \cap D_{\alpha_j^\vee}(Z)_y = - E.\]
So $E \subseteq D_{\alpha_i^\vee}(Z)_y$ is an anticanonical divisor, from which it follows that $E \subseteq (D_1')_y$ is also an anticanonical divisor in the blow down. So from the explicit identification of the surface $(D_1')_y$ given in Proposition \ref{prop:subregularresolutions4} as either a Hirzebruch surface or $\mb{P}^2$, we have
\[ (E^2)_{(D_1')_y} = K_{(D_1')_y}^2 = \begin{cases} 9, & \text{in types}\; E\;\text{and}\;G, \\ 8, & \text{in type}\;\; F.\end{cases}\]
Substituting the values of $N$, $n_0$ and $d$ into \eqref{eq:linebundledegrees3} in each of the different cases gives the desired expressions for $(E^2)_{(D_1)_y}$.
\end{proof}

\begin{prop} \label{prop:steinfactorisation}
Fix any geometric point $y \colon \spec k \to \Theta_Y^{-1}$, and let
\[ \psi_y \colon \tilde{\chi}^{-1}(y) \longrightarrow \chi^{-1}(y) \]
be the pullback of the elliptic Grothendieck-Springer resolution. We have $\psi_y{\vphantom{p}}_* \mc{O} = \mc{O}$.
\end{prop}
\begin{proof}
Since $\chi^{-1}(y)$ is a local complete intersection, hence Cohen-Macaulay, it is enough to prove the claim on an open substack of the target whose complement has codimension at least $2$. We therefore reduce to proving that the map
\[ \psi_{Z, y} \colon \tilde{\chi}_Z^{-1}(y) \longrightarrow \chi_Z^{-1}(y) \]
satisfies $\psi_{Z, y}{\vphantom{p}}_*\mc{O} = \mc{O}$ for each of the slices $Z \to \bun_{G, rig}$ of Theorem \ref{thm:introsubregularsliceexistence}.

Note that by the proof of Lemma \ref{lem:connectedspringerfibres}, we have $\psi_Z'{\vphantom{p}}_*\mc{O} = \mc{O}$, where
\[ \psi_Z' \colon \tilde{Z} \longrightarrow Z \times_{\hat{Y}\sslash W} \Theta_Y^{-1}\]
is the natural morphism induced by $\psi_Z$. Since the domain and codomain of $\psi_Z'$ are both flat over $\Theta_Y^{-1}$, it is enough by base change to show that $\mb{R}^i\psi_Z'{\vphantom{p}}_*\mc{O} = 0$ for all $i > 0$. By equivariance, this will follow from the claim that $\mb{R}^i\psi_{Z, y}{\vphantom{p}}_*\mc{O} = 0$ for all $i > 0$ and all $y \in 0_{\Theta_Y^{-1}} = Y$.

Since $\chi_Z^{-1}(0) \to Z_0$ is affine by construction, it is enough to show that $\mb{R}^i\pi_* \mc{O} = 0$ for $i > 0$, where $\pi \colon \tilde{\chi}_Z^{-1}(y) \to Z_0$ is the natural morphism. This holds by inspection for the fibre over $y \in Y$ of the reduced normal crossings variety
\[ D = D_{\alpha_i^\vee}(Z) + D_{\alpha_j^\vee}(Z) + D_{\alpha_i^\vee + \alpha_j^\vee}(Z),\]
from the explicit descriptions of the components given by Theorem \ref{thm:introsubregularresolutions} and Proposition \ref{prop:subregularresolutions4}, using the fact that $\mb{R}f_*\mc{O} = \mc{O}$ whenever $f$ is either a $\mb{P}^1$-bundle or the blow up of a smooth surface at a point. This proves the claim in types $A$, $B$ and $D$. In type $C$, we claim that the morphism $\mb{R}\pi_*\mc{O}_{\bar{D}_y} \to \mb{R}\pi_* \mc{O}_{D_y}$ is a quasi-isomorphism, where $\bar{D} = \tilde{\chi}_Z^{-1}(0_{\Theta_Y^{-1}})$, from which the desired vanishing follows. To see this, note that we have a short exact sequence
\[ 0 \longrightarrow \mc{O}(-D)|_{D_{\alpha_i^\vee}(Z)} \longrightarrow \mc{O}_{\bar{D}} \longrightarrow \mc{O}_D \longrightarrow 0,\]
so it is enough to show that $\mb{R}^i\pi_*\mc{O}(-D)|_{D_{\alpha_i^\vee}(Z)_y} = 0$ for all $i$. From the explicit description of $D_{\alpha_i^\vee}(Z)_y$ given in Proposition \ref{prop:subregularresolutions4}, it is enough to show that $\mc{O}(-D)|_{D_{\alpha_i^\vee}(Z)_y}$ has degree $0$ on the exceptional curve $\gamma$ of the blowup and degree $-1$ on every irreducible fibre of $D_1' \to Z_0 = \mb{P}^1$. But since $\Theta_Y$ is trivial on $D_{\alpha_i^\vee}(Z)_y$, we have a linear equivalence
\[ -2D|_{D_{\alpha_i^\vee}(Z)_y} \sim  - D_{\alpha_j^\vee}(Z)_y \cap D_{\alpha_i^\vee}(Z)_y - D_{\alpha_i^\vee + \alpha_j^\vee}(Z)_y \cap D_{\alpha_i^\vee}(Z)_y = - E - \gamma,\]
from which the claim follows by Lemma \ref{lem:degree2}.
\end{proof}

\begin{prop} \label{prop:sl2sing}
Assume that $(G, P, \mu)$ is of type $A_1$, so that $G = SL_2$, $P = T$ and $\langle \varpi_1, \mu \rangle = -2$. Let $\spec k \to \bun_{T}^\mu = \bun_{\mb{G}_m}^{-2}$ be the slice classifying the line bundle $\mc{O}(-2O_E)$ of Remark \ref{rmk:sl2slice} and let $Z = \mrm{Ind}_T^{SL_2}(\spec k) \to \bun_{SL_2}$ be the induced equivariant slice. Then the unstable fibre $\chi_Z^{-1}(0)$ is isomorphic to the affine cone over $E$ obtained by contracting the zero section of a degree $-4$ line bundle to a point.
\end{prop}
\begin{proof}
If we identify $\bun_{SL_2}$ with the stack of rank $2$ vector bundles with trivial determinant, then the slice $Z$ is nothing but the vector space
\[ Z = \mrm{Ext}^1(\mc{O}(2O_E), \mc{O}(-2O_E)) \cong H^1(E, \mc{O}(-4O_E)),\]
with its tautological map to $\bun_{SL_2}$. To describe the unstable locus, note that a point $z \in Z$ corresponds to an unstable extension
\[ 0 \longrightarrow \mc{O}(-2O_E) \longrightarrow V_z \longrightarrow \mc{O}(2O_E) \longrightarrow 0 \]
if and only if there exists a degree $1$ line bundle $L$ on $E$ such that $z$ is in the ($1$-dimensional) kernel of the map
\[ \mrm{Ext}^1(\mc{O}(2 O_E), \mc{O}(-2 O_E)) \longrightarrow \mrm{Ext}^1(L, \mc{O}(-2O_E)) \]
induced by the unique (up to scale) nonzero morphism $L \to \mc{O}(2O_E)$. We deduce that $\chi_Z^{-1}(0)\setminus \{0\}$ must be a $\mb{G}_m$-torsor over $\mrm{Pic}^1(E) \cong E$, so the normal variety $\chi_Z^{-1}(0)$ must be an affine cone over $E$ as claimed.

To identify the degree, observe that since $Z \to \bun_{SL_2}$ is an equivariant slice with equivariance group $\mb{G}_m$ and weight $2$ by Proposition \ref{prop:inductionequivariantslice}, the morphism $\chi_Z \colon \mb{A}^4 \cong Z \to \hat{Y}\sslash W \cong \mb{A}^2$ is equivariant with respect to the weight $1$ action on $\mb{A}^4$ and the weight $2$ action on $\mb{A}^2$. So taking projectivisations, we deduce that the elliptic curve $(\chi_Z^{-1}(0)\setminus \{0\})/\mb{G}_m$ is presented as an intersection of two quadric surfaces in $\mb{P}^3$, from which we deduce that the polarising line bundle has degree $4$. The proposition now follows.
\end{proof}

\subsection{Deformation theory} \label{subsection:deformations}

In this subsection, we study the deformation theory of the unstable varieties $\chi_Z^{-1}(0)$ of \S\ref{subsection:singularities}. As in the previous subsection, We will assume for simplicity that $S = \spec k$ for some algebraically closed field $k$.

\begin{defn}
Let $H$ be a torus with character group $\mb{X}^*(H)$, let $\mb{X}^*(H)_+ \subseteq \mb{X}^*(H)$ be a sub-monoid (without unit), and let $X$ be an algebraic stack over $\spec k$ with $H$-action.
\begin{enumerate}[(1)]
\item An \emph{$\mb{X}^*(H)_+$-weighted deformation ring} is an $\mb{X}^*(H)$-graded Noetherian $k$-algebra
\[ R = \bigoplus_{\lambda \in -\mb{X}^*(H)_+ \cup \{0\}} R_\lambda \]
such that $R_0 = k$. Given such an $R$, we write
\[ \hat{R} = \prod_{\lambda \in -\mb{X}^*(H)_+ \cup \{0\}} R_\lambda \]
for the completion at the maximal ideal $\mf{m}_R = \bigoplus_{\lambda \in -\mb{X}^*(H)_+} R_\lambda$.
\item An \emph{$\mb{X}^*(H)_+$-weighted deformation of $X$} over an $\mb{X}^*(H)_+$-weighted deformation ring $R$ is a flat $H$-equivariant morphism $\bar{X} \to \spf \hat{R}$ of formal stacks equipped with an $H$-equivariant isomorphism $\bar{X}_s \cong X$, where $s \colon \spec k \to \spf \hat{R}$ is the unique ($H$-fixed) point.
\item We say that an $\mb{X}^*(H)_+$-weighted deformation $\bar{X} \to \spf R$ is \emph{versal} if for every surjective (graded) homomorphism $R' \to R''$ of $\mb{X}^*(H)_+$-weighted deformation rings, every homomorphism $\phi' \colon R \to R''$ and every weighted deformation $\bar{X}_{R'} \to \spf \hat{R}'$ with an isomorphism $\alpha \colon \bar{X}_{R'} \times_{\spf \hat{R}'} \spf \hat{R}'' \cong \bar{X} \times_{\spf \hat{R}} \spf \hat{R}''$, there exists a lift $\phi \colon R \to R'$ and an isomorphism $\bar{X}_{R'} \cong \bar{X} \times_{\spf \hat{R}} \spf \hat{R}'$ lifting $\alpha$.
\item We say that a versal $\mb{X}^*(H)_+$-weighted deformation $\bar{X} \to \spf R$ is \emph{miniversal} (or \emph{semi-universal}) if for all $\mb{X}^*(H)_+$-weighted deformation rings $R'$ and pairs $\phi, \phi' \colon R \to R'$ of graded homomorphisms with $\bar{X} \times_{\spf \hat{R}, \phi} \spf \hat{R}' \cong \bar{X} \times_{\spf \hat{R}, \phi'} \spf \hat{R}'$, the maps
\[ d\phi, d\phi' \colon T_{s'} \spec R' \longrightarrow T_s \spec R \]
on tangent spaces at fixed points are equal.
\end{enumerate}
\end{defn}

\begin{rmk}
Note that a versal (resp., miniversal) $\mb{X}_*(H)_+$-weighted deformation need not be versal (resp., miniversal) as a plain unweighted deformation.
\end{rmk}

Weighted deformation theory in this sense works in more or less the same way as unweighted deformation theory for schemes. (See, for example, \cite{hartshorne10} or \cite[Chapitre III]{illusie71} for the unweighted case.) For example, we have the following.

\begin{prop} \label{prop:weighteddeformations1}
Let $\mb{T}_X$ be the tangent complex to $X$ (the $\mc{O}_X$-linear dual to the cotangent complex). Let $R' \to R$ is a surjection of Artinian $\mb{X}_*(H)_+$-weighted deformation rings with kernel $I$ satisfying $\mf{m}_{R'} I = 0$, and let $\bar{X} \to \spf \hat{R} = \spec R$ be a weighted deformation of $X$.
\begin{enumerate}[(1)]
\item There is an $H$-invariant obstruction class $\mrm{ob} \in (H^2(X, \mb{T}_X)\otimes I)^H$ such that $\bar{X}$ lifts to a weighted deformation over $R'$ if and only if $\mrm{ob} = 0$.
\item If lifts exist, then the set of isomorphism classes of $H$-equivariant lifts form a torsor under the group $(H^1(X, \mb{T}_X)\otimes I)^H \subseteq H^1(X, \mb{T}_X) \otimes I$.
\end{enumerate}
\end{prop}

As a consequence, we deduce the following via the usual argument for existence and behaviour of a miniversal deformation.

\begin{prop} \label{prop:weighteddeformations2}
Assume that $\mb{X}^*(H)_+$-weighted subspace $H^1(X, \mb{T}_X)_+$ of $H^1(X, \mb{T}_X)$ is finite dimensional. Then there exists a miniversal $\mb{X}_*(H)_+$-weighted deformation $\bar{X} \to \spf \hat{R}$, and $\spf \hat{R}$ has tangent space $H^1(X, \mb{T}_X)_+$ at the fixed point. Moreover, if the $\mb{X}^*(H)_+$-weighted subspace $H^2(X, \mb{T}_X)_+$ of $H^2(X, \mb{T}_X)$ vanishes, then
\[ R \cong \sym(H^1(X, \mb{T}_X)_+)^\vee \]
as $\mb{X}^*(H)$-graded rings.
\end{prop}

\begin{rmk}
The obstruction class of Proposition \ref{prop:weighteddeformations1} is the same as the obstruction for lifting ordinary (unweighted) deformations. It follows that if $\bar{X} \to \spf \hat{R}$ is an unweighted miniversal deformation such that $R$ is $\mb{X}^*(H)$-graded and the $H$-action on $X$ lifts to a compatible action on $\bar{X}$, then the restriction to $\spf \hat{R}'$ is an $\mb{X}^*(H)_+$-graded miniversal deformation, where $R'$ is the quotient of $R$ by the ideal generated by all weight spaces of the maximal ideal of $R$ with weights $\not\in \mb{X}_*(H)_+$.
\end{rmk}

We now turn to the weighted deformation theory of the singularities of \S\ref{subsection:singularities}.

\begin{lem} \label{lem:contractiondeformations}
Assume we are in the setup of Construction \ref{cons:gluing}, and assume moreover that either $k$ has characteristic not $2$ or that the morphism $\pi \colon X \to X'$ is unramified. Let $X''$ be the surface obtained by gluing a line bundle $L$ on $X$ along $\pi$. Let $H$ be a torus equipped with a sub-monoid $\mb{X}^*(H)_+ \subseteq \mb{X}^*(H)$ (not containing $0$) acting on the line bundle $L$ (and hence on the surface $X''$) in such a way that for every $x \in X'$, the weights $\lambda_1$ and $\lambda_2$ at the two preimages of $x$ satisfy
\[ \lambda_1, \lambda_2 \in \mb{X}^*(H)_+ \quad \text{and} \quad \lambda_1 - \lambda_2, \lambda_2 - \lambda_1 \not\in \mb{X}^*(H)_+.\]
Then, in the notation of Proposition \ref{prop:weighteddeformations2},
\[  H^1(X'', \mb{T}_{X''})_+ = H^0(X', I^\vee) \oplus \bigoplus_{x \in \mrm{ram}(\pi)} L_x^{-1} \otimes I_{\pi(x)}^\vee, \quad \text{and} \quad H^2(X'', \mb{T}_{X''})_+ = H^1(X', I^\vee), \]
where $\mrm{ram}(\pi) \subseteq X$ is the set of ramification points of $\pi$ and
\[ I = \ker(\sym^2 \pi_*(L^{-1}) \longrightarrow \pi_*(L^{-2})).\]
\end{lem}
\begin{proof}
First note that since $H$ acts on $L$ with weights in $\mb{X}_*(H)_+$ it follows that $H^i(X'', f^*\mb{T}_{X'})$ has weights in $-\mb{X}^*(H)_+$ (and hence none in $\mb{X}_*(H)_+$) for all $i$, where $f \colon X'' \to X'$ is the structure map. So we can replace $\mb{T}_{X''}$ with $\mb{T}_{X''/X'}$ in the statement of the lemma. We can also assume without loss of generality that $X'$ is connected, so that the weights $\{\lambda_1, \lambda_2\}$ of $H$ on $L$ restricted to $\pi^{-1}(x)$ are independent of $x \in X'$.

By definition, we have $X'' = \spec_{X'} \mc{R}$, where $\mc{R}$ is the sheaf of algebras
\[ \mc{R} = \pi_* \bigoplus_{n \geq 0} L^{-n} \times_{\pi_*\mc{O}_X} \mc{O}_{X'} \cong \frac{\sym \pi_*(L^{-1})}{\sym \pi_*(L^{-1}) \otimes I}.\]
So $X''$ is a local complete intersection over $X'$ and the pushforwards of the relative tangent complex along the structure map $f \colon X'' \to X'$ is
\[ f_*\mb{T}_{X''/X'} = [\sym \pi_*(L^{-1}) \otimes \pi_*(L^{-1})^\vee \to \sym \pi_*(L^{-1}) \otimes I^\vee],\]
concentrated in cohomological degrees $0$ and $1$. Since $\pi_*(L^{-1})$ has weights $-\lambda_1$ and $-\lambda_2$ and $I^\vee$ has weight $\lambda_1 + \lambda_2$, it follows that the $\mb{X}^*(H)_+$-weight part of $f_*\mb{T}_{X''/X'}$ is
\[ (f_*\mb{T}_{X''/X'})_+ = I^\vee[-1] \oplus [\pi_*(L^{-1})^\vee \to \pi_*(L^{-1}) \otimes I^\vee].\]
A local computation now shows that the second term has vanishing cohomology away from the ramification points. If $k$ does not have characteristic $2$, then $\pi$ is given in formal local coordinates by $x \mapsto x^2$ near any ramification point, from which one can show that the natural map
\[ [\pi_*(L^{-1})^\vee \to \pi_*(L^{-1}) \otimes I^\vee] \longrightarrow \bigoplus_{x \in \mrm{ram}(\pi)} L^{-1}_x \otimes I^\vee_{\pi(x)}[-1] \]
is a quasi-isomorphism, where the terms of the direct sum on the right are interpreted as skyscraper sheaves at the branch points $\pi(x) \in X'$. The lemma now follows.
\end{proof}

\begin{lem} \label{lem:singularityweights}
Let $(G, P, \mu)$ be a subregular Harder-Narasimhan class, and identify the torus $Z(L)_{rig}$ with $\mb{G}_m$ (resp., $\mb{G}_m \times \mb{G}_m$ in type $A$) via the cocharacter $-\varpi_i^\vee \colon \mb{G}_m \to Z(L)_{rig}$ (resp., $(-\varpi_i^\vee, -\varpi_{i + 1}) \colon \mb{G}_m \times \mb{G}_m \to Z(L)_{rig}$), where $i$ is as in Notation \ref{notation:dynkindecomposition}. The weights of the $Z(L)_{rig}$-action on the line bundles in Theorem \ref{thm:subregularsingularities} are as follows.
\begin{enumerate}[(1)]
\item If $(G, P, \mu)$ is of type $A$ (but not $A_1$), then the line bundles $L_1$ and $L_2$ have weights $(1, 0)$ and $(0, 1)$.
\item In all other cases, the line bundle $L$ has weight $1$.
\end{enumerate}
\end{lem}
\begin{proof}
We deduce this from the weights of the $Z(L)_{rig}$-action on the affine space $\hat{Y}\sslash W$ and the fibres of the affine space bundle $Z \to Z_0$. By construction, the $Z(L)_{rig}$-weights of $\hat{Y}\sslash W$ are the canonical $\mb{G}_m$-weights multiplied by $(\mu \mmid -)$. By a theorem of Looijenga \cite[Theorem 3.4]{looijenga76}, these $\mb{G}_m$ weights are $1, g_1, \ldots, g_l$, where $g_i$ are the \emph{coroot integers} defined by
\[ g_1 \alpha_1^\vee + \cdots + g_l\alpha_l^\vee = \tilde{\alpha}^\vee,\]
where $\tilde{\alpha}$ is the highest root of $G$. The $Z(L)_{rig}$-weights on $Z$, on the other hand, can be computed using, for example, the formula in \cite[Proposition 4.1.7]{davis19} for the weight multiplicities in a parabolic induction. These are all given in Table \ref{tab:weights} below.
\begin{table}[h]
\newcommand{\T}{\rule{0pt}{2.6ex}}
\renewcommand{\B}{\rule[-1.2ex]{0pt}{0pt}}
\begin{tabular}{l|l|l|l}
Type & $\mb{G}_m$-weights of $\hat{Y}\sslash W$ & $(\mu \mmid -)$ & $Z(L)_{rig}$-weights of $Z$ \\
\hline \T $A_1$ & $1^2$ & $2$ & $1^4$ \B \\
\hline \T $A_l, l > 1$ & $1^{l + 1}$ & $(1, 1)$ & $(1, 0)^1(0, 1)^1(1, 1)^l$ \B \\
\hline \T $B_l$ & $1^32^{l - 2}$ & $1$ & $1^52^{l - 3}$ \B \\
\hline \T $C_l$ & $1^{l + 1}$ & $2$ & $1^22^l$ \B \\
\hline \T $D_l$ & $1^42^{l - 3}$ & $1$ & $1^62^{l - 4}$ \B \\
\hline \T $E_5$ & $1^42^2$ & $1$ & $1^8$ \B \\
\hline \T $E_6$ & $1^32^33^1$ & $1$ & $1^62^3$ \B \\
\hline \T $E_7$ & $1^22^23^24^1$ & $1$ & $1^42^43^2$ \B \\
\hline \T $E_8$ & $1^12^23^24^25^16^1$ & $1$ & $1^22^33^34^25^1$ \B \\
\hline \T $F_3$ & $1^32^1$ & $2$ & $1^22^4$ \B \\
\hline \T $F_4$ & $1^22^23$ & $2$ & $1^12^33^14^2$ \B \\
\hline \T $G_2$ & $1^22^1$ & $3$ & $1^12^13^3$
\end{tabular}
\vspace{1ex}
\caption{Weights of the subregular slices}
\label{tab:weights}
\end{table}
In type $A_l$, $l > 1$, observe that the fixed loci of $\{1\} \times \mb{G}_m$ and $\mb{G}_m \times \{1\} \subseteq Z(L)_{rig}$ are necessarily contained in the zero fibre $\chi_Z^{-1}(0)$. Since these are both line bundles over $E = Z_0$, they must be $L_1$ and $L_2$. So the weights of these are $(1, 0)$ and $(0, 1)$ as claimed.

In the other types, assume for a contradiction that $L$ has weight $w > 1$. (Note that it cannot have negative weight, since all weights of $Z$ are positive.) So the whole zero fibre $\chi_Z^{-1}(0)$ is contained in the fixed locus of $\mu_w$, and hence the images of the weight $-1$ generators of the polynomial ring $\Gamma(\hat{Y}\sslash W, \mc{O})$ span the weight $-1$ part of $p_* \mc{O}_Z$, where $p \colon Z \to Z_0$ is the natural map. But in each case the multiplicity of the weight $-1$ in $p_*\mc{O}_Z$ is larger than in $\Gamma(\hat{Y}\sslash W, \mc{O})$, so this is a contradiction, and the lemma is proved.
\end{proof}

\begin{lem} \label{lem:simplyellipticweights}
Let $X$ be a cone over $E$ of degree $1 \leq d \leq 4$, and assume that $(\mrm{char}(k), d) \neq (2, 2), (3, 3)$. Then the miniversal $\mb{Z}_{>0}$-deformation space of $X$ is an affine space with the same weights as $\hat{Y}\sslash W$ in type $E_{9 - d}$ given in Table \ref{tab:weights}.
\end{lem}
\begin{proof}
Except for $d = 4$, this is pointed out in \cite[Theorem 6.24]{grojnowski-shep19}: the result is well-known in characteristic $0$, and is due to M. Hirokado \cite[Theorem 4.4]{hirokado04} in positive characteristic.

For $d = 4$, we argue as follows. In this case, the cone $X$ is a complete intersection
\[ X = \spec \frac{k[x_1, x_2, x_3, x_4]}{(f_1, f_2)},\]
where $f_1, f_2$ are homogeneous polynomials of degree $2$. The deformation theory of $X$ is unobstructed, and weight $d$ part of the tangent space is the degree $2 - d$ part of the cokernel of the $2 \times 4$ Jacobian matrix
\[ A = \left(\frac{\partial f_i}{\partial x_j}\right)\]
with entries in $R = k[x_1, x_2, x_3, x_4]/(f_1, f_2)$. So the miniversal $\mb{Z}_{>0}$-weighted deformation space is an affine space with weights $1^42^2$ as long as intersection of the kernels of the two Hessian matrices
\[ H_i = \left(\frac{\partial^2 f_i}{\partial x_j \partial x_k}\right)_{1 \leq j, k \leq 4}, \quad i = 1, 2 \]
is zero.

If the characteristic of $k$ is not $2$, then
\[ f_i(x) = \frac{1}{2}x^tH_ix, \quad \text{and} \quad df_i = dx^t H_i x.\]
So any nonzero vector $v$ in the intersection of the kernels gives a singular point in the curve $E = \mathrm{Proj}(R)$, which is a contradiction.

If the characteristic of $k$ is $2$, then $(H_i)_{j, j} = 0$ (so $H_i$ is the matrix of an alternating form on $k^4$), and $f_i$ is of the form
\[ f_i(x) = \sum_{1 \leq j < k \leq 4} (H_i)_{j, k} x_jx_k + \sum_{j = 1}^4 a_{i, j}x_j^2 \]
for some vectors $a_i = (a_{i, 1}, a_{i, 2}, a_{i, 3}, a_{i, 4}) \in k^4$. The variety of all possible tuples $(H_1, H_2, a_1, a_2)$ such that $H_1$ and $H_2$ have a common vector in their kernels is irreducible (it admits a surjection from a vector bundle over the projective space $\mb{P}^3$), and the subset of such tuples defining a smooth elliptic curve is an open subset; we will show that this subset must be empty.

Consider the open subset of tuples as above such that $\dim \ker H_1 = \dim \ker H_2 = 2$, $\dim \ker H_1 \cap \ker H_2 = 1$, $a_1$ and $a_2$ are linearly independent, and some vector in the span of $a_1$ and $a_2$ does not lie in $\ker H_1 + \ker H_2$. For any such tuple, after performing an invertible linear transformation on $f_1$ and $f_2$, and changing basis on $k^4$, we can arrange that
\[ H_1 = \left(\begin{matrix} 0 & 0 & 0 & 0 \\ 0 & 0 & 0 & 1 \\ 0 & 0 & 0 & 0 \\ 0 & 1 & 0 & 0 \end{matrix}\right), \quad H_2 = \left(\begin{matrix} 0 & 0 & 0 & 0 \\ 0 & 0 & 0 & 0 \\ 0 & 0 & 0 & 1 \\ 0 & 0 & 1 & 0 \end{matrix}\right), \quad a_1 = \left(\begin{matrix} 0 \\ 0 \\ 0 \\ 1 \end{matrix} \right), \quad a_2 = \left(\begin{matrix} a \\ b \\ c \\ 0 \end{matrix}\right),\]
for some $a, b, c \in k$. A straightforward Jacobian calculation shows that $\mrm{Proj}(R)$ is singular in this case. So this nonempty open subset is disjoint from the open subset yielding smooth elliptic curves, so the latter must be empty by irreducibility, and we are done.
\end{proof}

\begin{thm} \label{thm:subregulardeformations}
Assume we are in the setup of Theorem \ref{thm:subregularsingularities} and, moreover, that $k$ does not have characteristic $2$ if $(G, P, \mu)$ is of type $A_1$, $B$, $C$, $E_7$ or $F_3$, and that $k$ does not have characteristic $3$ if $(G, P, \mu)$ is of type $E_6$. Then (the formal completion of) the family $\chi_Z \colon Z \to \hat{Y}\sslash W$ is a miniversal $\mb{Z}_{>0}(\mu \mmid -)$-weighted deformation of $\chi_Z^{-1}(0)$ with respect to the action of the torus $Z(L)_{rig}$.
\end{thm}
\begin{proof}
We first argue that there is no non-constant $\mb{G}_m$-orbit (equivalently, $Z(L)_{rig}$-orbit) closure in $\hat{Y}\sslash W$ on which the family $\chi_Z$ is equivariantly trivial over the completion at $0$. To see this, note that every $\mb{G}_m$-orbit closure is of the form $q(\Theta_{Y, y}^{-1})$, where $q \colon \Theta_Y^{-1} \to \hat{Y}\sslash W$ is the quotient map and $\Theta_{Y, y}^{-1} \cong \mb{A}^1$ is the fibre of the line bundle $\Theta_Y^{-1}$ over a point $y \in Y$. If the pullback $X \to \mb{A}^1$ of $Z$ to this fibre is equivariantly trivial over the formal completion at $0$, then it is trivial relative to $Z_0$ (since there are no deformations of the map $\chi_Z^{-1}(0) \to Z_0$ of the relevant weights). Since $X$ is affine over $Z_0$, the equivariant formal trivialisation therefore lifts uniquely to an equivariant isomorphism $X \cong \chi_Z^{-1}(0) \times \mb{A}^1$.

Now, there is a simultaneous log resolution $\pi \colon \tilde{X} = \tilde{Z}_y \to X \cong \chi_Z^{-1}(0) \times \mb{A}^1$ over $\mb{A}^1$, relative to the divisor $0 \in \mb{A}^1$. Moreover, for any $t \in \mb{A}^1 \setminus \{t\}$, the map $\pi_t \colon \tilde{X}_t \to X_t$ is an isomorphism over the dense open locus of points in $X_t$ whose associated $G$-bundle is regular, and has positive dimensional fibres over all other points. Since $\tilde{X}_t$ is smooth, it follows that $X_t = \chi_Z^{-1}(0)$ is regular in codimension $1$. This contradicts Theorem \ref{thm:subregularsingularities} in the classical types $A_l$ ($l > 1$), $B$, $C$ and $D$. In types $A_1$, $E$, $F$ and $G$, we instead note that Corollary \ref{cor:canonicalformula} implies that $\tilde{X}_t$ has trivial canonical bundle (since $Z_0 = \spec k$ in these cases) for $t \neq 0$. In particular, from adjunction, every projective curve in $\tilde{X}_t$ is rational with self-intersection $-2$. But $\tilde{X}_t$ is a resolution of the elliptic cone $\chi_Z^{-1}(0)$ and is therefore birational (over $\chi_Z^{-1}(0)$) to a line bundle over an elliptic curve. Since a birational map between smooth surfaces projective over a common base is a sequence of blowups at points and contractions of $(-1)$-curves, this implies that $\tilde{X}_t$ must contain a projective curve whose normalisation is elliptic, which is a contradiction. So the deformation is formally nontrivial on all orbit closures as claimed.

Now let $\mf{X} \to \spf(\hat{R})$ be a miniversal $\mb{Z}_{>0}(\mu \mmid -)$-weighted deformation of $\chi_Z^{-1}(0)$. Then the completion of $\chi_Z \colon Z \to \hat{Y}\sslash W$ is the pullback of $\mf{X}$ along some $Z(L)_{rig}$-equivariant map $\hat{Y}\sslash W \to \spec R$ such that the preimage of the origin is (set-theoretically) the fixed point. We will show that $\spec R$ is an affine space with linear $Z(L)_{rig}$-action of the same weights as $\hat{Y}\sslash W$, from which it follows that $\hat{Y}\sslash W \to \spec R$ must be an isomorphism.

In type $E$, the claim is proved in Lemma \ref{lem:simplyellipticweights}.

In type $A_l$, $l > 1$, in the notation of Lemma \ref{lem:singularityweights}, $Z(L)_{rig} \cong \mb{G}_m \times \mb{G}_m$ acts on $L_1$ and $L_2$ with weights $(1, 0)$ and $(0, 1)$ respectively and $I = L_1^{-1} \otimes L_2^{-1}$ is a line bundle of degree $-l - 1$. So applying Lemma \ref{lem:contractiondeformations} and Proposition \ref{prop:weighteddeformations2}, we have that the miniversal $\mb{Z}_{> 0}^2$-weighted deformation is an affine space with weights $(1, 1)^{l + 1}$. Since $(\mu \mmid -) = (1, 1)$ in this presentation, this is also a miniversal $\mb{Z}_{> 0}(\mu \mmid -)$-weighted deformation with weights $(\mu \mmid -)^{l + 1}$ as required to prove.

In type $B$, we identify the line bundle $I$ on $\mb{P}(1, 2)$ of Lemma \ref{lem:contractiondeformations} as follows. First, note that since line bundles on $\mb{P}(1, 2)$ are rigid, we may assume without loss of generality that $L = \mc{O}((l - 6)p) = \pi^*\mc{O}_{\mb{P}(1, 2)}(l - 6)$, where $p \in E$ maps to the stacky point of $\mb{P}(1, 2)$. So we have
\[ \pi_*L^{-1} = \pi_*\mc{O} \otimes \mc{O}(6 - l) \quad \text{and} \quad \pi_*(L^{-2}) = \pi_*\mc{O} \otimes \mc{O}(12 - 2l).\]
Since $\pi \colon E \to \mb{P}(1, 2)$ is finite and flat of degree $2$ (and the characteristic is not $2$), we have $\pi_*\mc{O} = \mc{O} \oplus \mc{O}(d)$ for some $d \in \mb{Z}$. Since $h^1(\mb{P}(1, 2), \pi_*\mc{O}) = h^1(E, \mc{O}) = 1$, we deduce that $d = -3$ or $-4$. If $d = -4$, then the $\mu_2$-stabiliser of the stacky point of $\mb{P}(1, 2)$ would act trivially on the fibre of $\pi$, which contradicts the fact that $E$ is a scheme. So $d = -3$. We deduce that
\[ I = \ker(\sym^2(\mc{O}(6 - l) \oplus \mc{O}(3 - l)) \to \mc{O}(12 - 2l) \oplus \mc{O}(9 - 2l)) \cong \mc{O}(6 - 2l).\]
Since $l \geq 3$, $2l - 6 \geq 0$, so $H^1(\mb{P}(1, 2), I^\vee) = 0$, and $h^1(\mb{P}(1, 2), I^\vee) = l - 2$. Since $Z(L)_{rig} = \mb{G}_m$ acts on $L$ with weight $1$ by Lemma \ref{lem:singularityweights} and hence on $I^\vee$ with weight $2$, we deduce from Lemma \ref{lem:contractiondeformations} and Proposition \ref{prop:weighteddeformations2} that the miniversal weighted deformation space $\spec R$ is an affine space with weights $1^32^{l - 2}$, which are the same weights has $\hat{Y}\sslash W$ from Table \ref{tab:weights}.

The proof in type $D$ is similar: comparing Euler characteristics, we see that the rank $2$ bundles $\pi_*L^{-1}$ and $\pi_*L^{-2}$ on $\mb{P}^1$ have degrees $6 - l$ and $14 - 2l$. It follows that the kernel of the surjection $\sym^2 \pi_*L^{-1} \to \pi_* L^{-2}$ is $I = \mc{O}(4 - l)$. Since $l \geq 4$, we have $H^1(\mb{P}^1, I^\vee) = 0$ and $h^0(\mb{P}^1, I^\vee) = l - 3$. So the weighted miniversal deformation space $\spec R$ is an affine space with weights $1^42^{l - 3}$, which again agree with the weights of $\hat{Y}\sslash W$ from Table \ref{tab:weights}.

For the remaining types, we note that in type $A_1$ (resp., $C_l$, $F_l$, $G_2$), the unstable variety $\chi_Z^{-1}(0)$ is equivariantly isomorphic to the unstable variety for type $E_5$ (resp., $D_{l + 4}$, $E_{l + 4}$, $E_8$), with $(\mu\mmid -) = 2$ (resp., $2$, $2$, $3$). So the miniversal $\mb{Z}_{>0}(\mu\mmid -)$-weighted deformation is just the $\mu_2$- (resp., $\mu_2$-, $\mu_2$-, $\mu_3$-)fixed part of the miniversal $\mb{Z}_{>0}$-weighted deformation. It follows from the cases proved above and inspection of Table \ref{tab:weights} that this is an affine space with the desired weights.
\end{proof}

\bibliography{bibliography.bib}

\end{document}